\documentclass[aap]{imsart}

\RequirePackage{amsthm,amsmath,amsfonts,amssymb}
\RequirePackage[numbers]{natbib}
\RequirePackage[colorlinks,citecolor=blue,urlcolor=blue]{hyperref}%% uncomment this for coloring bibliography citations and linked URLs
\RequirePackage{graphicx}%% uncomment this for including figures

\usepackage[normalem]{ulem}
\startlocaldefs
\numberwithin{equation}{section}
\theoremstyle{plain}

\newtheorem{theorem}{Theorem}[section]
\newtheorem{theorem*}{Theorem}

\newtheorem{lemma}[theorem]{Lemma}

\newtheorem{proposition}[theorem]{Proposition}
\newtheorem{corollary}[theorem]{Corollary}

\newtheorem*{question*}{Question}
\newtheorem{problem}[theorem]{Problem}
\newtheorem*{problem*}{Problem}

\newtheorem*{conjecture*}{Conjecture}

\theoremstyle{remark}
\newtheorem{remark}[theorem]{Remark}
\newtheorem{example}[theorem]{Example}

\newtheorem{definition}[theorem]{Definition}
\newcommand{\nn}{\nonumber}
\newcommand{\R}{\mathbb{R}}
\newcommand{\E}{\mathbb{E}}
\newcommand{\e}{{\rm e}^}
\newcommand{\ep}{\epsilon}

\usepackage{color}
\definecolor{purple}{rgb}{0.8,0,0.1}

\newcommand\red{}

\definecolor{darkgreen}{rgb}{0,0.4,0}
\newcommand\redf{}
\endlocaldefs

\begin{document}

\begin{frontmatter}
\title{The Stefan problem and free targets of optimal Brownian martingale transport}
%\title{A sample article title with some additional note\thanksref{T1}}
\runtitle{The Stefan problem and optimal Brownian martingale transport}
%\thankstext{T1}{A sample of additional note to the title.}

\begin{aug}
%%%%%%%%%%%%%%%%%%%%%%%%%%%%%%%%%%%%%%%%%%%%%%%
%% Only one address is permitted per author. %%
%% Only division, organization and e-mail is %%
%% included in the address.                  %%
%% Additional information can be included in %%
\author[A]{\fnms{Inwon C.}~\snm{Kim}\ead[label=e1]{ikim@math.ucla.edu}}
\and
\author[B]{\fnms{Young-Heon}~\snm{Kim}\ead[label=e2]{yhkim@math.ubc.ca}}
%\and
%\author[B]{\fnms{???}~\snm{???}\ead[label=e3]{???@???}}
%%%%%%%%%%%%%%%%%%%%%%%%%%%%%%%%%%%%%%%%%%%%%%
%% Addresses                                %%
%%%%%%%%%%%%%%%%%%%%%%%%%%%%%%%%%%%%%%%%%%%%%%
\address[A]{Department of Mathematics\\ University of California at Los Angeles\\ Los Angeles, CA, USA\printead[presep={,\ }]{e1}}

\address[B]{Department of Mathematics\\ University of British Columbia\\ Vancouver, V6T 1Z2 Canada\printead[presep={,\ }]{e2}}
\end{aug}

\begin{abstract}
We formulate and solve a free target optimal Brownian stopping problem from a given distribution while the target distribution is free and is conditioned to satisfy a given density height constraint.  
  The free target optimization problem exhibits  monotonicity, from which a remarkable universality follows, in the sense that the optimal target is independent of its Lagrangian cost type. In particular, the solutions to this optimization problem generate solutions to both unstable and stable type of the Stefan problem, where former stands for freezing of supercooled fluid $(St_1)$ and the latter for ice melting $(St_2)$. This unified approach to both types of Stefan problem is new. In particular we obtain global-time existence and weak-strong uniqueness for the ill-posed freezing problem $(St_1)$, for a given initial data and for a well-prepared class of initial domains generated from the initial data.
\end{abstract}

\begin{keyword}[class=MSC]
\kwd[Primary ]{60, 49}
%\kwd{35, 80}
\kwd[; secondary ]{35, 80}
\end{keyword}

\begin{keyword}
\kwd{Stopping Times, Optimal Transport, Free Boundary, Stefan Problem, Supercooled Stefan Problem}
%\kwd{???}
\end{keyword}

\end{frontmatter}
%%%%%%%%%%%%%%%%%%%%%%%%%%%%%%%%%%%%%%%%%%%%%%
%% Please use \tableofcontents for articles %%
%% with 50 pages and more                   %%
%%%%%%%%%%%%%%%%%%%%%%%%%%%%%%%%%%%%%%%%%%%%%%
\tableofcontents

%%%%%%%%%%%%%%%%%%%%%%%%%%%%%%%%%%%%%%%%%%%%%%
%%%% Main text entry area:

\maketitle

%\tableofcontents

\section{Introduction} \label{sec:introduction}

Given a configuration $\mu$ of agents subject to a random motion, what would happen if they try to stay under a given allowed population density height $f$,  in particular, while  trying to optimize  their collective dynamics?
In the present paper we  consider the Brownian motion $W_t$ in $\R^d$, and the collective dynamics is controlled only by a stopping time $\tau$, that is, when each random particle stops. Such dynamics is called Brownian martingale.
Precisely, we consider the following:
 \begin{problem}
 \label{prob:mainquestion}
Consider a nonnegative function $f\in L^{\infty}(\R^d)$.
  Given an absolutely continuous, compactly supported measure  $\mu$ on $\R^d$,  find a stopping time $\tau^*$ and  corresponding distribution  $W_{\tau^*} \sim \nu^*$, that solve
   \begin{align}\label{eqn:problem-upper}
\mathcal{P}_f(\mu) :=  \inf_{(\tau, \nu)} \{ \mathcal{C}(\tau) \  |\  W_0 \sim \mu, \ W_\tau \sim  \nu, \,   \nu \le f, \ \& \ \hbox{ $\nu$ is compactly supported} 
\}.
\end{align}
 \end{problem}
 \noindent Notice that here we do not assume that $\mu$, $\nu$ are probability measures.
 So the expression `$W_0 \sim \mu$ and $ W_\tau \sim \nu$' should be understood as for each measurable set $E$, 
\begin{align*}
 \nu[E]= \int Prob[W_\tau \in E \ | \ W_0 =x] d\mu(x).
\end{align*}
The problem \eqref{eqn:problem-upper} obviously requires some condition for $f$ because if $f$ is too small there will be no admissible $\nu$. We consider those cases where there is at least one admissible $\nu$. Our model case is $f\equiv 1$ on $\R^d$, however a much more general $f$ can be considered.
\noindent 
We are particularly interested in the cost $\mathcal{C}(\tau)$ of the `Lagrangian form':
\begin{align}\label{eqn:Lag-cost}
 \mathcal{C}(\tau) =  \mathbb{E}\left[ \int_0^\tau L(W_t, t) dt \right] \hbox{ for a continuous $L: \R^d \times \R_{\ge 0}\to \R_{\ge 0}$}.
\end{align}

\medskip

The problem \eqref{eqn:problem-upper}  is in the spirit of what is studied in ``constrained transport", in particular, for analyzing congested crowd motion (see e.g. \cite{Santambrogio, santam-BV, AKY14})
 and tumor growth (see e.g. \cite{PQV14, JKT21}) which is subject to contact inhibition.   
 {\red 
In those cases, the optimal dynamics is deterministic, and it amounts to control a vector field (drift force) to guide the dynamics; additional noise to their dynamics  would generate stochastic aspects.  In our problem, which is a  Brownian martingale version of the free target  optimal transport problem with density upper bound constraint \cite{santam-BV}, 
 the dynamics is completely random as it follows the Brownian motion and is controlled only by the choice of stopping rule.}

\medskip

Throughout this paper we make assumptions on the cost $\mathcal{C}$ \eqref{eqn:Lag-cost}:
\begin{align}\label{eqn:boundedness-assumption}
 \hbox{ $0\le L\le D$ for some constant $D$ and, }
\end{align}
\begin{align}\label{eqn:monotone-assumption}
 \hbox{$t \mapsto L(\cdot, t)$ is either strictly increasing (Type (I)), or strictly decreasing (Type (II)). }
\end{align}

\medskip

We obtain existence, uniqueness, and characterization of the optimal solution of Problem~\ref{prob:mainquestion}; see Theorem~\ref{thm:existence}.  Remarkably  the problem exhibits universality (see Theorem~\ref{thm:universality}), namely,  the resulting optimal target $\nu^*$ does not depend on a particular choice of the function $L$ and its type,  as long as $L$ satisfies \eqref{eqn:monotone-assumption}. {\redf In fact, it turns out that $\nu^*$ coincides with the shadow of $\mu$ in $f$ as introduced in  
\cite{Rost71, beiglboeck-juillet2016} (c.f. \cite{bruckerhoff2021shadows}); see Remark~\ref{rmk:shadow}.}
These are due to a monotonicity property of the problem regarding the  optimal stopping $\tau^*$ and target $\nu^*$ (Theorem~\ref{thm:monotonicity}), which also gives $L^1$ contraction (Theorem~\ref{thm:L1-contraction}) and $BV$ estimates for $\nu^*$ (Theorem~\ref{thm:BV}). Such monotonicity is a novel feature, resulting from not fixing the target distribution $\nu$ but treating it as the solution of the problem. Another important feature of Problem~\ref{prob:mainquestion}  is that the optimal solution $\nu^*$, also  the flow of the distribution $\mu_t$ of $W_t$ toward it, 
saturates the constraint $f$ (Theorem~\ref{thm:saturation}) {\red  wherever   active Brownian particles pass by before they stop}. % possible. 	

\medskip

The free target feature of  Problem~\ref{prob:mainquestion} distinguishes it from the more well studied, so-called optimal Skorokhod problem (see e.g. \cite{beiglboeck2017optimal, PDE19, GKL20, GKP3, GKP4} as well as  the 1 dimensional results considered by   \cite{mcconnell1991two, cox2013root, gassiat2015root, deangelis2018})  
where optimal stopping times are studied for given {\em fixed}  initial and target measures $\mu, \nu$  and {\em without} the density constraint $f$, 
 that is, 
 \begin{align}\label{eqn:problem-OSP}
\mathcal{P}(\mu, \nu) :=  \inf_{\tau} \{ \mathcal{C}(\tau) \  |\  W_0 \sim \mu, \ W_\tau \sim  \nu 
\}.
\end{align} 
Solutions to Problem \eqref{eqn:problem-OSP} solve the Skorokhod problem  \cite{skorokhod1965studies, Root, baxter-chacon, Rost, azema1979solution, falkner1980skorohod, perkins1986cereteli, obloj2004skorokhod} for which 
a connection to optimal transport was hinted in the work of 
 \cite{hobson2011skorokhod},
 and established in 
\cite{beiglboeck2017optimal}
who used  optimal transport theory (see e.g. the books \cite{V1, V2, santambrogio2015optimal}) and randomized stopping times to unify all the previously known solutions of the Skorokhod problem  as the solutions of \eqref{eqn:problem-OSP}. We also mention that optimal Skorokhod problem is a special case of the optimal martingale problem (see e.g.  \cite{strassen1965existence, beiglbock2013model, Dolinsky2014, beiglboeck-juillet2016, guo2016monotonicity, GKL2019, HuesmannMartingaleBB, GKL20} and references therein) which has been recently popularized in mathematical finance. 

\medskip

As a main application we combine the above mentioned results for   Problem~\ref{prob:mainquestion} with the PDE methods  developed in \cite{PDE19} for 
\eqref{eqn:problem-OSP} to make progress in the supercooled Stefan problem, a free boundary problem of the heat equation, which has been poorly understood in the literature.

\medskip

  The supercooled Stefan problem describes freezing of supercooled water into ice, and  can be written in divergence form  in $\R^d \times \R_{\ge 0}$ as
$$ (\eta-\chi_{\{\eta>0\}})_t - \Delta \eta = 0. \leqno (St_1)
$$
 Here $-\eta$ denotes the temperature of the supercooled water (water with below-freezing temperature), and  the set $\{ \eta > 0\}$ describes the region for the supercooled water. The solution of $(St_1)$ is famously known to exhibit irregular, fractal-like interfaces that could jump discontinuously over time. This is in sharp contrast to the Stefan problem for melting ice into water,
$$
 (\eta+\chi_{\{\eta>0\}})_t - \Delta \eta = 0, \leqno (St_2)
$$ 
which can be viewed as a singular nonlinear parabolic equation of the form $[\beta(\eta)]_t - \Delta \eta = 0$ with an increasing function $\beta$; see 
\cite{AL83} where one proves comparison and contraction properties of weak solutions based on this property.

\medskip

While $(St_2)$ is known to generate stable and regularizing solutions (\cite{AL83, ACS96, CK09}), solutions of $(St_1)$ are shown to develop discontinuity and non-uniqueness in finite time, even in one space dimension and with smooth initial data (\cite{She70, CK12}). For higher dimensions, the global existence of weak solutions for $(St_1)$ has stayed largely open. 

\medskip

It is natural to consider these problems 
as generated by a particle system, which diffuses but stops its motion when it hits the interface between water and ice. In one space dimension there has been several works
 which construct solutions of $(St_1)$ as a continuum limit of such particle systems, including \cite{CS96, CK08, LM15, DIRT15, DNS19}.  In higher dimensions the  global-time existence of solutions for $(St_1)$ has  remained open, due to the difficulty of  obtaining a stable deterministic limit \cite{NSZ21}. Such global-time existence, as well as a "weak-strong'' uniqueness" theorem is what we establish in this paper. Similar to aforementioned works, we also consider solutions based on Brownian particles with hitting times, but with optimization. Indeed a novelty of our approach lies in the optimization structure given by Problem \eqref{eqn:problem-upper}, for the choice of stopping time $\tau^*$ and the final target distribution $\nu^*$. 
   The monotonicity types \eqref{eqn:monotone-assumption} of the cost, (I), (II), correspond to $(St_1)$,  $(St_2)$, respectively. The optimality gives a certain `stability'' to  $(St1)$ when its initial domain is given by the solution of Problem~\eqref{eqn:problem-upper}.  
Roughly speaking, we will show the following:

\medskip

{\it For  any compactly supported $\eta_0\in L^{\infty}(\R^d) \cap L^1(\R^d)$, there is a class of  domains $E$ that include the support of $\eta_0$ such that, for each choice of $E$, there exists a unique global-time solution  $\eta$ of $(St_1)$ with initial data $\eta_0$ and initial domain $E:= \limsup_{t\to 0^+}\{\eta(\cdot,t)>0\}$.}

\medskip

The initial domains $E$ are obtained by our optimization problem \eqref{eqn:problem-upper},  giving an initial ``regularization" of the support of $\eta_0$, by means of an obstacle problem;  see Section~\ref{sec:global-Stefan}. {\red In particular,  
our initial domains $E$ are  strictly larger than the support of $\eta_0$.}  The point is that they are to be chosen appropriately, not arbitrarily given. This limitation of our approach  is natural though, given the stability of our solutions and the ill-posed nature of $(St_1)$. 
  For instance, we conjecture that the class of solutions we find in this paper do not include those with jump discontinuities observed in the aforementioned literature. In this sense our class of solutions can be understood as the "strong'' solutions of $(St_1)$ which allow uniqueness results, featuring weak-strong uniqueness. Indeed we will show that, when we do not choose the initial domain $E$ carefully, there can be infinitely many weak solutions of $(St_1)$, even in the  radially symmetric case: see Proposition~\ref{prop:non-uniquenses-St1}.    Regularity of our solutions, and potential extension of our approach to general, discontinuous solutions, remain to be exciting open questions. % to be investigated.
 
 \medskip

If we further require certain finite-time vanishing conditions on the solution, then the corresponding initial domain $E$ is uniquely determined, so is the solution. 
Below are particular examples of our solutions which are generated by specific choices of initial domain $E$: 
\begin{itemize}
\item[(a)](Theorem~\ref{Stefan:I} and Corollary~\ref{vanishing}) For compactly supported $\eta_0$ uniformly larger than $1$ in its support, there exists a unique weak solution of $(St_1)$ with initial data $\eta_0$ that vanishes in  finite time. 
\medskip
\item[(b)](Theorem~\ref{Stefan:I:2} and Corollary~\ref{unique_char_2}) For any compactly supported $\eta_0$, there exists a unique weak solution of $(St_1)$ with initial data $\eta_0$ 
that vanishes in finite time outside  the support of $\eta_0$ and never vanishes in the support of $\eta_0$.
\end{itemize}

\medskip 

Similar  results hold  for $(St_2)$ (Theorem~\ref{Stefan:II}).
These results rely on  a rigorous connection that we establish between the weak solution and the optimal Brownian stopping (see Theorems~\ref{thm:consistency} and \ref{reverse}).

\medskip

  While the cost types (I) and (II) in \eqref{eqn:monotone-assumption} generate different optimal stopping times,  and solutions to $(St_1)$ and $(St_2)$, respectively, we prove that their optimal  
 target measure $\nu^*$ is independent of the cost, or even of its monotonicity type (Theorem~\ref{thm:universality}). 
 This remarkable {\em universality} allows us to connect the two quite different problems $(St_1)$ and $(St_2)$ together; {\red in particular, the solutions to the problem $(St_2)$,  give us  the initial trace of the solutions of $(St_1)$  given in (a) and (b) above:  we refer to Section~\ref{sec:global-Stefan} for further discussions.}

\medskip

Understanding further regularity of the free boundary for the corresponding supercooled Stefan problem $(St_1)$, is a wide open problem. In this direction, we prove a strict monotonicity result for the barrier function $s$, a version of comparison principle, which may shed a light on the challenging question;  see Theorems~\ref{thm:s_mono} and \ref{thm:stric-s-mono}.  This  is a novel feature of Problem~\ref{prob:mainquestion} as it is a consequence of the optimization of the free target distribution.  

\medskip

As a byproduct of our approach we define the notion of {\em subharmonically generated sets} in Section~\ref{sec:subharmonic-generating}, which connects the Brownian stopping time and existence of the solution to the supercooled Stefan problem $(St_1)$. It is a pair $(\Sigma, E)$ of sets determined by a certain type of Brownian stopping time from an initial measure $\mu$; Definition~\ref{def:subharmonic-set}. The solvability of $(St_1)$ with the initial data $(\mu, E)$ is determined by whether there is $\Sigma$ such that $(\Sigma, E)$ is subharmonically generated from $\mu$; Theorem~\ref{thm:sh-gen-equiv-St1}. This notion is used to describe unique solvability of  $(St_1)$ and also to characterize vanishing in finite time solution of it in Section~\ref{sec:global-Stefan}; Theorem~\ref{thm:sh-generate} and Corollaries~\ref{vanishing} and \ref{unique_char_2}.

\subsection{The connection between the Stefan problem, Problem \eqref{eqn:problem-OSP}, and Problem~\ref{prob:mainquestion}} 

The present paper relies on the previous works on the optimal Skorokhod problem \eqref{eqn:problem-OSP}.
Notice that after getting the optimal target $\nu^*$ in Problem~\ref{prob:mainquestion}, finding the optimal stopping time $\tau^*$ is reduced to solving \eqref{eqn:problem-OSP} with fixed $\mu$ and $\nu^*$. 
  It is  shown  {\red first \cite{beiglboeck2017optimal}  then \cite{PDE19}} that the solution to \eqref{eqn:problem-OSP} is the first hitting time of $(W_t, t)$  to a space-time barrier $R^*$, which enjoys time-monotonicity. Such a barrier set can be defined by a barrier function $s^*$, such that $R^*:=\{ (x, t) \ | \ s^* (t) \le t\}$ for type (I) and  $R^*:=\{ (x, t) \ | \ s^* (t) \ge t\}$ for type (II). Our results above imply  that the free boundary of the Stefan problem, at time $t$, corresponds to the boundary of the time $t$-slice of $R^*$, that is, $\{ x \ | \ (x, t) \in R^*\}$. In fact, to derive our well-posedness of the Stefan problem $(St_1), (St_2)$, it is crucial for us to verify that $R^*$ is a closed set, so its complement $(R^*)^c$ is open.  This  is handled %can  be handled 
    by Theorem~\ref{thm:tau-equal-tau-U}, which connects the barrier set of $\tau^*$ to a closed set generated by the potential  function $U_{\mu_t}$ of the mass flow of $W_t$ under $\tau^*$. Interestingly its proof essentially uses the Eulerian method in \cite{PDE19}. We remark that in the type (I) case, Theorem~\ref{thm:tau-equal-tau-U} can be derived even in a sharper form, based on \cite{gassiat2021}; they showed  that the hitting time to a type (I) barrier set coincides with the stopping time generated by the corresponding potential flow, for a general class of stochastic processes \cite{gassiat2021}.

\medskip

The PDE methods of \cite{PDE19} is based on  the correspondence between stopping time $\tau$ and its Eulerian flow $(\eta, \rho)$ given on the space-time, that describes the distribution $\eta$ of $(W_t,t)$ before stopping by $\tau$, and the stopped distribution of mass $\rho$ of $(W_\tau, \tau)$. This $(\eta, \rho)$ solves weakly
$$
\eta_t - \frac{1}{2}\Delta \eta = -\rho.
$$
Moreover, under the assumptions \eqref{eqn:boundedness-assumption}-\eqref{eqn:monotone-assumption} the $(\eta, \rho)$ resulting from a solution of $\mathcal{P}(\mu, \nu)$ \eqref{eqn:problem-OSP} yields that $\rho$ is formally supported on the boundary of the active region $\{ \eta >0\}$ \cite{PDE19}.   Assuming sufficient regularity, $\rho$ is concentrated  %should be concentrated 
 on the boundary of the barrier set, which is also the boundary of the active region $\{\eta>0\}$. Thus $\rho = \nu(\chi_{\{\eta>0\}})_t$ or $\rho=  - \nu(\chi_{\{\eta>0\}})_t$ depending on the monotonicity of the set  $\{\eta>0\}$ over time, and we end up with 
the weighted Stefan problem, 
\begin{equation}\label{Stefan-weighted}
(\eta \pm \nu\chi_{\{\eta>0\}})_t - \frac{1}{2}\Delta\eta = 0, 
\end{equation} 
where $\nu$ is the  target measure for the optimal stopping time.
This equation has been formally considered for a given fixed target measure $\nu$ in \cite{PDE19}, and our Theorem~\ref{thm:consistency} and \ref{reverse} provide a rigorous justification.

\medskip

For the connection between Problem~\ref{prob:mainquestion} and the original Stefan problem $(St_1), (St_2)$, it is crucial to have 
the saturation result (Theorem~\ref{thm:saturation}) that 
gives $\nu=1$ (when $f\equiv 1$) in the active region; it reduces the above equation \eqref{Stefan-weighted} to $(St_1), (St_2)$ for type (I), (II), respectively.

 \subsection{Remarks on future work}

{\red  It should be mentioned that most of the methods in this paper carry over to diffusions with  uniformly elliptic generators, thus more general type Stefan problems; Brownian motion corresponds to the Laplacian, and more generally uniformly elliptic operators generate continuous Markov processes.   For instance our results can be extended to Riemannian manifolds and Laplace-Beltrami operators.
More interestingly, our method is robust enough that in an ongoing work in progress \cite{CKKN23} we treat the non-local Stefan problem, extending the Laplacian to the fractional Laplacian, where even the more stable problem $(St_2)$ is not well understood.} 

\medskip

 The main results in this paper appear rather surprising, given the lack of understanding in the literature on multi-dimensional solutions of $(St_1)$.  
In a subsequent work, we will further develop our method  to extend well-posedness theory for $(St_1)$ to more general initial boundaries. Also, the connection (Remark~\ref{rmk:universality-PDE})  between $(St_1)$ and $(St_2)$ that we find in this paper through the universality  result (Theorem~\ref{thm:universality}) seems to be unexpected, and understanding this connection at a more physical level, is  an interesting open problem. 
It would also be interesting to understand the physics behind why type (I) corresponds  to the supercooled unstable Stefan and type (II) to the usual stable Stefan problem.

\medskip

In this paper our method for finding the optimal target $\nu$ is solely based on optimizing with probability measures. Here we do not consider 
	duality which is a powerful tool for understanding optimal Brownian stopping, though we heavily rely on the results from \cite{PDE19} that are obtained via solving the dual problem of the fixed target case, to  find the barrier set of the optimal stoping time. 
	Our free target problem gives a new aspect to the duality that we will discuss in a subsequent work.

\section{Preliminaries}\label{sec:prelim}
{\red Throughout the paper, measurability means Borel measurability.
%Throughout
In} this section we assume for simplicity, that the measures $\mu$ and $\nu$ are compactly supported in $\R^d$. Much of the discussion (and notation) in this section is borrowed from \cite{PDE19} which gives foundational results for the development given in this paper. 
Note that the purpose of this section is to collect the precise definitions and results necessary for the discussion in the subsequent sections. We try to include the least amount of information. 

\subsection{Notation}
\begin{itemize}
 \item $W_t$ denotes the Brownian motion, while $W^y_t$ denotes the Brownian motion with $W_0=y$. 
 \item 
$B_r (x)$ denotes the ball of radius $r$ centered at $x$. 
\item $a \wedge b:= \min [a, b]$. 
\item For each Borel measurable set $S \subset \R^d$, $|S|$ denotes its $d$-dimensional Lebesgue measure. 
\item For each (nonnegative) Borel measure $\mu$ on $\R^d$, $|\mu|$ denotes $\mu(\R^d)$.
\item LSC = the set of lower semi-continuous functions on $\R^d$.
\item $H^1_0 =H_0^1(\R^d)$  
{\red 
\item $\Omega =C(\R_{\ge 0};\R^d)$ the space of continuous curves.
\item $\mathbb{P}^\mu=$ the Wiener measure on $\Omega$, with initial distribution $\mu$; this means  for our Brownian motion $W_t$, we have  $W_0 \sim \mu$. 
\item$ \{\mathcal{F}_t\}_{t\in \R_{\ge 0}=}$ the natural filtration of the Brownian motion.
\item $(\Omega, \mathcal{F}, \{\mathcal{F}_t\}_{t\in \R_{\ge 0}},\mathbb{P}^\mu)=$ the filtered probability space over $\Omega$, with the filtered $\sigma$-algebra $\mathcal{F}$ and the probability measure $\mathbb{P}^\mu$,
\item $\mathcal{M}(\R_{\ge 0})=$ the space of Radon measures on $\R_{\ge 0}$. 
}

\end{itemize}

\subsection{Randomized stopping times and optimal Skorokhod problem in  $\R^d$}
Related to Problem~\ref{prob:mainquestion}, it is important to understand the notion of stopping times and randomized stopping times. 

\subsubsection{Stopping times and randomized stopping times}
Consider the filtered probability space $(\Omega, \mathcal{F}, \{\mathcal{F}_t\}_{t\in \R_{\ge 0}},\mathbb{P}^\mu)$.
A {\em randomized stopping time} $\alpha$ \cite{Baxter-Chacon-randomized, Meyer-randomized} is a measure-valued random variable
$\alpha: \Omega \ni \gamma \mapsto \alpha_\gamma \in  \mathcal{M}(\R_{\ge 0})$
such that as a measure on $\R_{\ge 0}$,  
\begin{align*}
 \hbox{$\alpha_\gamma \geq 0, \,  \alpha_\gamma (\R_{\ge 0})=1$, and  the map $\gamma \mapsto \alpha_\gamma ([0,t])$  is   $\mathcal{F}_t$-measurable $ \forall\, t $.}
\end{align*}
{\red One can view $\alpha$ as a certain type of probability measure over $\Omega \times \R_{\ge 0}$ whose marginal on $\Omega$ is $\mathbb{P}^\mu$. }
  Along each Brownian path,  mass is dropped according to the distribution on $\R_{\ge 0}$ determined by the randomized stopping time.
A {\em stopping time} $\tau: \Omega \to \R_{\ge 0}$ is a random variable such that for each $t$, $\tau\wedge t$ is  $\mathcal{F}_t$-measurable; it  can be understood as  a special case of randomized stopping time, where the measure $\alpha_\gamma$ on each path $\gamma$ is a Dirac mass at the value $\tau (\gamma)$ in $\R_{\ge 0}$.   To distinguish the difference between randomized stopping time and stopping time, we often call the latter `nonrandomized'.
 {\red However, we often abuse notation and use $\tau$ for either stopping times or randomized stopping times, whenever its meaning is clear from the context; in the latter case we will also use the notation $\alpha$ for the actual measure corresponding to the randomized stopping time $\tau$. In fact, one may identify a randomized stopping time given as the measure $\alpha$ with a non-randomized stopping time, by extending the probability space $\Omega$ to $\Omega \times [0,1]$, augmented by the unit interval $[0,1]$ and its Lebesgue measure $Leb_{[0,1]}$. Here $\alpha$ induces a nonrandomized stopping time 
$\rho$, given as 
\begin{align}\label{eqn:nonrandomized}
 \rho (\gamma, u) = \inf\{ t \ge 0 \ | \ \alpha_\gamma ([0,t] \ge u\} \hbox{ for each $(\gamma, u) \in \Omega \times [0,1]$};
\end{align}
 see \cite[Section 3.2]{beiglboeck2017optimal}. This point of view allows us to treat  randomized stopping times as  usual (non-randomized) stopping times, and it enables us to treat expressions about probability of randomized stopping times as those about nonrandomized stopping times. 
 }

\medskip

We will say that a subset {\red  $Q\subset \Omega\times \R_{\ge 0}$ occurs almost surely} for a randomized stopping time $\alpha$ if
	$
		\mathbb{E}\left[\int_{\R_{\ge 0}} {\red  \mathbf{1}\{(\gamma, t)\not\in Q\}} d\alpha_\gamma(t) \right]=0;
	$
	{\red here the expectation $\mathbb{E}$ is with respect to the probability measure $\mathbb{P}^\mu$, that is, $\mathbb{E}=\mathbb{E}^{\mathbb{P}^\mu}$, and $\mathbf{1}$ is the indicator function of the set.}  
	{\red From the correspondence $\alpha \mapsto \rho$ as above, this is equivalent to 
	$\mathbb{E}\left[ \mathbf{1}\{ (\gamma, \rho(\gamma, u)) \not \in Q\} \right]  =0 $ where the expected value is with respect to $\mathbb{P}^\mu \otimes Leb_{[0,1]}$. \\
	For a (randomized) stopping time $\tau$ and a fixed time $t$, we will often use the expression 
\begin{align*}
\hbox{$t \le \tau$ almost surely (a.s.), } 
\end{align*}
	 which means that for the distribution $\alpha$ of $\tau$, % for each path $\gamma \in \Omega$, and 
\begin{align*}
\mathbb{E}\left[\int_{\R_{\ge 0}} {\red  \mathbf{1}\{(\gamma, s) \ | \ t \not\le s\}} d\alpha_\gamma(s) \right]=0,
\end{align*}
equivalently,
\begin{align*}
Prob\left[ t > \rho(\gamma, u)  \right]=0.
\end{align*}
}
We also say $\nu$ is the distribution of the  stopped Brownian motion at the (randomized) stopping time $\tau$, that is, $W_\tau \sim \nu$, if 
\begin{align*}
 	\mathbb{E} \big[g(W_\tau)\big] = \int g(z)d\nu(z),\ \forall \hbox{ continuous $g: \R^d \to \R$.}
\end{align*}
{\red Here, when $\tau$ is a randomized stopping time with the  distribution $\alpha$, the expected value 
$\mathbb{E} \big[g(W_\tau)\big] $ is defined as
\begin{align*}
 \mathbb{E} \big[g(W_\tau)\big]  := \E\left[\int_{\R_{\ge 0}} g (\gamma(s)) d\alpha_\gamma(s)\right], 
\end{align*}
equivalently,
\begin{align*}
  \mathbb{E} \big[g(W_\tau)\big]  :=  \mathbb{E} \big[g(\gamma(\rho(\gamma, u))\big] =\int_{\Omega \times [0,1]} g(\gamma(\rho(\gamma, u)) d\mathbb{P}^\mu\otimes Leb_{[0,1]}.
\end{align*}
}

\subsubsection{Subharmonic order and optimal Skorokhod problem}
Subharmonic order relates two measures {\red from the perspective of the} Brownian motion. 
First, $f$ is  a subharmonic function on an open set $O$ 
in $\R^d$ if it is upper-semicontinuous with values on $\R \cup \{-\infty\}$, and for every $x \in O$ and every closed ball $B$ in $O$ with center at $x$, it satisfies 
 $$f(x) \le \frac{1}{|B|}\int_B f(y)dy.$$
When $f$ is $C^2$, the latter condition is equivalent to $\Delta f \ge 0$ on $O$, where $\Delta=\sum_{i=1}^d \frac{\partial^2}{\partial x_i^2}$ is the Laplacian. The reason why we specify the domain $O$ is because subharmonic functions in $O$ may not necessarily be extended to the whole set $\R^d$.

\medskip

Two measures $\mu, \nu$ on $\R^d$ with $|\mu|=|\nu|$ is said to be in {\em subharmonic order} $\mu\le_{SH} \nu$, if for each open set $O$ containing $\mathop{supp} (\mu+\nu)$, 
\begin{align*}
\int \varphi (x) d\mu(x) \le \int \varphi (x) d\nu(x) \quad \hbox{ for every smooth subharmonic function on $O$.}
\end{align*}
See e.g. \cite[Definition 1.2]{GKL20}.
It is known (see e.g. \cite[Theorem 1.5]{GKL20}) that for compactly supported $\mu$ and $\nu$ as we assume throughout the paper,  we have  $\mu\le_{SH}\nu$ if and only if there exists a randomized stopping time $\tau$ for the Brownian motion such that  $W_0 \sim \mu$ and $W_\tau \sim \nu$ with $\mathbb{E}[\tau]<\infty$.  
For such $\mu$ and $\nu$, for any lower semicontinuous cost functional $\mathcal{C}$ for {\red randomized} stopping times,  one can find an optimal randomized time between $\mu$ and $\nu$ using the compactness of the space of randomized stopping times; this has been  {\red established} by Beigleb\"ock, Cox, and Huesmann \cite{beiglboeck2017optimal}. In fact, it is a nontrivial problem to see that such optimal randomized stopping time is indeed a stopping time (non-randomized) to solve the problem $\mathcal{P}(\mu, \nu)$ \eqref{eqn:problem-OSP}.  This has been proved under the assumptions on the cost   \eqref{eqn:boundedness-assumption} and \eqref{eqn:monotone-assumption} {\red  \cite{beiglboeck2017optimal, PDE19}; especially see \cite[Theorem 4.7]{PDE19}.}
(For other types of cost, for example the one based on the distance, $\mathcal{C}(\tau) =\mathbb{E}[|W_0 - W_\tau|]$, similar results have been proved in \cite{GKP3}; see also \cite{GKP4}.)
\begin{theorem}[{\rm Existence/Uniqueness of Optimal Skorokhod Problem \cite{beiglboeck2017optimal} \cite[Theorem 4.7]{PDE19}}]\label{thm:optimal-Skorokhod-existence}
Given compactly supported measures $\mu, \nu$ on $\R^d$ with $\mu\le_{SH}\nu$ and $\mu, \nu \ll Leb$, the optimal Skorokhod problem $\mathcal{P}(\mu, \nu)$ \eqref{eqn:problem-OSP}
with  the cost $\mathcal{C}$ in \eqref{eqn:Lag-cost}
 under the assumption   \eqref{eqn:boundedness-assumption} and \eqref{eqn:monotone-assumption}, 
 has a unique optimal stopping time $\tau^*$, which is not randomized  in the type (I) case,  and in the type (II) case, randomized only at $t=0$; in the latter case  $W_0=x$ and $\tau^*=0$ occurs with probability $\nu(x)/\mu(x)$. 
Moreover, there exists a space-time barrier set $R^* \in \R^d \times \R_{\ge 0}$ such that $\tau^*$ is given by the hitting time to $R^*$, that is, 
\begin{align*}
 \tau^* = \inf\{ t >0\ | \ (W_t, t) \in R^*\} \quad \hbox{(in type (II) case, it holds for those paths with  $\tau^*>0$)}.
\end{align*}
\end{theorem}

For our purpose it is important to understand the barrier set $R^*$. 
It can be characterized by using the dual solutions of $\mathcal{P}(\mu, \nu)$.
 In \cite{PDE19}) they used dynamic programming principle to establish the duality:
 ${\mathcal P}(\mu,\nu)= 	{\mathcal D}(\mu,\nu) $ where 
$${\mathcal D}(\mu,\nu) := \sup_{\psi \in LSC}\Big\{\int_{\R^d}\psi(z)d\nu(z) - \int_{\R^d} J_\psi (x,0)d\mu(x)\Big\}.$$
Here, the function $(x, t)\mapsto J_\psi (x, t)$ is called the `value function' which is the result of dynamic programming: 
	\begin{align*} J_\psi(y, t) := \sup_{\sigma}\Big\{\mathbb{E}\Big[\psi(W^y_\sigma)- \int_0^\sigma L(t+ s,W^y_s)ds\Big]\Big\}
	\end{align*}
	where $\sigma$ are randomized stopping times. 
	{\red 
	The strict monotonicity assumption on $L$ \eqref{eqn:monotone-assumption} for type (I) and (II) implies (forward or backward) time-monotonicity of the function $J_\psi$. 
	}
Also, it is shown \cite{PDE19} that under the assumptions in Theorem~\ref{thm:optimal-Skorokhod-existence} (in fact for more general $\mu$ and $\nu$), there exists an optimal dual function 
$ \psi^* \in LSC\cap H^1_0$.  {\red Then the barrier set, the barrier function $s^*$, and the optimal stopping time $\tau^*$ of \eqref{eqn:problem-OSP}  are given by 
\begin{align}\label{eqn:prelim-optimal-barrier}
 & R^* = \{ (x,t)  \ | \ t \ge s^*(x) \} \hbox{ for type (I)}, \quad R^* = \{ (x,t)  \ | \ t \le s^*(x) \} \hbox{ for type (II)};\\\nonumber
  s^*(x) & = \inf\{t\in \R_{\ge 0}\  | \ J_{\psi^*}(x, t)=\psi^*(x)  \} \ \&  \ \tau^* = \inf \{ t>0 \ | \ t \ge s^*(W_t)  \} \quad   \text{ for type (I)}, \\\nonumber
  s^*(x) &= \sup\{t\in \R_{\ge 0} \ | \ J_{\psi^*}(x, t)=\psi^*(x)  \} \  \&  \  \tau^* = \inf \{ t >0\ | \ t \le s^*(W_t)  \} \quad    \text{ for type (II)}. 
\end{align}}
\subsection{Eulerian formulation}

It is important for us to relate the Brownian motion with stopping time $\tau$, with its corresponding mass flow, which we call Eulerian flow; considering such Eulerian formulation and effectively using it to analyze optimal stopping times is one of the main innovations of \cite{PDE19}. 
An Eulerian flow  is a pair of measures $(\eta, \rho)$ on $\R^d \times \R_{\ge 0}$ such that in the weak sense, 
   \begin{align}\label{eqn:EulerianPDE}
	& \rho(x,t)+\partial_t \eta(x,t)  
	=\frac{1}{2}\Delta \eta(x,t), \\ \nonumber
	& \int_{\R_{\ge 0}}d\rho = \nu, \qquad 
 \eta(x,0) = \mu (x).
\end{align}
For a more precise description, let $O\subset \R^d$ be a bounded open convex set that contains the supports of $\mu$ and $\nu$. Pick $\gamma<\lambda$ the Poincar\'{e} constant of $O$. 
We consider  		\begin{align*}
 			C_{-\gamma}(\overline{O} \times \R_{\ge 0} ):=\{ w \in C( \overline{O}\times \R_{\ge 0}) \ | \  \e{-\gamma t} w(x,t) \rightarrow 0  \hbox{ as $t\to \infty$,   uniformly in $x$} \}.
		\end{align*}
Its  dual space $\mathcal{M}_\gamma( \overline{O}\times \R_{\ge 0})$ 
	  is the finite Radon measures with $\gamma$-exponential decay {\red in time}.  
	  We let $C_{-\gamma}^{1,2}(\overline{O}\times \R_{\ge 0})$ denote the functions whose first derivative in time and second derivatives in space lie in $C_{-\gamma}( \overline{O}\times \R_{\ge 0})$.
Then, following \cite{PDE19} we shall  say that $(\eta,\rho)$ is {\em an admissible pair}, provided 
$\eta, \rho \in \mathcal{M}_\gamma(\overline{O}\times \R_{\ge 0})$ and  $\eta, \rho \geq0$, and they satisfy the following two equations: 
	\begin{align} \label{eqn:Eulerian_target} 
 			\int_{\overline{O}} u(z)\,d \nu(z) =&\ \int_{\overline{O}}\int_{\R_{\ge 0}} u(x)d{\rho}(x,t)\quad \hbox{for all $u\in C(\overline{O})$.}	
 	\end{align}
	\begin{align}\label{eqn:Skorokhod_evolution}
		-\int_{\overline{O}}w(y,0)\, d\mu(y)=&\ \int_{\overline{O}}\int_{\R_{\ge 0}}\Big[\frac{\partial}{\partial t}w(x,t)+\frac{1}{2}\Delta w(x,t)\Big]d\eta(x,t)\\
		&\ -\int_{\overline{O}}\int_{\R_{\ge 0}} w(x,t)d\rho(x, t) \quad \hbox{ for all $ w \in C_{-\gamma}^{1,2}( \overline{O}\times \R_{\ge 0})$}. \nn
	\end{align}

\medskip

Let us translate the Brownian stopping into Eulerian flow:  
\begin{proposition}[see \cite{PDE19}  Proposition 2.2]\label{thm:stochastic_embedding}
 		Given $\mu$ and $\nu$ compactly supported with $\mu, \nu \ll Leb$. 
		Suppose $\tau $ is a (randomized) stopping time with $W_0\sim \mu$, $W_\tau \sim \nu$. 
Then, 
	there is an admissible pair $(\eta,{\rho})$ of measures on $\R^d \times \R_{\ge 0}$, such that for every $g \in C_c (\R^d \times \R_{\ge 0})$ 
	$$
 			\E\big[g(W_\tau, \tau)\big]=\int_{\overline{O}}\int_{\R_{\ge0}}g(x,t){d\rho}(t,x),
 		$$
		and 
 	\begin{align}
 			\E\Big[\int_0^\tau g(W_t, t)dt\Big]   
 			=\ \int_{\overline{O}}\int_{\R_{\ge 0}}g(x,t)d\eta(x,t).\nn
 		\end{align}	
 	\end{proposition}
\noindent This gives equivalence between the optimization problem for the stopping time and the one for the Eulerian flow \cite{PDE19}:
\begin{align*}
\mathcal{P}(\mu, \nu)=\mathcal{P}_1(\mu, \nu):= \inf_{(\eta, \rho)}   \int_{\R^d}\int_{\R_{\ge 0}}L\big(x,t) d\eta(x,t)
\end{align*}
\text{where $(\eta, \rho)$'s are  admissible pairs.}

\medskip

We close this section with an essential lemma for us to consider the Eulerian flow in the context of the barrier of the stopping time. It will be later used in Section~\ref{sec:potential}. 

\begin{lemma}[\cite{PDE19} Lemma 4.5]\label{lem:Eulerian_uniqueness}	
Suppose $\mu \ll Leb$ is  a compactly supported measure and $R \subset \R^d \times \R_{\ge 0}$. 
Let $(\eta, \rho)$ be admissible 
 with the condition  $\eta(R)=0$ and $\rho(R)=1$.
   
\begin{itemize}
 \item We suppose $R$ is a measurable forward-barrier, namely, $(x,r)\in R$ whenever $(x,t)\in R$ with $t\leq r$, and $(x,t)\in R$ if there is $(x, t_i)\in R$ with $t_i\rightarrow t$. 
		Then $(\eta, \rho)$ is unique. 

\item 
		Suppose instead  that $R$ is a measurable backward-barrier, namely, $(x,r)\in R$ whenever $(x,t)\in R$ with $t\geq r$, and $(x,t)\in R$ if there is $(x,t_i)\in R$ with $t_i\rightarrow t$. Then $(\eta,\rho)$  is uniquely determined {\red by the restriction of $\rho$} on the set  $\{(x,0) \ |  \ s(x)=0\}$ {\red where $s$ is given by $s(x) = \sup \{ t \ | \   (x, t) \in R \}.$ }

\end{itemize}
			\end{lemma}

\begin{remark}
 Let us comment on why the additional condition $\eta (R) =0, \rho(R)=1$ to the corresponding PDE \eqref{eqn:Skorokhod_evolution} is sufficient to determine the unique solution. Because $\rho(R)=1$ for the probability measure $\rho$, we have $\rho(R^c)=0$. Therefore, the \eqref{eqn:Skorokhod_evolution} behaves like the heat equation for $\eta$ in $R^c$, with the Dirichlet condition which is due to $\eta (R)=0$; therefore the uniqueness is not surprising, especially if $R$ is closed (so $R^c$ is open). In fact we will show that we can take $R$ as a closed set {\red when $(\eta, \rho)$ comes from the optimal solution in Theorem~\ref{thm:optimal-Skorokhod-existence}}; see Section~\ref{sec:potential}, Theorem~\ref{thm:tau-equal-tau-U}.  
 \end{remark}

\section{Properties of  stopping times}\label{sec:hitting}
In this section we prove several basic but important results for the properties of hitting times as well as optimal stopping times.  We handle the general case where the barriers are  merely measurable sets, say, not open or closed. This may be of its own interest.
For the present paper, because of Section \ref{sec:potential} (see Thereom~\ref{thm:tau-equal-tau-U}),  it suffices to consider  only the case where the barrier set is closed, in which case some proofs of this section can possibly be simplified.  Results in this section may possibly be known to experts, however, we were unable to find a suitable reference.
{\red In this section, we focus on nonrandomized stopping times, and 
for a measurable set $G \subset \R^d$, we let $\tau_G$ dennote the first hitting time to $G$, that is, $$\tau_G:= \inf\{ t >0\  | W_t \in G \}.$$}

\subsection{Properties of stopping times with respect to a barrier set}\label{sec:properties-stopping-barrier}
We first recall the following result from potential theory:
\begin{lemma}\label{lem:harmonic_measure}
{\red  Let $\sigma$ denote the surface measure of the round sphere $\partial B_r$ of radius $r$ centered at $0$. 
 For $z \in \R^d$ consider the Brownian motion $W^z_t$ starting from $z$,  and the first hitting time $\tau_{\partial B_r}$ to the round sphere $\partial B_r$. Let $\mu_z$ be the distribution of $W_{\tau_{\partial B_r}}^z$ along $\partial B_r$, that is, $W_0 \sim \delta_z$, $W_{\tau_{\partial B_r}} \sim \mu_z$. 
}
Suppose $|z| \ge 2r$ or $|z| \le \frac12 r$. Then, there exists a universal constant $C>0$ depending only on the dimension,  such that the  Radom-Nikodym derivative satisfies
\begin{align*}
 \frac{d\mu_z}{d\sigma} \le {\red  \frac{C}{\sigma[\partial B_r]}.} 
\end{align*}
In particular, we have
$$Prob[W^z_{\tau_{\partial B_r}} \in E] \le C \frac{\sigma[E]}{\sigma[\partial B_r]} \hbox{ for  any } E \subset \partial B_r.$$
\end{lemma}
\begin{proof}
{\red Notice that $\mu_z$ is nothing but the harmonic measure along $\partial B_r$, generated by the point $z$.
Therefore,} the desired upper bound is  a standard result of potential theory.  This is a consequence of Dahlberg's Theorem \cite{D}.
\end{proof}

Using this  lemma we now prove a key lemma, which takes care of a subtle possibility that Brownian paths blocked by a barrier set may possibly drop mass in the barrier set; the lemma says it does not happen when the final distribution is absolutely continuous.

{\red Below we will consider the following assumptions on $\mu, \nu$ and $\tau$: 
\begin{equation}\label{assumption_measure}
\hbox{$\mu, \nu$ are probability measures such that } 
\mu \ll Leb \hbox{ and $\nu \in L^\infty$;} 
\end{equation} 
\begin{align}\label{eqn:assumption-tau}
 \hbox{
 $\tau$ is a (nonrandomized)
  stopping time with 
  $W_0\sim \mu$ and $W_\tau \sim \nu$.  }
\end{align}
}
\begin{lemma}\label{lem:zero_measure_G}
  {\red Let $\mu, \nu, \tau$ satisfy \eqref{assumption_measure} \eqref{eqn:assumption-tau}.}
{\red  Let $G$ be a measurable set, 
 and suppose} $\tau \le \tau_G$  and $\mu\wedge (\nu|_G) =0$. 
 Then 
  $\nu[G]=0$. 
\end{lemma}

\begin{remark}
 Notice that the condition {\red $\nu \in L^\infty$}, 
 especially $\nu \ll Leb$, is essential in Lemma~\ref{lem:zero_measure_G}. For example,  the distribution $\nu \sim W_\tau$ of the Brownian motion of the hitting time $\tau$ to a sphere $G$, does not satisfy the result of this lemma.
\end{remark}
\begin{proof}
Suppose $\nu[G] >0$ for contradiction. {\red Then, since $\nu \ll Leb$  by \eqref{assumption_measure},} there exists $\delta>0$ such that the set 
\begin{align*}
 G_\delta:= \{ z \in G \ | \  \frac{d\nu}{dx} \ge \delta \} \quad \hbox{has $|G_\delta|>0$.}
\end{align*}
{\red Use the Lebesgue differentiation theorem and find a point $x \in G_\delta$ of $G_\delta$, which is a Lebesgue point for both $\chi_{G_\delta}$ and the function $\frac{d\mu}{dx}$ (because almost every point $x$ are Lebesgue points of both functions). % where the Lebesgue density is $1$. 
  Then, from Lebesgue density property and the condition $\mu\wedge (\nu|_G) =0$, we see that }
 for each small $\epsilon >0$, there exists $\bar r>0$  (depending on $\epsilon$) such that % the following holds:
\begin{align*}
 & \nu[B_r(x)]\ge \delta/2 |B_r(x)|,  \qquad |G_\delta\cap B_{r}(x)| \ge (1-\epsilon)|B_{r}(x)|,\\
 & {\red \hbox{and } \mu[B_{r}(x)] \le \epsilon |B_{r}(x)|,} \qquad \text{ for all } 0< r \le \bar r.
\end{align*}
From the Fubini's theorem, there are  $0< r_1<  r_2 < \bar r $ with $8r_1 < r_2< 10 r_1 $ such that 
\begin{align*}
\mathcal{H}^{d-1} [G_\delta\cap \partial B_{r_i}(x)] \ge (1-5\epsilon)\mathcal{H}^{d-1} [\partial B_{r_i}(x)], \quad i=1,2,
\end{align*}
{\red for the $d-1$ dimensional Hausdorff measure $\mathcal{H}^{d-1}$.}

From Lemma \ref{lem:harmonic_measure},  for a dimensional constant $C$ we have
$$ Prob[ W_\tau \not\in \overline{B_{r_2} (x)}  \  | \ W_t \in B_{4r_1}(x)  \hbox{ for some }    t \le  \tau   ] \le C\epsilon,
$$
where the $\overline{B_{r_2} (x)}$ means the closure. 
This is because for the Brownian path from a point inside $B_{4r_1}(x)$ to go outside $B_{r_2}(x)$, it has to go through $\partial B_{r_2}(x)$ but avoiding hitting $G$, thus $G_\delta$; {\red here the condition $\tau \le \tau_G$ is crucial}. 
 {\red To give details, 
 since the ball $B_{4r_1}(x)$ is an open set and Brownian paths are continuous almost surely,  the above probability is the same as 
 $$ Prob[ W_\tau \not\in \overline{B_{r_2} (x)}  \  | \ W_t \in B_{4r_1}(x) \ \text{ for some $t\in \mathbb{Q}$ with }  t < \tau ].% \le C\epsilon.
$$
Now, give  an ordering to to $\mathbb{Q}$ as $\mathbb{Q} =\{ q_1, q_2, \cdots, \}$. Then, the event 
\begin{align*}
 \hbox{``$W_t \in B_{4r_1}(x) \ \text{ for some $t\in \mathbb{Q}$ with }  t < \tau$''}
\end{align*}
 can be decomposed into the disjoint $E_j$, $j=1,2,...$ defined as follows. Let $Q(t)$ be the statement ``$W_{t} \in B_{4r_1}{x}$ and  $t < \tau$''. Then define $E_j$'s inductively, 
\begin{align*}
 E_1:=\hbox{``$Q(q_1)$ holds.''}, \quad E_2 = \hbox{ ``$Q(q_2)$ holds but not $Q(q_1)$.''}, \quad \cdots 
\\
\cdots E_k : =  \hbox{ ``$Q(q_k)$ holds but none of $Q(q_1), ... , Q(q_{k-1})$."}, \cdots \cdots. 
\end{align*}
Then,  $Prob[ W_\tau \not\in \overline{B_{r_2} (x)}  \  |    \ E_k  ]  \le C\epsilon$ for each $k$. 
Therefore, 
\begin{align*}
 & Prob[ W_\tau \not\in \overline{B_{r_2} (x)}  \  | \ W_t \in B_{4r_1}(x) {\red \ \text{ for some }  t \in \mathbb{Q}\  \hbox{ with } \  t < \tau }  ] \\
 &= \frac{\sum_{k=1}^\infty Prob[ W_\tau \not\in \overline{B_{r_2} (x)}  \  |    \ E_k  ] Prob[E_k]}{\sum_{k=1}^\infty Prob[E_k]}\\
 & \le C\epsilon \hbox{ as desired.}
\end{align*}
 }
 From this  estimate we get 
\begin{align*}
  Prob[ W_\tau \in \overline{B_{r_2} (x)} \  | \  W_t \in B_{4r_1}(x) \ \text{ for some } t  \le \tau ] \ge 1-C \epsilon.
\end{align*}
By  a similar argument of using Lemma \ref{lem:harmonic_measure}, %the same lemma, 
we get 
\begin{align*}
 Prob[ W_\tau \in B_{ r_1}(x) \ | \   \ W_t \in B_{4r_1}(x) \setminus \overline{B_{2r_1}(x)} \ \text{ for some }  \ t \le \tau ] \le C \epsilon.
\end{align*}
Now let us consider
\begin{align*}
 \nu [B_{r_1}(x)] & = Prob[ W_\tau \in B_{r_1}(x) ]\\
 & \le Prob[   W_0 \in \overline {B_{2r_1} (x)} ] 
+ Prob[ W_\tau \in B_{r_1}(x) \ \& \ W_0 \not\in \overline{B_{2r_1} (x)} ]\\
& = I + II.
 \end{align*}
 {\red Notice that from our choice of $x$ and $\mu \ll Leb$,} 
 $$I =\mu[\overline{B_{2r_1}(x)}] = \mu[B_{2r_1}(x)] \le   \epsilon |B_{2r_1}(x)|.$$ On the other hand, from continuity of Brownian paths, for the Brownian path from a point outside $\overline{B_{2r_1}(x)}$ to arrive in $B_{ r_1}(x)$, it first needs to arrive in $B_{4r_1}(x) \setminus \overline{B_{2r_1}(x)}$. Therefore, 
 \begin{align*}
II &\le  Prob[ W_\tau \in B_{r_1}(x) \ \&  \ W_t \in B_{4r_1}(x) \setminus \overline{B_{2r_1}(x) }\ \text{ for some }  \ t \le \tau ] \\
& = Prob[ W_\tau \in B_{r_1}(x) \ | \   \ W_t \in B_{4r_1}(x) \setminus \overline{B_{2r_1}(x)} \ \text{ for some }  \ t \le \tau ]\\
& \qquad \times  Prob[\ W_t \in B_{4r_1}(x) \setminus \overline{B_{2r_1}(x) } \ \text{ for some }  \ t \le \tau ].
\end{align*}
Also, notice that 
\begin{align*}
 \nu[\overline{B_{r_2}(x)}] = Prob[W_\tau \in \overline{B_{r_2}(x)} ] & \ge Prob[W_\tau \in \overline{B_{r_2}(x)} \ | \ W_t \in B_{4r_1}(x) \text{ for some } t \le \tau]\\
 & \qquad \times Prob[W_t \in B_{4r_1}(x) \text{ for some } t \le \tau].
 \end{align*}
Combining all these with the previous probability estimates with $\epsilon$, we see that
\begin{align*}
 II \le \frac{ C \epsilon}{1-C\epsilon} \nu[ \overline{B_{r_2}(x)} ]
\end{align*}
which implies
\begin{align*}
 \nu [B_{ r_1}(x)] \  \le  \ I +II  \  \le \  \epsilon |B_{2r_1}(x)| + \frac{C\epsilon}{1-C\epsilon} \nu[\overline{ B_{r_2}(x)} ]. 
\end{align*}
Recall that  $\nu [B_{r_1}(x)] \ge \delta |B_{r_1}(x)|$, $\nu$ is bounded, and that the volumes $|B_{r_1}(x)|,$ $|B_{2r_1}(x)|,$ and  $|\overline{B_{r_2}(x)}|$ are all comparable up to a constant factor depending only on the dimension.  Then, we get a contradiction by letting $\epsilon \to 0$. This completes the proof. 
\end{proof}

There are several important consequences of Lemma~\ref{lem:zero_measure_G}. 
An immediate consequence is this.
\begin{corollary}\label{cor:zero_s_ae}
  {\red Let $\mu, \nu, \tau$ satisfy \eqref{assumption_measure} \eqref{eqn:assumption-tau}.}
 In addition, suppose that $\tau$ is  given as the first hitting time to a barrier function $s$ {\red (which is a measurable function $s: \R^d \to \R_{\ge 0} \cup \{\infty\}$)}  such that 
\begin{align*}
 \tau &= \inf\{ t >0\  | \ t \ge s(W_t) \} \hbox{ in type (I) case,}
 \\
 \tau &= \inf\{ t >0\  | \ t \le s(W_t) \} \hbox{ in type (II), respectively.}
\end{align*}
 Let  $G:= \{ x \ | \  s(x) =0\} $ in type (I) case, and $G:= \{ x \ | \  s(x) =\infty\} $ in type (II) case, respectively. If $\nu \wedge \mu =0$, then $\nu[G]=0$.
\end{corollary}
\begin{proof}
 Notice that $\tau \le \tau_G$ 
   from the definition of $\tau$ and $\tau_G$. Therefore, the result follows from Lemma~\ref{lem:zero_measure_G}.
\end{proof}

Later  in this  {\red paper (e.g. Section~\ref{sec:saturation})}  we will {\red use} another technical fact regarding  stopping times.

\begin{lemma}\label{lem:tau_le_tauG}  
Let $\tau$ be a  stopping time of the Brownian motion $W_t$. 
Suppose that $G$ is a (Borel) measurable set  in $\R^d$ {\red with $|G| >0$}, 
and every point of $G$ has Lebesgue density larger than  a constant $a>0$. {\red Suppose there is  a countable and dense subset $D$} of $(0,\infty)$ such that 
\begin{align}\label{eqn:slice_zero_prob}
   Prob[W_{\bar t} \in G \ \& \ \bar t < \tau] =0 \quad \hbox{ for each {\red fixed} } \bar t\in D.
\end{align}
 Then, {\red $\tau \le \tau_G$ 
 almost surely. }
\end{lemma}
\begin{proof}  
{\red
{\bf Step 1.} We first prove 
\begin{align*}
 &Prob[ \exists t'' \in D  \hbox{ such that $\tau_G<t'' < \tau$ and $W_{t''} \in G $} \ | \  \ \tau_G < \tau \ ] =1.
 \end{align*}
 To see this, observe 
from the definition of the hitting time $\tau_G$, 
 that for each $0<\delta$,  
\begin{align*}
 (*) \quad \cdots \cdots \qquad Prob[ \exists t'  \hbox{ such that $t\le t'< t+\delta/2 $ and $W_{t'} \in G $} \ | \  (W_{\tau_G}, \tau_G) =(x, t)] =1.
\end{align*}
(Notice that $(*)$ may not hold if we require $t <t'$ instead of $t\le t'$.)\\
Moreover, at each $y \in G$, there is $r_y>0$ such that for every $0< r <r_y$, 
\begin{align*}
|B_r(y) \cap G| \ge a/2 |B_r(y)|.
\end{align*}
For each fixed $t''> t'$, and  for given $y=W_{t'}$, the probability density $\sigma$ for $W_{t''}$  is nothing but the heat kernel  at  $y$ with time $t''-t'$, 
that is, 
\begin{align*}
 \sigma (z)  = \frac{C_d}{{(t''-t')}^{d/2}} e^{-|z-y|^2/(t''-t')}.
 \end{align*}
 Therefore, by choosing $t'' \in D$ close to $t'$ (possible because of denseness of $D$) that $t'' -t' =r^2$ for $r \ll r_y$, we see that for $y \in G$,
\begin{align*}
 Prob[ W_{t''} \in G  \ | \ W_{t'} =y ] \ge C. 
 \end{align*}
for $C=C(d,a)>0$; this is for each $r \ll r_y$. 
This means that, for each $\delta>0$, 
\begin{align*}
 Prob[ \exists t'' \in D \hbox{ such that  $t' < t'' < t' +\delta/2 $ and } W_{t''} \in G \ | \ W_{t'} =y ]  \ge C, 
\end{align*}
Then by applying the $0$-$1$ law, we get
for each $\delta>0$, 
\begin{align*}
 Prob[ \exists t'' \in D \hbox{ such that  $t'
 < t'' < t' +\delta/2 $ and } W_{t''} \in G \ | \ W_{t'} =y ]  =1.
\end{align*}
Combining this with the above conditional probability $(*)$, we get
 for each $\delta>0$,  
\begin{align*}
 Prob[ \exists t'' \in D  \hbox{ such that $t<t''< t+\delta $ and $W_{t''} \in G $} \ | \  (W_{\tau_G}, \tau_G) =(x, t)] =1.
\end{align*}
By the strong Markov property, this implies, for each $\delta>0$, 
\begin{align*}
 &Prob[ \exists t'' \in D  \hbox{ such that $t<t''< t+\delta $ and $W_{t''} \in G $} \ | \  (W_{\tau_G}, \tau_G) =(x, t) \ \& \ \ \ \tau_G < \tau \ ] =1.
 \end{align*}
Notice that 
\begin{align*}
Prob[\tau > t +\delta  \ | \   \  (W_{\tau_G}, \tau_G) =(x, t)\ \& \ \tau_G < \tau ] 
 \to  1 \hbox{ as $\delta \to 0$ for $\delta > 0$}.
\end{align*}
Therefore, 
\begin{align*}
 &Prob[ \exists t'' \in D  \hbox{ such that $t<t'' < \tau $ and $W_{t''} \in G $} \ | \  (W_{\tau_G}, \tau_G) =(x, t) \ \& \ \ \ \tau_G < \tau \ ] =1, 
 \end{align*}
which implies 
\begin{align*}
 &Prob[ \exists t'' \in D  \hbox{ such that $\tau_G<t'' < \tau $ and $W_{t''} \in G $} \ | \  \ \tau_G < \tau \ ] =1.
 \end{align*}

{\bf Step 2.}
From Step 1, we get 
\begin{align*}
 Prob[ \exists t'' \in D \ \hbox{\red such that } \  \tau_G <  t'' < \tau \ \& \ W_{t''} \in G]  =  Prob[\tau > \tau_G].
\end{align*}
 On the other hand, since $D$ is countable, we have
\begin{align*}
 Prob[ \exists t'' \in D \ \hbox{\red such that } \  \tau_G < t'' < \tau \ \& \ W_{t''} \in G] \le \sum_{\bar t \in D} Prob[ W_{\bar t} \in G \ \& \  \bar t < \tau].
\end{align*}
Therefore, if $  Prob[\tau > \tau_G]>0$, then 
 there exists $\bar t \in D$ such that 
\begin{align*}
 Prob[W_{\bar t} \in G \ \& \ \bar t < \tau] >0,
\end{align*}
contradicting \eqref{eqn:slice_zero_prob}. Therefore $  Prob[\tau > \tau_G]=0$ completing the proof.

}

\end{proof}

\begin{remark} 
 In Lemma~\ref{lem:tau_le_tauG} the assumption that $G$ consists of its Lebesgue point {\red (of positive density $>a$)} is essential. For example, let $d=1$  and $G_1:= {\red [-4, -3]\cup[3,4]}$, and consider
 $G:= \{0\} \cup G_1$. Let $\mu$ be the uniform distribution on the interval $[-1/2, 1/2]$. One can find a stopping time $\tau$ with %and its distribution $\nu$,
  $W_0\sim \mu$, $W_\tau\sim\nu$ such that $\nu$ is  supported   on the small intervals around the points $-3/2$ and $3/2$. % $[-3/2, 3/2]$.
  Then, because of the support of $\nu$, we have that $\tau \le \tau_{G_1}$. Therefore, for each $t>0$, the probability $Prob[W_t \in G \ \& \  t < \tau] =Prob[W_t =0 \ \& \  t < \tau] =0$. However, on the other hand, $\tau \not \le \tau_G$, as Brownian paths will pass through $0$ before $\tau$, with positive probability. 
\end{remark}

Lemma~\ref{lem:tau_le_tauG} together with Lemma~\ref{lem:zero_measure_G} give the following useful result. 

\begin{corollary}\label{cor:mass_below_s}
  {\red Let $\mu, \nu, \tau$ satisfy \eqref{assumption_measure} \eqref{eqn:assumption-tau},}
and let $G$ be a measurable subset of $\R^d$. 
Let $D \subset (0,\infty)$ be a countable dense subset such that \eqref{eqn:slice_zero_prob} holds.
Then either {\red $\mu \wedge (\nu |_G) >0$ } 
 or $\nu[G]=0$. 
\end{corollary}
\begin{proof}
The case $|G|=0$ is obvious since $\nu\ll Leb$. Let $\tilde G$ be the set of Lebesgue points of $G$.  Suppose $|G|>0$ and $\mu\wedge (\nu |_G)=0$. {\red Then, $|\tilde G|>0$ and  in $\mu \wedge  (\nu|_{\tilde{G}})=0$.} We will show that $\nu[G]=0$.   Since \eqref{eqn:slice_zero_prob} remains valid for $\tilde G$,  Lemma \ref{lem:tau_le_tauG} yields $\tau\le\tau_{\tilde G}$ (almost surely),  which 
 then implies  $\nu[\tilde G]=0$ by
 Lemma~\ref{lem:zero_measure_G} as $\mu \wedge (\nu|_{\tilde G}) =0$.  Since $\nu[G] =\nu[\tilde G]$ for $\nu \ll Leb$, this completes the proof. 
\end{proof}

\begin{remark}\label{rmk:1-cor:mass_below_s}
 Note that the  {\red condition $\nu \in L^\infty$,} 
 especially $\nu \ll Leb$,
  is essential in Corollary~\ref{cor:mass_below_s}. In particular, if $|G|=0$, then $Prob[W_{t} \in G \ \& \ t < \tau] =0$ for each constant $t>0$ when $W_0\sim \mu \ll Leb$. So, if a singular measure $\nu$ supported on such a zero Lebesgue measure set, for example the distribution $\nu \sim W_\tau$  for the hitting time $\tau$ to a sphere, does not satisfy the result of this Corollary.
\end{remark}

{\red 
\begin{remark}\label{rmk:2-cor:mass_below_s} This corollary will be used in Section~\ref{sec:saturation}, where we prove a saturation result of the density upper bound constraint for optimal target measures, in Theorem~\ref{thm:saturation}. However, one may bypass using this, by using the Eulerian formulation and closedness of the barrier set for the optimal stopping times in Section~\ref{sec:potential}; there the Proposition~\ref{lem:hitting_s}, which does not use this corollary or Lemma~\ref{lem:tau_le_tauG}  but Lemma~\ref{lem:zero_measure_G}, is used crucially, We still think Lemma~\ref{lem:tau_le_tauG} and Corollary~\ref{cor:mass_below_s} are of their own interest. 
\end{remark}
}

We now focus on stopping times  that are optimal solutions  in Theorem~\ref{thm:optimal-Skorokhod-existence}. Such a stopping time is given
as the hitting time to a barrier function $s^*: \R^d \to [0,\infty]$ {\red that is determined by the dual optimal solution}.  Lemma~\ref{lem:zero_measure_G} and its consequences give a useful characterization for such a stopping time,  which will play a key role later in the paper (see for instance  {\red Theorems~\ref{thm:tau-equal-tau-U}} and \ref{thm:consistency} %\textcolor{blue}{the LSC property} 
and  Section~\ref{sec:mono_barrier_function}). {\red Especially, the following proposition is crucial, which roughly speaking, endows a certain regularity to the set $\{(x, t) \ | \ t=s^*(x)\}$ from the condition  \eqref{assumption_measure}, by ensuring that the Brownian particles stop only at the set given as the graph of $s^*$, nor above or below (in $t$), % in its limit point, 
for instance.  Such a property does not hold in general, and in our case it comes from optimality that the barrier set and the corresponding stopping time are determined by solving \eqref{eqn:problem-OSP}.}
\begin{proposition}\label{lem:hitting_s}
 Let  $\mu, \nu $ satisfy  \eqref{assumption_measure}, and {\red let $\tau^*$, $R^*$, $s^*$ be as in Theorem~\ref{thm:optimal-Skorokhod-existence} and  \eqref{eqn:prelim-optimal-barrier}  with the same assumptions (on $\mu, \nu$), the optimal stopping time and the corresponding barrier set $R^*$ and the barrier function.  
 In the (II) case further assume that $\mu\wedge \nu=0$.}
 Then, 
\begin{equation}\label{stop}
\tau^*=s^*(W_{\tau^*}) \hbox{ almost surely.} 
\end{equation}
\end{proposition}
\begin{proof}
{\red First recall that, by \cite[Corollary 3.6]{PDE19}, the optimal stopping occurs only in the barrier set which is given by the optimal solution to the dual problem. More precisely, in \cite[Corollary 3.6]{PDE19},  for the dual optimal solution $\psi^*$,  it is shown that $J_{\psi^*} (W_{\tau^*}, \tau^* )= \psi^* (W_{\tau^*}) $ almost surely, equivalently, we get the following key property for us:
\begin{align*}
 \hbox{$(W_{\tau^*}, \tau^*) \in R^*$ almost surely.}
\end{align*}
} 
Thus, if \eqref{stop} is not true, then 
 \begin{align*}
 & Prob[\tau^* > s^*(W_{\tau^*})] >0 \hbox{ in the (I) case,} \\
& Prob[\tau^* < s^*(W_{\tau^*})] >0 \hbox{  in the (II) case, respectively, }
\end{align*}
 so 
 there exists $\mathbb{Q} \ni \bar t >0$ such that 
\begin{align*}
& \hbox{$Prob[\tau^*  > \bar t>   s^*(W_{\tau^*})] >0$ in the (I) case,}\\
 & \hbox{$Prob[\tau^*  < \bar t <   s^*(W_{\tau^*})] >0$ in the (II) case,   respectively.}
\end{align*}
 
Let us separate the (I) and (II) cases.

\medskip

1. (I) case.  Consider the set $S:=\{ x \ | \ s^*(x) < \bar t \}$ and the distribution $\nu_{\bar t}$ of $W_{\bar t\wedge\tau^*}$. Note that $W_{\bar t} \in S$ implies that $\tau^* \le \bar t$. Therefore, 
$W_{\bar t \wedge \tau^*} \in S$ is equivalent to ``$W_{\tau^*} \in S$ and $\tau^* \le \bar t$'', which is then equivalent to ``$s^*(W_{\tau^*}) < \bar t$ and $\tau^* \le \bar t$'', whose probability is strictly $<1$ due to  $Prob[\tau^*  > \bar t>   s^*(W_{\tau^*})] >0$. Therefore, $\nu_{\bar t} [S] <1$, thus the measure $\mu_{\bar t}:= \nu_{\bar t} - \nu_{\bar t}|_S$ has a positive total mass $|\mu_{\bar t}| >0$.\\
We then start the Brownian motion from the distribution $\mu_{\bar t}$ and initial time $
\bar t$; we immediately stop if $\tau^*\le \bar t$, otherwise we continue the Brownian motion until $\tau^*$; we call this stopping time $\bar \tau$. 
Let $\bar \nu$ denote that distribution of $W_{\bar \tau}$, then, from the condition $\nu \ll Leb$, we also have $\bar \nu \ll Leb$. Moreover,   from the definition of $S$ and the fact $Prob[\tau^*  > \bar t>   s^*(W_{\tau^*})] >0$, we have  $$\bar \nu [S]>0.$$ 
 On the other hand, let $\tau_S$ be the first hitting time to the set $S$ for the Brownian motion starting from the time $\bar t$ with the initial distribution $\mu_{\bar t}$.
Notice that if $W_t \in S$ (with $t\ge \bar t$) then $t\ge \bar t >  s^*(W_t)$ by the definition of $S$, therefore $\tau^* \le t$, thus $\bar \tau \le t$ from the construction of the stopping time $\bar \tau$. This implies that  $\bar \tau \le \tau_S$. 
 From  the construction of $\mu_{\bar t}$ we see that $\mu_{\bar t}\ll Leb$ and $\mu_{\bar t} \wedge \bar \nu|_S =0$. 
 Therefore from Lemma~\ref{lem:zero_measure_G} we see that $\bar \nu [ S] = 0$, a contradiction. This completes the proof in the (I) case.

\medskip

2. (II) case. 
Let $S=\{ x \ | \ s^*(x) > \bar t \}$. 
Let $\bar \tau= \bar t \wedge \tau^*$ and $\tau_S :=\inf \{ t \ | \ W_t \in S\}$. 
 Notice that obviously if $t \ge \bar t$, {\red then $t \ge \bar \tau$.} On the other hand, if $t < \bar t$ and $W_t \in S$ then $s^*(W_t) > t$ so {\red $t\ge \tau^*$.} These imply that $\bar \tau \le \tau_S$.  
 Moreover, let $\bar \nu$ be the distribution of $\bar \tau$, that is, $W_{\bar \tau} \sim \bar \nu$. 
Then from  $Prob[\tau^* < \bar t < s^*(W_{\tau^*})] >0 $  and the assumption $\mu \wedge \nu=0$ in this (II) case, 
there exists $S_1 \subset S$ such that $\mu [S_1] =0$ and $\bar \nu[S_1]>0$.
Since $\mu \wedge (\bar \nu|_{S_1})=0$,  and $\bar \tau \le \tau_S \le \tau_{S_1}$, 
we have from Lemma~\ref{lem:zero_measure_G}, that $\bar \nu[S_1] =0$, a contradiction. This completes the proof.
\end{proof}

\section{Optimal stopping times and potential flows}\label{sec:potential}
 
For each stopping time $\tau$ we can consider the time flow of  distribution $\mu_t$ of $W_{\tau\wedge t}$, $0\le t \le \infty$. {\red which then defines} potential functions by solving the corresponding Poisson equation. These potential functions can be used back to study the stopping times, which is especially beneficial because they have regularity properties coming from the elliptic regularity and Ito's formula. 
 Using potential functions to study stopping times has been considered in the literature starting from \cite{baxter-chacon} for existence of Brownian stopping times (Skorokhod problem), and more recently for analyzing optimal stopping problems (see \cite{cox2013root, gassiat2015root, deangelis2018}, for one space dimension, and \cite{PDE19, gassiat2021} for higher dimensions).
 
 \medskip

  The main result of this section (Theorem \ref{thm:tau-equal-tau-U})  shows that when $\tau$ 
  is {\red the} optimal stopping time in Theorem~\ref{thm:optimal-Skorokhod-existence},
  given {\red as} the hitting time to a barrier of type (I) or (II), the barrier  is  a closed set;  this  is a consequence from the correspondence between the  optimal stopping time $\tau$ and its potential flow. 
We remark that in the type (I) case, such correspondence (a sharper form than ours) is proved in \cite{gassiat2021} even for much more general class of stochastic processes.  Their proof does not use optimality while ours that covers both type (I) and (II), uses the optimality assumption in an essential way; in particular, we uses Proposition~\ref{lem:hitting_s}.
 
 \medskip
 
 Closednesss of the barrier set is a key fact that will be used in Section~\ref{sec:consistency-Stefan} where we establish the consistency of our probabilistic formulation of the Stefan problem with the PDE formulation.

\subsection{Potential flows, definition and regularity}
In this subsection we define the potential flow $U_{\mu_t}$ and establish its space-time continuity (Corollary \ref{cor:uniform-U}). 

\medskip

Let $N(y)$ be the Newtonian potential function on $\R^d$, namely
\begin{align*}
 N(y) :=  \begin{cases}
 \frac{1}{2} |y|  & \text{ $d=1$},\\
   2\pi  \log |y|   & \text{ $d=2$}, \\
 \frac{1}{d(2-d) \omega_d} |y|^{2-d}    & \text{otherwise},
\end{cases}
\end{align*}
where $\omega_d$ is the volume of the unit ball in $\R^d$  {\red so that} \begin{align*}
  \Delta N (y) = \delta_0.
\end{align*}
Given a measure $\mu$, define the potential $U_\mu$ {\red as} 
$$
U_{\mu} (x) := \int N (x - y) d\mu (y). 
$$
Since $\Delta U_{\mu} = \mu$,  {\red elliptic regularity} yields the following:
\begin{lemma}[Spatial continuity]\label{lem:x-continuity-U}
Assume that there are constants $M,R>0$ such that $\mu \le M$ and $\mu=0$ outside $B_R$.
Then $\|U_{\mu}\|_{C^{1,\alpha}(\R^d)} \leq C$ for any $0<\alpha<1$, with $C = C(\alpha,M, R)$.
\end{lemma}

\begin{remark}\label{rem:uniform-bound-mu-t}
 Our main focus is on  the stopping times  $\tau$ generating a compactly supported $\nu$, $W_\tau \sim \nu$, with the upper density constraint {\red $\nu \le f \in L^\infty$.} For such $\tau$,  the measure $\mu_t$, {\red that is} the distribution of $W_{\tau\wedge t} \sim \mu_t$,  is uniformly bounded with respect to $t$, as $\mu_t$ is bounded by the greater of the solution to the heat equation (with initial value $\mu$) and $f$. Also, the support of $\mu_t$ is contained in the convex hull of the support of $\nu$; {\red one can easily verify this from the subharmonic order $\mu_t  \le_{SH}  \nu$  by using linear functions as test functions $\phi$ for $\int \phi d\mu_t \le \int \phi d\nu$.}  
 In what follows we thus  focus on the case where the result of Lemma \ref{lem:x-continuity-U} holds. 
\end{remark}

We now recall a simple consequence of Ito's formula. For each $0\le g \in C^1_c (\R^d)$, consider a subharmonic function $u \in C^2$ such that $\Delta u =g$. Note that one can find such $u\in C^{2, \alpha}$ from elliptic regularity.
Using Ito's formula,  for  each stopping time $\sigma \ge 0$  we have
\begin{align*}
\mathbb{E}[ u (W_{\sigma})  \ | \ W_{0} = y  ]  - u(y) =  \mathbb{E}\left[ \int_0^\sigma   \frac{1}{2}\Delta u (W_t) dt \ \Big| \ W_0 =y \right] \\= \mathbb{E}\left[ \int_0^\sigma   \frac{1}{2} g (W_t) dt  \ \Big| \ W_0 =y \right].
\end{align*}
In particular, we can show 
\begin{lemma}\label{lem:Ito-closed-set}
Suppose that $\mu$ is a probability measure on $\R^d$ and $W_0 \sim \mu$ and $W_{\tau_i} \sim \nu_i$, $i=1,2$ and that the potentials $U_{\nu_1}, U_{\nu_2}$ are measurable functions.  Assume that $\tau_1 \le \tau_2$ {\red almost surely}. 
 Then, for each closed or open  set $E \subset \R^d$, we have for the characteristic function $\chi_E$ that 
\begin{equation}\label{formula:0}
 \int_E \left(U_{\nu_2} - U_{\nu_1}\right)(y) d y \,= \,\mathbb{E}\left[ \int_{\tau_1}^{\tau_2}   \frac{1}{2} \chi_E (W_t) dt\right].
\end{equation}

In general, for a measurable $E$, there exists a monotonically increasing sequence of compact sets $K_n \subseteq E$ with $\lim_{n\to\infty} |E\setminus K_n|=0$ and 
{monotonically decreasing sequence of open sets  with $E\subseteq O_n$ and $\lim_{n\to\infty} |O_n\setminus E|=0$} 
such that 
\begin{align*}
  \int_E \left(U_{\nu_2} - U_{\nu_1}\right)(y) d y  = \lim_{n\to \infty} \mathbb{E}\left[ \int_{\tau_1}^{\tau_2}   \frac{1}{2} \chi_{K_n} (W_t) dt   \right]
=  \mathbb{E}\left[ \int_{\tau_1}^{\tau_2}   \frac{1}{2} \lim_{n\to \infty}\chi_{K_n} (W_t) dt   \right],\\
  \int_E \left(U_{\nu_2} - U_{\nu_1}\right)(y) d y  = \lim_{n\to \infty} \mathbb{E}\left[ \int_{\tau_1}^{\tau_2}   \frac{1}{2} \chi_{O_n} (W_t) dt   \right]
=  \mathbb{E}\left[ \int_{\tau_1}^{\tau_2}   \frac{1}{2} \lim_{n\to \infty}\chi_{O_n} (W_t) dt   \right].
\end{align*}
\end{lemma}
\begin{proof}

{\bf Step 1.} First let us consider the case where $E$ is compact or open {\red and bounded}. 
For a compact set $E=K$, there exists a monotonically decreasing sequence $0\le g_k \in C^\infty_c(\R^d)$, point-wisely converging to $\chi_K$, that is, $\chi_K =\inf_k g_k$. 
 Indeed, take a continuous function $h_k$ that equals $d(x, K)$ if $d(x,K) \leq 1/k$, $1$ if $ d(x,K) \geq 2/k$, and the linear interpolation, {\red $ (k-1)(d(x,K)-1/k) + 1/k$,}  for $1/k \leq d(x,K)\leq 2/k$.  After mollifying, this generates a monotone increasing sequence of $C^1$ functions that converges to $1-\chi_{K}$.
When $E=O$ is open and bounded, we can modify the above construction for $K= (\R^d \setminus O) \cap B_r$ for $r\gg 1$ to find the corresponding $\{g_k\}_k$ converging to $\chi_O$.
 
\medskip

Now, consider the subharmonic functions {\red  $u_k \in C^\infty$ such that $\Delta u_k = g_k$; because $g_k \in C^\infty_c$ we have $u_k \in C^\infty$ from elliptic regularity.} 
We have 
\begin{align*}
\int_E \left(U_{\nu_2} - U_{\nu_1}\right)(y) d y & = \lim_{k\to \infty} \int \left(U_{\nu_2} - U_{\nu_1}\right)(y) g_k (y)d y\\
& \qquad \qquad \qquad  \hbox{ (by the monotone  convergence theorem)} \\
&= \lim_{k\to \infty}  \int \left (\nu_2 - \nu_1\right) (y) u_k  (y)  \, dy  \hbox{ (integration by parts)} \\
  &= \lim_{k\to \infty}  \mathbb{E}[ u_k (W_{\tau_2}) - u_k(W_{\tau_1})] \\
  &= \lim_{k\to \infty}  \mathbb{E}\left[ \int_{\tau_1}^{\tau_2}   \frac{1}{2} g_k (W_t) dt   \right]  \hbox{ (by Ito's formula)}
\end{align*}
where by the  monotone convergence theorem (applied to the Wiener measure on the path space as the monotone convergence $g_k \to \chi_{E}$ can be extended to the path space),   the last line is the same as 
\begin{align*}
 \mathbb{E}\left[ \int_{\tau_1}^{\tau_2}   \frac{1}{2} \chi_E (W_t) dt   \right].
\end{align*}

\medskip

{\bf Step 2.} Next for open  or cloesd $E$,  consider $E_r= E \cap B_r$. Notice $\chi_{E_r} \to \chi_E$ monotonically as $r \to \infty$. {\red We get \eqref{formula:0}  for $E_r$ from  Step 1,  and applying the monotone convergence theorem as  $r\to \infty$, we get it for $E$. }

\medskip

{\bf Step 3.} Now for a measurable $E$, 
one can find a sequence of increasing compact sets $K_1 \subseteq K_2 \subseteq K_3  \subseteq ....$ contained in $E$ such that $\lim_{n\to \infty}  |E\setminus K_n| =0$ and monotonically decreasing sequence of open sets  with $E \subseteq O_n$, $O_1\supseteq O_2 \supseteq O_3 \supseteq \cdots$, and $\lim_{n\to\infty} |O_n\setminus E|=0.$
Let $\{E_n\}$ denote either the sequence $\{K_n\}$ or $\{O_n\}$. 
Then, from the monotone convergence theorem applied to $\chi_{E_n}$, we have
\begin{align*}
 \int_E \left(U_{\nu_2} - U_{\nu_1}\right)(y) d y =\lim_{n\to \infty} \int_{E_n} \left(U_{\nu_2} - U_{\nu_1}\right)(y) d y.
\end{align*}
Since for the sets $ \int_{E_n} \left(U_{\nu_2} - U_{\nu_1}\right)(y) d y =  \mathbb{E}\left[ \int_{\tau_1}^{\tau_2}   \frac{1}{2} \chi_{E_n} (W_t) dt   \right]$ from the previous cases,  we get
\begin{align*}
  \int_E \left(U_{\nu_2} - U_{\nu_1}\right)(y) d y  =  \lim_{n\to \infty}\  \mathbb{E}\left[ \int_{\tau_1}^{\tau_2}   \frac{1}{2} \chi_{E_n} (W_t) dt   \right]
 = \mathbb{E}\left[ \int_{\tau_1}^{\tau_2}   \frac{1}{2} \lim_{n\to \infty}\chi_{E_n} (W_t) dt   \right]
\end{align*}
where for the last equality we applied the monotone convergence theorem to the Wiener measure on the path space as the monotone sequence  $\chi_{E_n}$ can be extended to the path space.
 This completes the proof. 
\end{proof}

The following is a standard fact, but we provide its proof for clarity.
\begin{corollary}\label{lem:U-stopping-strict} 
 Suppose $\mu$ is a bounded, compactly supported measure on $\R^d$ and that  $\tau_1, \tau_2$ are stopping times with the given initial distribution $W_0 \sim \mu$ and with the respective final distribution $W_{\tau_i} \sim \nu_i$, $i=1,2$. Then
\begin{align*}
\hbox{ If $ \tau_1 \leq \tau_2$  almost surely, then $U_{\nu_1} \leq  U_{\nu_2}$ a.e.}.
\end{align*}
Moreover, $U_{\nu_1} = U_{\nu_2}$  a.e. if and only if $\tau_1=\tau_2$ almost surely.
\end{corollary}
\begin{proof} 
 The first statement follows immediately from  Lemma~\ref{lem:Ito-closed-set}. 
Next suppose $U_{\nu_1} = U_{\nu_2}$ {\red a.e.}. {\red Apply Lemma~\ref{lem:Ito-closed-set} to $E=\R^d$ and obtain}
\begin{align*}
 0=\int_{\R^d}  \left(U_{\nu_2} - U_{\nu_1}\right)(y) d y & = \mathbb{E}\left[ \int_{\tau_1}^{\tau_2}   \frac{1}{2}  dt   \right]  =\frac{1}{2} \mathbb{E}[\tau_2 - \tau_1]
\end{align*}
Hence $\mathbb{E}[\tau_2-\tau_1]=0$. Since $\tau_2 \ge \tau_1$, we conclude that $\tau_2 =\tau_1$ almost surely. 
\end{proof}
\begin{remark}
 Note that there are many stopping times with the same final distribution. Therefore, $U_{\nu_1}=U_{\nu_2}$ without an assumption like $\tau_1 \le \tau_2$ does not give much information. 
\end{remark}

Now we define our potential flow. 
\begin{definition}[Potential flow]\label{def:potential_flow}
For {\red compactly supported measures $\mu, \nu \in L^\infty(\R^d)$ with $|\mu| =|\nu|$, } let $\tau$ be a stopping time between $\mu$ and $\nu$, that is $W_0 \sim \mu$ and $W_\tau \sim \nu$. Let $\mu_t$  be the distribution of the Brownian stopping time $\tau\wedge t$, that is, $W_{\tau\wedge t} \sim \mu_t$. 
 We then call $\{U_{\mu_t}\}_{t>0}$ the {\em potential flow} associated with $\tau$.
\end{definition}
From Corollary~\ref{lem:U-stopping-strict} we immediately have 
\begin{corollary}\label{lem:mono-potential}
{\red Under the assumptions of Definition~\ref{def:potential_flow},}
\begin{itemize}
 \item[(a)] $U_{\mu_t}$ is monotonically nondecreasing in $t$;
 \item[(b)]  $U_{\mu} \leq U_{\mu_t}\le U_\nu$ for all $t$;
\end{itemize}
\end{corollary}
{\red 
\begin{proof}
 The items (a) and (b) are obvious from Corollary~\ref{lem:U-stopping-strict}, for under our assumptions and Lemma~\ref{lem:x-continuity-U}, the functions $U_{\mu_t}\, U_\mu, U_\nu$ are all continuous (in fact, $C^{1,\alpha}$ for some $\alpha>0$). 
\end{proof}
}

We now prove continuity of the potential flow in time.
\begin{lemma}[Time continuity]\label{lem:t-continuity-U}
Suppose $C:=\|\mu_t\|_{L^{\infty}(\R^d\times [0,\infty))} <\infty$, and that the support of $\mu_t$ is bounded uniformly in $t$.  Then $U_{\mu_t}$ is Lipschitz in time, more precisely,  for each $x$,
\begin{align*}
0\le  U_{\mu_{t'}} (x) - U_{\mu_t} (x) \le  \frac{C}{2} [t'-t] \quad \hbox{ for all $0\le t <  t'$}.
\end{align*}
\end{lemma}
\begin{proof} 
Recall from Lemma~\ref{lem:Ito-closed-set} that for each open set $E$, 
\begin{align*}
 \int_{E} \left(U_{\mu_{t'}} - U_{\mu_t}\right)(y) d y & = \mathbb{E}\left[ \int_{\tau\wedge t}^{\tau\wedge t'}   \frac{1}{2} \chi_E (W_r) dr  \right] .
\end{align*}
The first inequality then follows from spatial continuity of $U_{\mu_t}$ {\red given in Lemma~\ref{lem:x-continuity-U}.} For the second inequality, 
apply Fubini in the above equality to obtain
\begin{equation}\label{formula:1}
\int_{E} \left(U_{\mu_{t'}} - U_{\mu_t}\right)(y) d y =\frac{1}{2} \int_{t}^{t'} Prob[ W_r \in E \ \& \  r < \tau ] dr.
\end{equation}
Notice that for each $r>0$, by the definition of $\mu_t$,
\begin{align*}
 Prob[ W_r \in E \ \& \  r < \tau ] \le \mu_r [E].
\end{align*}
Since $\mu_r \le C$, this and \eqref{formula:1} implies
\begin{align*}
 \frac{1}{|E|} \int_{E} \left(U_{\mu_{t'}}- U_{\mu_t}\right)(y) d y \le \frac{C}{2} (t'-t).
\end{align*}
Now set $E=B_r(x)$ and let $r\to 0^+$, which completes the proof.  
 \end{proof}

{\red From Lemma~\ref{lem:x-continuity-U} 
and Lemma~\ref{lem:t-continuity-U},} the following holds: 
\begin{corollary}[Space-time Lipschitzness of the potential]\label{cor:uniform-U}
 Assume that there exists a constant $C>0$ such that $\|\mu_t\|_\infty \le C$ for all $ t \ge 0$, and that the support of $\mu_t$ is bounded uniformly in $t$. Then,
 {\red  $U_{\mu_t}$ is uniformly $C^{1,\alpha}$ in space and Lipschitz in time in $\R^d  \times [0,\infty)$.} 
\end{corollary}

\subsection{Hitting times to monotone barriers and potential functions}
In this section we investigate the relation between the stopping time $\tau$ and 
the hitting time to a barrier set generated by the potential $U_{\mu_t}$.
 \begin{definition}\label{def:potential-barrier}
 For the potential flow $\{U_{\mu_t}\}_{t\geq0}$ given in Definition~\ref{def:potential_flow},  we define the forward/backward stopping times 
 \begin{align*}
 \tau^{U,f}  := \inf \{ t \ | \ U_{\mu_t}  (W_t) =  U_\nu(W_t)\}, \quad  \tau^{U,b}     := \inf \{ t > 0\ | \ U_{\mu_t}  (W_t)  =U_\mu(W_t) \}.
\end{align*}
{\red (Notice that the condition $t>0$ for $\tau^{U,b}$ is necessary: otherwise $\tau^{U,b} \equiv 0$.
)}
 We also define the corresponding barrier functions and the barrier sets:
$$
 s^{U,f} (x):= \inf \{ t \ | \  U_{\mu_t}(x) =U_\nu (x) \}, \,\,R^{U,f} : = \{ (x, t) \ | \ t \ge s^{U, f} (x)\} \\
$$
and
$$
s^{U,b} (x):= \sup \{ t  \ | \  U_{\mu_t}(x) =U_\mu (x) \}, \,\, R^{U,b} : = \{ (x, t) \ | \  t \le s^{U,b} (x)\}.
$$ 
\end{definition}
\medskip

It is important for us to observe the following: 
 \begin{lemma}\label{lem:Z-U-cloed}
Assume $\mu_t$ and its support are uniformly bounded for all {\red $0\le  t \le \infty$.} Then
\begin{itemize}
 \item[(a)]
  $s^{U,f}$ and $-s^{U,b}$ are lower semicontinuous. {\red Equivalently,}
the sets $R^{U, f}$ and  $R^{U,b}$ are closed in $\R^d \times [0,\infty]$;

 \item[(b)] we have
   \begin{align*}
 \tau^{U,f}  = \inf \{ t \ | \ s^{U, f} (W_t) \le t\}, \quad  \tau^{U,b}     = \inf \{ t \ | \ s^{U, b} (W_t) \ge t >0\}.
\end{align*}
\end{itemize}
\end{lemma}
\begin{proof}
These  follow from Corollary~\ref{cor:uniform-U} and time monotonicity of $U_{\mu_t}$ in Corollary~\ref{lem:mono-potential}(a)(b).
\end{proof}

\medskip

  The potential flow $U_{\mu_t}$ can be defined for any stopping time, even for randomized stopping times, so the hitting times $\tau^{U, f}, \tau^{U,b}$ can be associated to any (randomized) stopping time $\tau$. On the other hand it is easy to see that $\tau$ and either of $\tau^{U, f}, \tau^{U,b}$ do not coincide in general; even {\red for} a {\red hitting time} $\tau$, it will not be equal to $\tau^{U, f}, \tau^{U,b}$  {\red when its barrier is not monotone in time. }
  {\red For type (I) case, from the result of \cite{gassiat2021}, we have $\tau=\tau^{U,f}.$
To our knowledge {\red a similar result in such a generality} is not available for the type (II) case. 
}
  {\red When  $\tau=\tau^*$ is an optimal stopping time 
   (as in Theorem~\ref{thm:optimal-Skorokhod-existence}), 
  we verify that it is equal  to  $\tau^{U, f}, \tau^{U,b}$, for type (I), (II), respectively;} 
  this is proved in Theorem~\ref{thm:tau-equal-tau-U} below.   
{\red 
 It will then imply that the sets $R^{U, f}$ and  $R^{U,b}$, in type (I) and (II), respectively, give the barrier set of $\tau^*$. 
 Closedness of these barrier sets (Lemma~\ref{lem:Z-U-cloed} (a)) allows us to apply the maximum principle in the PDE formulation in Section~\ref{sec:consistency-Stefan}.
 \medskip
 }

Note that the equivalence between $\tau$ and  $\tau^{U, f}$ or $\tau^{U, b}$, is hinted in the formula \eqref{formula:0} which is a consequence of Ito's formula. One sees there that spending more time  for the Brownian motion increases the potential function, which is precisely controlled but, only in integral/expected value sense; it is not obvious how to use such a control on average quantities to get information for each {\red individual} Brownian path. 
{\red We overcome this difficulty by optimality of $\tau$, in particular, using Proposition~\ref{lem:hitting_s}.} 
\medskip

{\red Towards the main result of this section,} in the next proposition
we show that the barrier set $R$ 
 of $\tau$ (for example, the set $R^*$ for the optimal stopping time $\tau^*$ as in Section~\ref{sec:prelim}) {\red is contained, except a measure zero set, \red in the barrier set $R^{U,f}$, $R^{U,b}$, for  type (I),(II), respectively.}
 This is basically a result of the Ito's formula via Lemma~\ref{lem:Ito-closed-set}.

\begin{proposition}\label{lem:nu-s-less-s-u-zero}
Let $\mu,\nu$ and $\tau$ be as given in Definition~\ref{def:potential_flow}.  Suppose $\tau$ is characterized as {\red the hitting time}  
$\tau= \inf \{ t \ | \ t \ge s(W_t)\}$ for type (I), $\tau= \inf \{ t \ | \ 0<  t \le s(W_t)\}$ for  type (II), for a measurable function {\red $s:\R^d\to [0,\infty]$.}  Then we have
{\red 
$$
|\{ x  \in \R^d \ | \ s(x)  < s^{U, f}(x)\}|=0\hbox{  for } {\rm (I)} \,\,\hbox{  and  } |\{ x \in \R^d \ | \ s(x)  > s^{U, b}(x)\}|=0 \hbox{ for } {\rm (II)}.
$$ 
In particular, since $\nu \ll Leb$,}
$$
\nu[\{ x  \in \R^d \ | \ s(x)  < s^{U, f}(x)\}]=0\hbox{  for } {\rm (I)} \,\,\hbox{  and  }\nu[\{ x \in \R^d \ | \ s(x)  > s^{U, b}(x)\}]=0 \hbox{ for } {\rm (II)}.
$$ 

\end{proposition}
{\red 
\begin{proof}
We will prove for costs of type (II), since the argument is parallel for the other case. 

\medskip

Suppose not, that is, $|\{ s >  s^{U, b}\}|>0 $.  We can then find a constant $\bar t \in [0,\infty)$ such that 
$$
 |E| >0 \hbox{ where } E:= \{ x \ | \ s(x)  > \bar t >  s^{U, b}(x)\}.
$$ 
Let us consider
  $w(x, t) := U_{\mu_t} (x) - U_{\mu} (x)$.  {\red From the definition of $s^{U,b}$ and $E$ we have $w(\cdot , \bar t ) >0$ on $E$,   yielding}
  \begin{align*}
\int_{E} w(y, \bar t) dy >0.
\end{align*}
\medskip

From Lemma~\ref{lem:Ito-closed-set}, there exists a monotonically increasing sequence of compact sets $K_n \subseteq E$ such that 
\begin{align*}
  \int_E w(y, \bar t) dy &= \lim_{n\to \infty} \mathbb{E}\left[ \int_0^{\bar t \wedge \tau}   \frac{1}{2} \chi_{K_n} (W_t) dt \right].
  \end{align*}
  By Fubini's theorem, we thus can write
  \begin{align*}
  \int_E w(y, \bar t) dy   \,  = \lim_{n\to \infty} \frac{1}{2} \int_0^{\bar t} Prob[ W_r \in K_n \ \& \  r < \tau ] dr
   \,   \le \frac{1}{2} \int_0^{\bar t} Prob[ W_r \in E \ \& \  r < \tau ] dr.
    \end{align*}
On the other hand, recall that $\tau= \inf\{t \ | \ 0<  t\le s(W_t)\}$.  Hence {\red for $r \le \bar t$, if $W_r \in E$, %and $r\geq \bar t$,
 then $s(W_r) >  \bar t \geq r$ and thus  $\tau \leq r$. } Therefore
\begin{align*}
  Prob[ W_r \in E \ \& \  r < \tau ] =0 \hbox{ for $r \le \bar t$.}
\end{align*}
Back to the previous inequalities of integrals, we get
\begin{align*}
   \int_{E} w(y, \bar t) dy 
\le0,
\end{align*}
a contradiction, completing the proof. 
\end{proof}

}

\begin{definition}\label{potential_barrier}
We let $\tau^U$ commonly denote  $\tau^{U,f}$ and $\tau^{U,b}$ depending on whether $\mu_t$ is generated from  type (I) or (II) stopping time $\tau$. {\red Likewise we use $R^U$ for  the barrier sets $R^{U,f}$ and $R^{U,b}$, and $s^U$ for  $s^{U,f}$ and $s^{U,b}$.}
\end{definition}

\begin{remark}
 Proposition~\ref{lem:nu-s-less-s-u-zero} by itself may not give much information on the relation between $\tau$ and $\tau^U$, as the stopped distribution of $(W_\tau, \tau)$ in general can even be outside the barrier set $R$, in particular, if the barrier function has discontinuity in a dense subset of $\R^d$, of the Hausdorff dimension $d-1$. \end{remark}

{ \red 
We are now ready to prove our main result of this section, which is the equivalence between $\tau$ and $\tau^U$.  Such a result is known for type (I)  \cite{gassiat2021}. 
We prove  the equivalence  for both type (I) and (II) by employing  
 the connection between 
the Brownian motion with stopping time  and  a parabolic flow, assuming optimality of the stopping time $\tau$. 
In Section~\ref{sec:consistency-Stefan} consistency results will be proved using a slightly weaker result,  namely,  that the set  $\{\eta>0\}$, for the Eulerian variable $(\eta, \rho)$, is open.

\begin{theorem}\label{thm:tau-equal-tau-U}[see also \cite{gassiat2021} for type (I)]. 
Use the same assumptions as in Propositions~\ref{lem:hitting_s} and \ref{lem:nu-s-less-s-u-zero}
; in particular, $\tau=\tau^*$ is an optimal stopping time of Theorem~\ref{thm:optimal-Skorokhod-existence}, and in the type (II) case $\mu\wedge \nu=0$. Then, 
\begin{itemize}
\item[(a)]
 for the Eulerian variable $(\eta, \rho)$ that corresponds to $\tau^*$ we have 
 \begin{align*}
 \{\eta=0\} = R^U,  \ \ {\red \hbox{in particular,  the set $\{ \eta >0\}$  is open;}}
 \end{align*} 
% without loss of generality
 \item[(b)]
 $\tau^* = \tau^U$ almost surely, in particular, 
the set $R^U$ gives a barrier for $\tau^*$ as a hitting time to it. 
 \end{itemize}
\end{theorem}
\begin{proof}
Because of Proposition~\ref{lem:nu-s-less-s-u-zero}, we have 
\begin{align*}
 s^*(W_{\tau^*}) \ge s^U(W_{\tau*}) \hbox{ a.s. for type (I)},  \quad s^*(W_{\tau^*}) \le s^U(W_{\tau*}) \hbox{ a.s. for type (II)}.
\end{align*}
This implies 
\begin{align*}
 \tau^* \ge \tau^U  \quad \&  \quad (W_{\tau^*}, \tau^*) \in R^U\cap R^* \quad \hbox{ a.s.,}
\end{align*}
  since $\tau^*=s^*(W_{\tau^*})$ a.s. from Proposition~\ref{lem:hitting_s}.
In particular, we can find a barrier set $R$ of $\tau^*$ such that $R \subseteq R^U$.  To see this, let $R = R^* \cap R^U$ and let $\tilde \tau$ be its hitting time. Since $(W_{\tau^*}, \tau^*) \in R^U\cap R^*$ we have $\tilde \tau \le \tau^*$ almost surely. 
On the other hand from the definition of hitting time we have $\tilde \tau \ge \max[\tau^* , \tau^U]$ while the latter is nothing but $\tau^*$ from $\tau^* \ge \tau^U$, so we conclude that $\tau^*$ is the hitting time $\tilde \tau$ to $R=R^*\cap R^U$.  Below we use this $R$. 
\medskip

From Proposition~\ref{thm:stochastic_embedding}, 
the stopping time $\tau^*$   induces its Eulerian flow $(\eta, \rho)$ 
 which satisfies \eqref{eqn:EulerianPDE} with $\eta[R]=0$ and $\rho[R]=1$; {\red in particular, $\rho$ is the distribution of $(W_{\tau^*}, \tau^*)$}. Furthermore from \cite{PDE19} we have {\red  $\eta\in L^2(\R_{\ge 0}; H^1_0(\R^d))$.}
 
 \medskip
 
 Let us denote $w:= U_{\nu} - U_{\mu_t}$ for (I),  and $w:= U_{\mu_t} - U_{\mu}$ for (II). For the rest of the proof we focus on type (II), as (I) follows a parallel proof, and also because for (I) a different proof is available in \cite{gassiat2021}.  

\medskip

Since $\Delta w = \mu_t -\mu$ at each $t\geq 0$,  {\red from the definition of $\mu_t$ and  \eqref{eqn:EulerianPDE} we have} 
\begin{equation}\label{eq:potential-U-t}
		\Delta w(x) = \eta(t,x) + \int_0^t\rho(dr,x) -\mu \,\,\hbox{ in } \R^d, {\red \hbox{ for type (II)}.}
	\end{equation}
Now, for $g\in C(\R_{\ge 0}\times \R^d)$ with $g(t,\cdot)\in C_c(\R^d)$ for each  $t>0$, consider $\varphi$ solving {\red $\Delta \varphi(t,\cdot)=g(t,\cdot)$} for each $t>0$, with decay at infinity. 
Using \eqref{eqn:EulerianPDE} for $(\eta, \rho)$ and \eqref{eq:potential-U-t}, we integrate by parts to obtain 
\begin{align*}
	\int_{\R^d}\int_{\R_{\ge 0}} g(x,t) \frac{\partial}{\partial t} w (x)dtdx & = 
	\int_{\R^d}\int_{\R_{\ge 0}} \varphi(\partial_t\eta +\rho)dtdx 
	=\int_{\R^d}\int_{\R_{\ge 0}} \varphi (\frac{1}{2}\Delta\eta)dtdx \\
	&= \int_{\R^d}\int_{\R_{\ge 0}} \frac {1}{2} g(x,t)\eta (x,t) dtdx.
\end{align*}
Therefore we have
\begin{align}\label{eqn:partial-t-w}
\hbox{$\partial_t w = \frac{1}{2} \eta$ (in type (I), the sign is opposite.)}
\end{align}

\medskip

   From \eqref{eqn:partial-t-w} and the fact that $w(x,0)=0$, it follows that 
   \begin{equation}\label{integral}
   w(x,t)= \frac{1}{2}\int_0^t\eta(x,r) dr.
   \end{equation} Since $\eta$ is nonnegative, it follows that the set $(R^U)^c=\{w>0\}$ includes $\{\eta>0\}$, namely $$\{\eta>0\} \subseteq (R^U)^c.$$ 
       {\red As $\rho$ is the distribution of $(W_{\tau^*}, \tau^*)$, Proposition~\ref{lem:hitting_s} and Proposition~\ref{lem:nu-s-less-s-u-zero} together}
          imply that $$\rho((R^U)^c)=0,$$ and thus $\eta$ solves the heat equation in $(R^U)^c$. {\red Therefore}  $\eta$ is positive in the connected open component $\mathcal{R}$ of $(R^U)^c$ that contains $\{\mu >0\} \times \{t=0\}$, and $\eta$ is zero in any other component of $(R^U)^c$, yielding that $\{\eta>0\} = \mathcal{R}$.  It follows that the set $\{\eta>0\}$ increases in time, and from \eqref{integral} it follows that $\{w>0\} = \mathcal{R}$. We now conclude {\red from the definition of $w$} that $\mathcal{R}=(R^U)^c$ and thus  shows the item (a).  {\red  The set is open due to  Lemma~\ref{lem:Z-U-cloed} (a).}

\medskip

Now let us prove the item (b).  
Let $(\eta^U, \rho^U)$ be the Eulerian flow of $\tau^U$ that is the hitting time to $R^U$. 
From the definition of the hitting time we see  $\eta^U(R^U)=0$.  Also, since $R^U$ is a closed set, we have that $\rho^U(R^U) =1$.
From the previous paragraph, {\redf $\eta(R^U) =0$, $\rho (R^U)=1$},  thus, by applying the uniqueness in Lemma~\ref{lem:Eulerian_uniqueness}, we get 
$(\eta, \rho)=(\eta^U, \rho^U)$. In particular, this implies that $W_{\tau^U} \sim \nu$. Since  we already have $\tau^* \ge \tau^U$ a.s. and $W_{\tau^*} \sim \nu$, we conclude $\tau^U = \tau^*$ a.s. verifying the item (b). 
  \end{proof} 
  
\begin{remark}\label{rmk:with-non-local}
The simple but powerful argument of part (b) is borrowed from  \cite{CKKN23}.  In the older argument of a previous version of this paper, we used an obstacle PDE \eqref{eqn:obstacle1} to derive a weaker result (only $s^*=s^U$ $\nu$-a.e with an additional assumption). In  \cite{CKKN23}, the results of Theorem~\ref{thm:tau-equal-tau-U} are extended to nonlocal case (jump processes), especially in the type (II) case; type (I) case on the other hand is already a result of \cite{gassiat2021} even for the nonlocal case. 
\end{remark}

  }

 \medskip
 
 Theorem~\ref{thm:tau-equal-tau-U} states that if $\tau=\tau^*$ is the {\red optimal hitting time obtained in Theorem~\ref{thm:optimal-Skorokhod-existence}, 
then the barrier can be given by $R^U$. }
 {\red As the potential flow is a natural object associated to a stopping time, we regard $R^U$ as the canonical barrier associated to $\tau=\tau^*$, $\mu$ and $\nu$. 
  From now on we mean $R^U$ whenever we say a monotone barrier $R$ for an optimal stopping time $\tau^*$.
 Using closedness of $R^U$ we can in fact show that $R^U$ (so $R$ in our convention)
  is not dependent on $\tau^*$ (more precisely, the cost function it optimizes)}  but only on $\mu$ and $\nu$, as long as the optimal stopping time is given by the hitting time to a monotone barrier. 
  In particular this yields {\red uniqueness of  the optimal stopping time $\tau^*$} (hitting a forward or backward barrier set) for given $\mu$ and $\nu$, {\red depending only on the type, not on a particular cost. }

 \begin{lemma}[See also Remark 6.19 of \cite{beiglboeck2017optimal}]\label{lem:unique-R}
 {\red  

Let $\mu, \nu$ satisfy the assumptions of  Definition~\ref{def:potential_flow} and Theorem~\ref{thm:optimal-Skorokhod-existence}.
Then, there exists unique optimal stopping time $\tau^*$ between them for a type (I) cost, the same optimal $\tau^*$ for any choice of  type (I) cost; in particular, the barrier set $R^U$ for an optimal $\tau^*$
 is uniquely determined by $\mu$ and $\nu$. A same result holds for type (II) case, further assuming $\mu\wedge\nu=0$. }
\end{lemma}
{\red
\begin{proof}
 As the barrier $R^U$ is closed, it follows from the argument of Loynes \cite{Loynes1970} that the barrier set is unique; for example, follow the argument of \cite[Remarks 2.3 and 6.19]{beiglboeck2017optimal} using the closedness. Since $R^U$ determines $\tau^*$ from Theorem~\ref{thm:tau-equal-tau-U}, uniqueness of $\tau^*$ follows as well. 
\end{proof}
\begin{remark}\label{rmk:barrier-unique}
 A first result about uniqueness of the barrier set for a stopping time between two given measures has been made by Loynes \cite{Loynes1970} by introducing the notion of regular barriers  in the $1$ dimensional case with $\mu=\delta_0$, which then was extended to more general measures (in $1$-D) by Gassiat, Oberhauser, and dos Reis  
 \cite[Section 3.3. Lemma 2]{gassiat2015root} who introduced $(\mu,\nu)$-regular barriers. Their notion coincides with $R^U$ in our case. 
\end{remark}
}

\begin{definition}\label{barrier_set}
 For $\mu, \nu$ as given in Lemma~\ref{lem:unique-R},
$(R, s(x))$ is each {\it the barrier set} and {\it the barrier function} associated to $(\mu, \nu)$ for the  cost type (I) or the cost type (II), if {\red $R=R^U$, $s=s^U$ is the barrier set and the barrier function, uniquely given  
 in  Lemma~\ref{lem:unique-R},  either in the form of $R:= \{(x,t): t\geq s(x)\}$ for cost (I), or $R:= \{(x,t):0< t\leq s(x)\}$ for cost (II).}
\end{definition}

For a given barrier $R$, the corresponding Eulerian flow is uniquely determined, due to Lemma~\ref{lem:Eulerian_uniqueness}. This justifies the following definition:	
\begin{definition}\label{Eulerian}
We say the pair $(\eta, \rho)$ is the {\it Eulerian variables associated with } $(\mu, \nu)$
for the cost type (I) (or (II)) if  it is the unique pair $(\eta, \rho)$ solving \eqref{eqn:EulerianPDE} in the weak sense
with the property $\eta \in L^2([0,\infty);H^1_0(\overline{O})$,  with $\eta(R)=0$ and $\rho(R)=1$ for  {\red the  unique barrier }$R=R^U$ determined by $\mu, \nu$.  Here,  $O\subset \R^d$ is a bounded open convex set that contains the supports of $\mu$ and $\nu$.
\end{definition}

\subsection{Remarks on the potential approach and the parabolic obstacle problem}

 One may prove Lemma~\ref{lem:unique-R} 
by using a comparison principle, viewing the functions $w(x,t):= U_{\nu} - U_{\mu_t}$ for type (I), $w:= U_{\mu_t} - U_{\mu}$ for (II) as solutions to  parabolic obstacle problems:
\begin{align}\label{eqn:obstacle1}
 \min[\partial_t w - \frac{1}{2}\Delta w +\frac{1}{2} \nu,w]=0, \quad w(\cdot,0)=U_\nu -U_\mu \quad\hbox{ for type (I)};\\\nonumber
\min[ \partial_t w - \frac{1}{2} \Delta w +\frac{1}{2} \nu-\frac{1}{2}\mu, w]=0, \quad w(\cdot,0) = 0, \quad \hbox{ for type (II)}.
  \end{align}
Such problems have been actively studied in the literature.  In fact, this connection between the parabolic obstacle {\red problems} and the {\red supercooled} Stefan problem has been used in \cite{DbF84} to introduce a notion of weak solution solely based on  \eqref{eqn:obstacle1} 
 for a regularized version of $(St_1)$. However, even ignoring the regularization, {\red such a notion of weak solution} has its limitations due to lack of sufficient regularity of $w$-variable to track the problem back to $(St_1)$. We refer to \cite{C04} and \cite{LM15} for available results on regularity and singularity for solutions of \eqref{eqn:obstacle1}.  We point out that the low regularity of $w$ {\red for type (I), thus $(St_1)$, is related  to the fact that % $w_t$ 
{\red $\partial_t w$} 
is nonpositive; {\red see \eqref{eqn:partial-t-w}}.}   When {\red $\partial_t w$}  is nonnegative, which corresponds to the costs of type (II) and thus $(St_2)$, much stronger regularity results holds for the parabolic obstacle problem: see e.g. \cite{F21}.

\section{Consistency with the Stefan problem}\label{sec:consistency-Stefan}

Let  $(\eta, \rho)$ be the Eulerian variables associated with $(\mu, \nu)$, given in Definition~\ref{Eulerian}.  In this section we will show that $\eta$ solves the Stefan problem, with initial distribution $\mu$ and weight $\nu$, and vice versa. This connection has been indicated in \cite{PDE19} with formal analysis.

Let us define weak solutions of the (weighted) Stefan problems $(St_1)_{\nu}$ and $(St_2)_{\nu}$ with initial density $\eta_0$ and initial domain $E$:
\begin{equation} \label{Stefan}
(\eta \mp  \nu\chi_{\{\eta>0\}})_t - \frac{1}{2}\Delta\eta = 0,  \qquad \eta(\cdot,0) =\eta_0 \in L^1(\R^d), \quad E:=\limsup_{t\to0^+}\{\eta(\cdot,t)>0\}.
\end{equation}
\begin{definition}
A nonnegative function $\eta\in L^1(\R^d\times [0,\infty))$ is a weak solution of $(St_1)_{\nu}$ $({\rm or } \,\,(St_2)_{\nu})$ with initial data $(\eta_0, E)$  with $\eta_0\in L^1(\R^d)$ if 
\begin{itemize}
\item[(a)] the set $\{\eta(\cdot,t)>0\}$ decreases (or increases) in $t$;  
\item[(b)] $E= \limsup_{t\to0^+} \{\eta(\cdot,t)>0\}$; \\
\item[(c)] for any test function $\varphi\in C^{\infty}_c(\R^d\times [0,\infty))$,
 \begin{equation}\label{weak:stefan}
\int_0^{\infty}\int_{\R^d} [(\eta - (\hbox{ or } +)  \nu\chi_{\{\eta>0\}})\varphi_t +\frac{1}{2}\eta\Delta\varphi] dx dt =  \int_{\R^d} [(\eta_0 + (\hbox{ or } -)\nu\chi_E)\varphi ](x,0) dx.
\end{equation}
\end{itemize} 
 
We say that $\eta$ is a weak solutions of $(St_i)$ if it is for $(St_i)_{\nu}$ with $\nu=1$, for $i=1,2$.  
\end{definition}
\begin{remark}\label{rmk:St-St-nu}
Let us point out that  $(St_i)_{\nu}$, $i=1,2$ are characterized by only the value of $\nu$ restricted on the trace of the interface  {\red (free boundary)} within times in $(0,\infty)$. In other words, a weak solution $\eta$ of $(St_i)_{\nu_1}$ is also a solution of $(St_i)_{\nu_2}$, $i=1,2$,  if $\nu_1=\nu_2$ in the set $\{x: 0<s(x)<\infty\}$ where $s(x):= \sup\{t>0: \eta(x,t)>0\}.$ 
 To see this, note that 
$$
\int_0^{\infty} \int \nu\chi_{\{\eta>0\}} \varphi_t  \,\,dx dt = \int_0^{\infty} \nu(x) \int_0^{s(x)}\varphi_t dt dx= \int_0^{\infty} \nu(x)(\varphi(x,s(x)) - \varphi(x,0))dt dx, 
$$
So the only value of $\nu$ that matters is at $x$ with $0<s(x)<\infty$.  
 \end{remark}

\begin{remark}
Note that the weak solution for $(St_1)_{\nu}$ or $(St_2)_{\nu}$ requires specifying not only the initial data $\eta_0$ but the initial domain $E$  to be solved as an initial-value problem. With this information we can find a unique solution of $(St_2)$ by comparison principle, for instance see \cite{AL83}. On the other hand, even with specified $\eta_0$ and $E$, $(St_1)$ can exhibit a high degree of non-uniqueness, as we will see in Section~\ref{sec:nonuniqueness}.
\end{remark}

 \begin{theorem}\label{thm:consistency} 
Let $\mu,\nu$ be compactly supported and satisfy \eqref{assumption_measure}. Let $(\eta,\rho)$ be the associated Eulerian flow given in Definition~\ref{Eulerian} for either (I) or (II). If the optimal stopping time {\red $\tau$} associated with $(\mu, \nu)$ is strictly positive,  the following holds:
 \begin{itemize}
 \item[(a)] For (I), $\eta$ is a weak solution $(St_1)_{\nu}$ with initial data $(\mu, E)$, where $E$ is a set containing the support of \textcolor{blue}{$\nu$}. \\
 \item[(b)] For (II), $\eta$ is a weak solution $(St_2)_{\nu}$ with initial data $(\mu, E)$ where $E$ contains the support of $\mu$.  
 \end{itemize}
\end{theorem}

\begin{proof}
{\red Let us remark that in the type (II) case, optimal stopping time being positive is equivalent to $\mu \wedge \nu =0$ as seen in Theorem~\ref{thm:optimal-Skorokhod-existence}.}
\medskip

Let $R=R^U$ and $s(x)$ be the {\it barrier set and function}  associated with $(\mu, \nu)$ 
for either the cost of type (I) and (II) given in Definition~\ref{barrier_set}; {\red these are associated with the optimal stopping time (c.f. Lemma~\ref{lem:unique-R})}.  The following holds from Proposition~\ref{lem:hitting_s} and Theorem~\ref{thm:tau-equal-tau-U}:
\begin{itemize}
\item[(i)] $ \rho$ is supported on $\{t=s(x)\}$;
\item[(ii)]  $\{\eta>0\} = {\red R^C} = {\red \{0< t<s(x)\}}$ (or $\{t>s(x)\}$)  for (I) (or for (II)).
\end{itemize}

\medskip
 
For the next set of computations we focus on (I). Let {\red $\varphi \in C_c^{\infty}(\R^d \times [0,\infty))$}. 
Recall that $(\eta,\rho)$ then satisfies \eqref{eqn:EulerianPDE}, namely
$$
\int_0^{\infty}\int \eta \left( \varphi_t + \frac{1}{2}\Delta\varphi\right) dx dt = \int_0^{\infty}\int \rho\varphi dx dt + \int \mu\varphi(x,0) dx.
$$
We would like to see that $\eta$ satisfies the weak equation for $(St_1)_{\nu}$. To this end observe that
  $$
     \int_0^{\infty} \int \rho\varphi dx dt = \int\left( \int_0^{\infty}  \rho(x,t)\varphi(x,t) dt\right) dx = {\red  \int_{\{ s(x) >0\}} \nu(x) \varphi(x,s(x)) dx,}
     $$
   where the second equality comes from (i) and the fact that $\int_0^{\infty} \rho(x,t) dt = \nu(x)$ a.e. $x$, and {\red the fact that the optimal stopping time is strictly positive so the set $\{ s(x)>0\}$ contains the set $\{ \nu(x) >0\}$.}

\medskip

  Next, from (ii) {\red we have  $\{\eta>0\} = \{ t< s(x)\}$ and $E= \limsup_{t\to0^+} \{\eta(\cdot,t)>0\} =\{  s(x) >0\}$}. Thus  by Fubini's theorem
  $$
  \begin{array}{lll}
 {\red  \int_{\{s(x) >0\}} } \nu(x)\varphi(x,s(x)) dx &=& \int \nu(x) \int^{s(x)}_0 \varphi_t(x,t) dt dx + \int_{\{s(x)>0\}} \nu(x)\varphi(x,0) dx \\ \\
  &=& \int_{{\{\eta>0\}} }\nu(x)  \varphi_t (x,t) dt dx + \int \nu(x)\varphi(x,0) \chi_E(x) dx. 
  \end{array}
 $$
  Hence \eqref{eqn:EulerianPDE} can be written as 
 $$
\int_0^{\infty}\int \eta\left(\varphi_t +\frac{1}2{}\Delta\varphi\right) dx dt = \int \nu(x)\chi_{\{\eta>0\}} \varphi_t (x,t) dxdt +{\red \int [\mu(x)+\nu (x)]\chi_E(x)\varphi (x,0) dx.}
$$
 
 \medskip
 
 Hence $\eta$ is a weak solution $\eta\in  L^1(\mathbb{\R}^d\times [0,\infty))$ of  $(St_1)_\nu$, with initial data $\eta_0 = \mu$ and the initial trace {\red  $E =\{ s(x) >0\}$; notice that as $R=R^U$, this set is an open set. Note that $E$ contains the set $\{ \nu(x) >0\}$ and as it is an open set it thus contains the support of $\nu$. }
  
 \medskip
 
For (II), since  $\{\eta(\cdot,t)>0\}=\{s(x)<t\} $,  it follows from our assumption $\tau>0$ and {\red (i)},  that $\nu=0$ in $E$; {\red also, $E =\lim_{t\to 0^+} \{s(x)<t\} =\{ s(x) =0\}$. Moreover, since $\tau>0$ almost surely, for $\mu$-a.e. $x$, it holds that $s(x) >0$, therefore, the set $E$ contains the set $\{\mu (x)>0\}$. As $E$ is a closed set, it contains the support of $\mu$. Now,}  by Fubini's theorem
\begin{equation}\label{*}
 \int \nu(x)\varphi(x,s(x)) dx  = -\int \nu(x) \int_{s(x)}^{\infty}  \varphi_t(x, t) dt dx  = -\int_{\{\eta>0\}} \nu(x) \varphi_t(x,t) dt dx,
 \end{equation}
 
  Hence \eqref{eqn:EulerianPDE} for case (II) can then be written as 

 \begin{equation*}
\int_0^{\infty}\int \eta(\varphi_t +\frac{1}{2} \Delta\varphi) dx dt = -\int_{\{\eta>0\}} \nu(x)\varphi_t(x,t)dx dt +\int \mu(x)\varphi (x,0) dx,
\end{equation*}

which is the weak expression for $(St_2)_\nu$ with initial data $\mu$ and $E \subset \{\nu=0\}$. 
\end{proof}

\begin{remark}
In general  the theorem will hold with the revised initial data $\mu - \mu_0$ and revised target measure $\nu-\mu_0$, where $\mu_0$ is the portion of $\mu$ with $\tau=0$, namely 
 $$\mu_0 (x) := \,{\rm Prob} [ \tau =0\, |\, W_0 = x] \,\mu(x).$$
 \end{remark}

Next we consider the reverse direction, starting with $(St_1)_\nu$.

\begin{theorem}\label{reverse}
Let $\mu,\nu$ satisfy \eqref{assumption_measure} and compactly supported. Suppose  $\eta$ is a weak solution of $(St_1)_\nu$ with initial data $(\mu, E)$, where $E$ is bounded. Let us define $s$ and $\rho$ by 
 $$
 s(x):= \sup \{t>0: \eta(x,t)>0\} \hbox{ and } 
 $$
    \begin{equation}\label{def_rho}
 \int\int \rho(x) \varphi(x,t) \, dx dt =  \int \nu(x) \varphi(x,s(x)) dx
 \end{equation}
 for any test function $\varphi\in C^{\infty}_c(\R^d\times [0,\infty))$.  Then $\mu\leq_{SH} \tilde{\nu}:= \nu\chi_{\{s(x)<\infty\}}$. Moreover 
  $(\eta,\rho)$ is the Eulerian variables between $\eta_0=\mu$ and $\tilde{\nu}$, generated by the optimal stopping time with costs of type (I).
  
   \end{theorem}

\begin{proof}
By definition of $s(x)$, $\{\eta>0\}=\{t<s(x)\}$. This allows us to write, for instance in the case of $(St_1)$,
\begin{align*}
 \int \nu(x)\varphi(s(x),x) dx - \int \nu(x)\varphi(0,x) dx = \int \nu(x) \int^{s(x)}_0 \varphi_t(t,x) dt dx  \\= \int\int_{\{\eta>0\}} \nu(x) \varphi_t(t,x) dt dx.
 \end{align*}
 Arguing as in the proof of Theorem~\ref{thm:consistency}, we can then  verify that $(\eta,\rho)$ satisfies the equation \eqref{eqn:EulerianPDE}. Moreover, we have $\eta(R)=0$ for the barrier set $R:= \{(x,t): t \geq s(x)\}$ by the definition of $s(x)$. Observe also that 
 $$
 \rho(R)= \int \rho(x,t) dx dt = \int\eta(x,0) dx = \mu(\R^d)=1,
 $$ 
 where the first equality holds since $\rho$ is supported in $R$, and the second is due to the mass preserving property of the heat equation.

 \medskip
 
  Let us define $\xi:= \int_t^{\infty} \eta(x,t) dt$. Integrating in time the weak equation for  $(St_1)_{\nu}$, we see that $\xi$ solves the parabolic obstacle problem
 $$
 \xi_t - \frac{1}{2}\Delta \xi = -\frac{1}{2}\nu\chi_{\{\xi>0\}}. 
 $$ 
 On the other hand, since we have $\eta_t - \Delta \eta = -\rho$ in the distribution sense,  and since $\eta$ vanishes as time tends to infinity, it follows that  
 $$
\mu = \eta(\cdot,0)=  \int_0^{\infty} (-\eta_t) = \int_0^{\infty} (-\Delta \eta + \rho) =  -\Delta \left(\int_0^{\infty} \eta\right) + \tilde{\nu}.
 $$
 Hence $\mu - \tilde{\nu} = -\Delta \xi(\cdot,0)$. Thus, by uniqueness of the parabolic obstacle problem, $$\xi(x,t)=w(x,t) = \int_t^{\infty} \tilde{\eta}(x,s) ds,$$ where $\tilde{\eta}$ is the Eulerian variable generated by the optimal stopping time between $\mu$ and $\tilde{\nu}$. In particular $\{\tilde{\eta}>0\}= \{\eta>0\}$. 
 \medskip
 
 To conclude, note that  $\eta \in L^2([0,\infty), H^1_0(\R^n))$, due to the fact that $\eta$ is a subsolution of the heat equation with compactly supported initial data. Hence  taking their zero set as $R$ (where clearly both $\rho$ and $\rho_0$ are supported), we can  apply Lemma~\ref{lem:Eulerian_uniqueness}  to conclude that $\eta = \tilde{\eta}$. 
 \end{proof} 

We finish this section with the corresponding statement for $(St_2)_{\nu}$. The proof is parallel to $(St_1)_{\nu}$.

\begin{theorem}\label{reverse_2}
Let $\mu,\nu$ satisfy \eqref{assumption_measure} {\red and compactly supported}, and let $\eta$ be a weak solution of  $(St_2)_\nu$ and initial data $(\mu, E)$, where $E$ is bounded. Let us define $s$ and $\rho$ by  
$$
s(x):= \inf\{t>0:\eta(x,t)>0\}
$$
 and 
 \eqref{def_rho}. Then $(\eta,\rho)$ is the Eulerian variables between $\mu$ and $\nu$ generated by the optimal stopping time with costs of type (II), where 
 $$
 \mu =\eta_0\hbox{ and } \tilde{\nu} = \nu\chi_{E^C}. 
 $$
 \end{theorem}

In the next section, we will discuss a specific class of the target measures, from which solutions to the classical Stefan problems $(St_1)$ and $(St_2)$ are generated.

\section{Subharmonically generated sets}\label{sec:subharmonic-generating}

This section defines a central notion of the present paper, which characterizes a pair of sets by existence of a certain stopping time. This notion then is connected to solvability of the supercooled Stefan problem $(St_1)$.

\begin{definition}\label{def:subharmonic-set}
We say the pair $(\Sigma, E)$ is {\em subharmonically generated} by $\mu$, if there is a stopping time $\tau$  with the corresponding Eulerian flow $(\eta, \rho)$  \eqref{eqn:EulerianPDE},  such that 
 $\tau>0$ almost surely, $W_0\sim \mu$ and $W_\tau \sim \nu = \chi_\Sigma$ and $E =\{ x \  | \  \eta (x, t) >0 \text{ for some } t>0\}.$  
\end{definition}

This definition is motivated by the following equivalence result.
\begin{theorem}\label{thm:sh-gen-equiv-St1}
Consider a compactly supporeted measure $\mu \ll Leb$ on $\R^d$. 
Then for a given {\red bounded} set $E$, the following are equivalent. 
\begin{itemize}
\item[(a)]There exists a  weak solution $\eta$ of $(St_1)$ with initial data $(\mu, E)$ 
\item[(b)] There exists a measurable set $\Sigma$ such that 
$(\Sigma, E)$ is subharmonically generated by $\mu$.
\end{itemize}
 Moreover,  $\Sigma=\{z(x)<\infty\}$ with $z(x):=\sup\{t >0 \ : \  \eta(x,t)>0\}$.
 \end{theorem}

\begin{proof}

If $(a)$ holds, then $(b)$ follows from Theorem ~\ref{reverse}, with $\Sigma:=\{s(x)<\infty\}$ for $s(x):=\sup\{t >0: \eta(x,t)>0\}$. 
Let us remark that, since the set  $\{\eta(\cdot,t)>0\}$ decreases as $t$ increases, we have
$$
\limsup_{t\to 0^+} \{\eta(\cdot,t)>0\} = \cup_{t>0} \{\eta(\cdot,t)>0\} =  E. 
$$
Similarly, if $(b)$ holds, then $(a)$ follows from Theorem~\ref{thm:consistency}. 
\end{proof}

We mention that obtaining $\nu$ as a characteristic function $\chi_\Sigma$ does not guarantee a corresponding solution for $(St_1)$ unless the stopping time is strictly positive. Even with the stopping time strictly positive, for a given $\mu$ and $E$ there can be many $\Sigma$ such that $(\Sigma, E)$ is subharmonically generated by $\mu$. Below we will discuss an example that illustrates these points. 
{\red In Section~\ref{sec:optimal target} we will present results concerning the optimization problem \eqref{eqn:problem-upper} for the target measure $\nu$. Together with the results in Section~\ref{sec:saturation} this optimization problem uniquely yields a subharmonically generated pair $(E,E)$ for each given $\mu$; see Theorem~\ref{thm:sh-generate}. In Section~\ref{sec:global-Stefan}, we will see  that our optimal target scheme provides a mechanism to construct solutions of $(St_1)$ in a stable manner, by choosing an initial domain that  is strictly larger thatn the support of the initial data and yields an instant regularizing effect for the evolution, at the initial time.}

\subsection{Example for nonuniquenss for $(St_1)$ }\label{sec:nonuniqueness}

The weak solution of $(St_1)$ is known to have non-uniqueness  with given initial data. 
Here we give an example that yields infinitely many weak solutions to $(St_1)$ with the same initial data $(\mu, E)$: these are solutions that do not vanish in finite time.

\begin{proposition}\label{prop:non-uniquenses-St1}
  Let $\mu :=  \frac{1}{|B_\epsilon|} \chi_{B_\epsilon}$ for some $0<\epsilon<1/2$ and
$E=\{|x| \le 1\}$. Then  {\red for a sufficiently small $\epsilon$,} there are infinitely many weak solutions to $(St_1)$ with initial data $(\mu, E)$. 
\end{proposition}
\begin{proof}
Consider the set $\Sigma:= A_1 \cup A_2 \subset \R^d$ consisting of the union of two annuli, 
\begin{align*}
A_1:= \{2\epsilon \leq |x| \leq r\}, \quad A_2:= \{r' \leq |x| \leq 1\},\,\,  \quad \hbox{ where } 0<2\epsilon< r<r'<1.
\end{align*}
Assume $|\Sigma|=1$ by choosing appropriate $\epsilon, r,$ and $  r'$.
Let $\nu= \chi_\Sigma$. 
We will prove
\begin{align*}
\hbox{{\bf Claim:} there exists a randomized stopping time $\tau$, with $W_0 \sim \mu$ and $W_\tau \sim \nu$.}
\end{align*}
After verifying this claim, we can  find the optimal stopping time $\tau^*$ for the optimal Skorokhhod problem $\mathcal{P}(\mu, \nu)$ \eqref{eqn:problem-OSP}. Then, the corresponding Eulearian flow $(\eta, \rho)$ 
as given in Proposition~\ref{thm:stochastic_embedding} gives a solution to $(St_1)_\nu$ via Theorem~\ref{thm:consistency}, with the initial set 
$$E=\limsup_{t\to 0^+} \{\eta(\cdot,t)>0\} = \{|x|\le 1\}.$$ This is because, to reach the outer annulus $A_2$, the Eulerian variable $\eta$ should be active in the whole set  $\{|x| \le 1\}$.  From Remark~\ref{rmk:St-St-nu}, the solution we constructed here for $(St_1)_\nu$, for $\nu =\chi_\Sigma$,  actually solves the original supercooled Stefan $(St_1)$.
Hence, by choosing different combinations of $r, r'$, but still making $|\Sigma|=1$,  we can generate infinitely many solutions to $(St_1)$ with the above $(\mu, E)$.  In this case,    $\{s(x)=\infty\} =  \{ |x|\le 2\epsilon\} \cup \{ r\le |x|\le r'\}$, the outside part of the two annuli in $E=\{ |x|\le 1\}$.

\medskip

It remains to prove  the claim.  We first show that  there exists a randomized stopping time $\tau'$ with $W_0 \sim \delta_{\{x=0\}}$, such that 
$W_{\tau'} \sim \nu$. 
To construct such stoping time, let $\tau_a$ be the first hitting time to $S(a):= \{|x|=a\}$. 
 Then its distribution $\nu_a$ is a uniform $d-1$ dimensional measure along the set $S(a)$, that is, $\nu_a = C_a \delta_{S(a)}$ with some constant $C_a>0$. Thus  we can randomize at $t=0$  to find a stopping time $\tau_A$ with $W_0\sim \delta_{ \{x=0\}}$ whose distribution $\nu_a$ is given by $\nu_A = \int_{r_1}^{r_2} \nu_a f(a) da $ for some weight $f(a)$ that can be controlled by the randomization at $t=0$. In particular, we can find such a randomization at $t=0$ so that  $\nu_A $ becomes a uniform measure on the annulus $A=\{r_1\leq |x| \leq r_2\}$. Now, for such $\tau_{A_1}$ and $\tau_{A_2}$ we can consider the randomized stopping time $\tau'$ given as 
\begin{align*}
\tau' = \begin{cases}
    \tau_{A_1}  & \text{with probability }\frac{|A_1|}{|A_1|+|A_2|}, \\
  \tau_{A_2}     & \text{with probability }\frac{|A_2|}{|A_1|+|A_2|}.
\end{cases}
\end{align*}
Then  it follows that $W_{\tau'} \sim \nu$. Moreover, by the Markov property of the Brownian motion, we can change the randomization at $x=0$ to a randomization at $S(\epsilon)$, that is, when the Browinian motion from the origin first hits $S(\epsilon)$ we make a probabilistic choice how to move from there.

\medskip
 One can find a randomized stopping time $\tau''$ as in the above, with $W_0\sim \delta_{\{x=0\}}$ and $W_{\tau'' } \sim \frac{1}{|B_\epsilon|} \chi_{B_\epsilon}$ for the ball $B_\epsilon$. Note that $\tau' > \tau_{\epsilon}$ almost surely. Clearly $\tau'' \le \tau_\epsilon$, so $\tau'' \le \tau'$. This means that 
for the measure $\mu =  \frac{1}{|B_\epsilon|} \chi_{B_\epsilon}$ there exists a randomized stopping time $\tau$ with $W_0 \sim \mu$ and $W_\tau \sim \nu$, verifying the claim.
\end{proof}

\medskip

\section{Optimal target problem}\label{sec:optimal target}
{\red We now consider} the optimal target problem \eqref{eqn:problem-upper} proposed in our introduction. Namely we solve for the optimal target $\bar{\nu}$ minimizing the cost function $\mathcal{C}(\tau)$ under the upper density constraint 
\begin{equation}\label{eq:f} 
\nu \leq f, \hbox{ where a given function } f\hbox{ satisfies }  0\leq f\in L^{\infty}(\R^d) \hbox{ and } |\{f>0\}| >0.
\end{equation}
One can view this as a projection problem in the space of probability measures: Given $\mu \in P(\R^d)$, what is the closest measure $\nu$ in the constraint set $\{ \nu \le f\}$, under the subharmonic order condition  $\mu \le_{SH} \nu$? Here, the closeness between the two measures is measured by the cost $C(\tau)$ for an optimal stopping time $\tau$ between them.  Once we have the optimal target $\nu$ then we can apply Theorem~\ref{thm:optimal-Skorokhod-existence}, obtaining the optimal stopping time $\tau$ that is given by the hitting time to a barrier given by a barrier function $s$.
\medskip

Let  us point out that it is not {\red obvious how} to explicitly construct the optimal target even in simple cases. For instance  when $f\equiv 1$ and $\mu=\delta_{x=0}$, one may guess that the optimal stopping time $\tau$ is given by the constant time, when the heat kernel $K(t, y)$ becomes $K \le 1$.  This is not the case as we see below,  for the either cost type (I) or (II).

\medskip

For the rest of the paper we will see that the optimal target comes with many interesting characteristics, based on the two main features: monotonicity and saturation property.  These two properties allow us to connect the optimal target problem \eqref{eqn:problem-upper} with global solvability of the Stefan problem; see Section~\ref{sec:global-Stefan}.

\medskip

We begin our discussion with the monotonicity property.

\subsection{Monotonicity}

 The following theorem demonstrates  an ordering principle for optimal targets of Problem  \eqref{eqn:problem-upper} and the corresponding optimal stopping times.

\begin{theorem}[monotonicity]\label{thm:monotonicity}
Let $f$ is as given in \eqref{eq:f}, and 
let the compactly supported measures $(\mu_i, \nu_i)$, $i=1,2$  satisfy \eqref{assumption_measure} (except that they are not necessarily probabiliity measures) with  $\nu_i \le f$, {\red as well as $\mu_i \le_{SH} \nu_i$.} Lastly, let $\tau_i$ be the optimal stopping times with the  initial and target distributions $\mu_i$ and $\nu_i$, as given in Theorem~\ref{thm:optimal-Skorokhod-existence}, with the barrier functions $s_i$.   If  $\mu_1 \leq \mu_2$, then the following holds:
\begin{itemize}
 \item[(1)]  Suppose that $\nu_1$ is a solution to {\red the optimal free target problem} 
\eqref{eqn:problem-upper}
 for $\mu_1$. Then, as the stopping time from $\mu_1$, we have  
\begin{align*}
 \hbox{ $\tau_1 \le \tau_2$ almost surely.}
\end{align*}
  (Notice that from $\mu_1 \le \mu_2$, the stopping time $\tau_2$ from $\mu_2$ can be restricted  to the initial distribution $\mu_1$. ) Here, we do not necessarily assume that $\nu_2$ is a solution to \eqref{eqn:problem-upper}.
 \\
 \item[(2)] Suppose that both $\nu_i$ solve {\red the optimal free target problem} 
\eqref{eqn:problem-upper}
 for $\mu_i$. Then $$\nu_1 \le \nu_2.$$ In particular there is at most one solution to \eqref{eqn:problem-upper}.
\end{itemize}

\end{theorem}
\begin{remark}
As mentioned above, for (1) we only need $\tau_1$ to be optimal {\red for \eqref{eqn:problem-upper}}, not $\tau_2$. This in particular characterizes $\tau_1$ as the smallest stopping time of barrier type that starts from $\mu_1$ among eligible target distributions with the constraint $f$.
\end{remark}

\begin{remark}\label{rmk:tau2-tau1-general}
Note that $\tau_1 \le \tau_2$ should be understood to hold for those Brownian paths starting from the initial points $x$ distributed as the common mass $\mu_1 = \mu_1 \wedge \mu_2$ for $\mu_1 \le \mu_2$. Due to the Markov property of Brownian motion, the stopping time $\tau_1$ can be applied to the initial mass of  $\mu_2$ in the region where $\mu_1 >0$. From the region $\mu_1=0$, it should be understood that the stopping time $\tau_1 =0$  because, from where $\mu_1=0$ there is no motion for $\tau_1$. In this sense, the inequality $\tau_1 \le \tau_2$ can be applied to all Brownian paths from the initial distribution $\mu_2$ not only from $\mu_1$. 
\end{remark}

\begin{proof}[Proof of Theorem~\ref{thm:monotonicity}]
 To show (1), we first prove that $\tau_1 \le \tau_2$, when $\tau_2$ is restricted to the initial distribution $\mu_1$. For this it only needs optimality of $\nu_1$.

 \medskip

{\bf Type (I) case.}  In this case $\tau_i$, $i=1,2$, are the first hitting time to the set $$R_i :=\{ (x, t) \ | \ t \ge s_i (x) \}; \quad \tau_i = \inf\{ t >0\ | \ (W_t, t) \in R_i\}.$$  
Let $\bar \tau :=\tau_1\wedge \tau_2$.  Recall that from Proposition~\ref{lem:hitting_s}  we have 
{\red $\tau_i = s_i (W_{\tau_i})$} %$(W_{\tau_i}, \tau_i) \in R_i$ 
almost surely, $i=1,2$.  {\red Therefore  $\bar \tau = \min[s_1 (W_{\bar \tau}), s_2 (W_{\bar \tau}) ] $ almost surely. }\marginpar{$=$}
Now, for a Brownian path, if $W_{\bar \tau} \in \{ x \ | \ s_1 (x) \le  s_2(x)\}$ then $\bar \tau \ge \tau_1$ so $\bar \tau =\tau_1$. Similarly, if $W_{\bar \tau} \in \{ x \ | \ s_2 (x) \le s_1(x)\}$ then $\bar \tau =\tau_2$. 

Therefore, $\bar  \nu$ denoting the distribution of $\bar \tau$, we have
$$\bar \nu|_{ \{s_1 \le s_2\}} \le \nu_1|_{\{s_1 \le s_2\}} \le f \hbox{ and }  \bar \nu|_{ \{s_2 \le s_1\}} \le \nu_2|_{\{s_2 \le s_1\}} \le f;
$$
 so $\bar \nu \le f$ and $\bar \tau$ is admissible.

Now recall that $\tau_1$ is the minimizer for $\mathcal{C}$, {\red among admissible stopping times for \eqref{eqn:problem-upper}}, which implies $\bar \tau =\tau_1$ almost surely, as otherwise if $Prob [ \bar \tau < \tau_1 ] >0$ then from the strict monotonicity of $\mathcal{C}$, we see $\mathcal{C}(\bar \tau) < \mathcal{C}(\tau_1)$, a contradiction.

This immediately implies that $\tau_1 \le \tau_2$, where we understand  $\tau_2$  as the stopping time restricted to the initial distribution $\mu_1$. 
  
  \medskip

{\bf  Type (II) case.}
{\red In this case, from Theorem~\ref{thm:optimal-Skorokhod-existence}   $\tau_i$, $i=1,2$, are randomized at the initial time to give the distribution $\nu_i\wedge \mu_i$, $i=1,2$, and the rest Brownian particles stop when they first hit  the set $R_i :=\{ (x, t) \ | \ t \le s_i (x) \}$. This gives the corresponding initial and final distribution $\tilde \mu_i =  \mu_i - \nu_i \wedge \mu_i$ $\tilde \nu_i =  \nu_i - \nu_i \wedge \mu_i$ for those Brownian paths with $\tau_i > 0$. Observe that $\tilde \mu_i \wedge \tilde \nu_i =0$. Since $\tau_i$, $\nu_i$'s are optimal to \eqref{eqn:problem-upper}, it follows that $\tilde \nu_i$ is the optimal solution to \eqref{eqn:problem-upper} for the initial $\tilde \mu_i$ and upper bound $\tilde f_i = f - \nu_i \wedge \mu_i$. If $\tilde \mu_i \wedge \tilde f_i \not =0$, then it is possible for $\tau_i$ to stop more mass initially (than $\nu_i \wedge \mu_i$, so increasing $\tilde \mu_i$ but decreasing $\tilde \nu_i$) still satisfying the upper-bound constraint of $f$. The resulting stopping time will have strictly less cost than $\tau_i$ due to strict monotonicity of $\mathcal{C}$. This contradicts optimality of $\tau_i$, and as a result, we see that $\tilde \mu_i \wedge \tilde f_i =0$, namely, $\tilde \mu_i = \mu_i-f \wedge \mu_i $ and $\tilde f_i = f - f\wedge \mu_i$, and it follows that  
\begin{align}\label{eqn:wedge-f-the-same}
\nu_i \wedge \mu_i = f\wedge \mu_i.
\end{align}
Therefore, $\tilde \nu_i = \nu_i - f\wedge \mu_i$, thus $\tilde \nu_i \le  f - f\wedge \mu_i$, $i=1,2$.
  Notice that   $f - f\wedge \mu_2 \le  f - f\wedge \mu_1$ due to $\mu_1 \le \mu_2$.   
}

{\red We now define} a randomized stopping time $\bar \tau$ as follows: 
For the portion $\mu_1\wedge f$ stop immediately, and for the rest $\mu_1 - \mu_1 \wedge f$, follow the Brownian motion until $\tau_1 \wedge \tau_2$. 
Notice that $\bar \tau \le \tau_1$ from its construction. 
 The rest of proof is similar to (I) case with additional consideration for stopping at the initial time. Details follow.

We show that $\bar \tau$ is admissible, namely, if we let $\bar \nu$ be its distribution (with the initial distribution $\mu_1$), $W_{\bar \tau} \sim \bar \nu$, then $\bar \nu \le f$. 
To see this, first write $\bar \nu=\mu_1 \wedge f + \bar \nu^1$ where  $\bar \nu^1$ is the distribution of the rest. 
For Brownian particles with positive time $\bar \tau >0$, that is, those  accounting for $\bar \nu^1$,   if $W_{\bar \tau} \in \{ x \ | \ s_1 (x) \ge  s_2(x)\}$ then $\bar \tau \ge \tau_1$ so  $\bar \tau =\tau_1$. Similarly, if $W_{\bar \tau} \in \{ x \ | \ s_2 (x) \ge s_1(x)\}$ then $\bar \tau \ge \tau_2$ so  $\bar \tau =\tau_2$. Therefore, $$\bar \nu^1|_{ \{s_1 \ge s_2\}} \le (\nu_1 - f\wedge \mu_1)|_{\{s_1 \ge s_2\}}\le (f- f\wedge\mu_1)|_{\{ s_2 \ge s_1\}}$$ as well as  $$\bar \nu^1|_{ \{s_2 \ge s_1\}} \le (\nu_2 - f\wedge \mu_2)|_{\{s_2 \ge s_1\}} \le (f- f\wedge\mu_2)|_{\{ s_2 \ge s_1\}}  \le (f- f\wedge\mu_1)|_{\{ s_2 \ge s_1\}}.$$ From these, we have 
\begin{align*}
 \bar \nu = \mu_1 \wedge f + \bar \nu^1 \le \mu_1 \wedge f +  (f - f\wedge \mu_1) =f , \quad \hbox{ verifying admissibility of $\bar \tau$.}
\end{align*}

Now recall that $\tau_1$ is the minimizer for $\mathcal{C}$, which implies $\bar \tau =\tau_1$ almost surely, as otherwise if $Prob [ \bar \tau < \tau ] >0$ then from the strict monotonicity of $\mathcal{C}$ and from $\bar \tau \le \tau_1$, we see $C(\bar \tau) < C(\tau_1)$ for admissible $\bar \tau$, a contradiction.
This immediately implies that $\tau_1 \le \tau_2$ as desired. 

\medskip

 We now show (2): this requires optimality of both $\nu_1$ and $\nu_2$.

\medskip

Suppose not, i.e. $ \int_{\{ \nu_1 > \nu_2\}} f(x) dx >0$. For a small $\epsilon>0$, define $$E_\epsilon := \{ x \ | \ \nu_1(x) > \nu_2 (x) + 2 \epsilon \nu_1(x)\}.$$
and $\mu_2^{\epsilon} := \mu_2 - \epsilon\mu_1 \geq 0$.
Define an auxiliary randomized stopping time $\tilde \tau$, from the initial distribution $\mu_2$ as follows:
\begin{align*}
 \tilde \tau := \begin{cases}
 \tau_2    & \text{ from $\mu_2^\epsilon$ }, \\
   \tau_1    & \text{ if  $W_{\tau_1} \in E_\epsilon $ from $\epsilon \mu_1$}, \\
\tau_2    & \text{ if $ W_{\tau_1} \not\in E_\epsilon$ from  $\epsilon \mu_1$ }.\end{cases}
\end{align*}
Let us explain the meaning of this. Here the wording `from $\mu_2^\epsilon$' or `from $\epsilon \mu_1$' should be understood as that at the initial time the Brownian particles belong to the distribution $\mu_2^\epsilon$ or $\epsilon \mu_1$, respectively. Notice that such a decomposition in the initial distribution is allowed for the randomized stopping time as it is {\red the same as }  randomizing at the initial time. Notice that we already proved {\red in (1)} that $\tau_1 \le \tau_2$ for Brownian motion starting from $\mu_1$, so the second and third lines in the definition of $\tilde \tau$ is well defined. \\
Now let us show that $\tilde \tau$ is admissible. 
Define the distribution $\tilde \nu$ of $\tilde \tau$, that is, $W_{\tilde \tau} \sim \tilde \nu$. 
From the definition of $\tilde \tau$ we see that 
\begin{align*}
 \tilde \nu |_{E_\epsilon^c} = \hbox{ the distribution of $W_{\tau_2}$ from $\mu_2^\epsilon$ and from $\epsilon \mu_1$, restricted to the set  $E_\epsilon^c$}. 
\end{align*}
Therefore $\tilde \nu |_{E_\epsilon^c}  \le \nu_2$. \\
On the other hand, 
\begin{align*}
 \tilde \nu |_{E_\epsilon} & {\red \le} \hbox{ the distribution of $W_{\tau_1}$ from $\epsilon \mu_1$}\\
  & \qquad \hbox{ $+$  the distribution of $W_{\tau_2}$ from $\mu_2 = \mu_2^\epsilon + \epsilon \mu_1$, restricted to the set  $E_\epsilon$} \\
 &\le \epsilon \nu_1 + \nu_2 \hbox{ on $E_\epsilon$}\\
& \le \nu_1 \hbox{ on $E_\epsilon$ by the definition of $E_\epsilon$}\\
& \le f
\end{align*}
We thus have verified that $\tilde \nu \le f$, therefore, $\tilde \tau$ is admissible as the stopping time from the initial distribution $\mu_2$. \\
Now notice that by the construction $\tilde \tau \le \tau_2$, where the latter is optimal. 
From the strict monotonicity of the cost $C$, we see that $\tilde \tau = \tau_2$. 

\medskip

To see why this leads to a contradiction, notice that it implies that for those Brownian particles starting  from $\epsilon \mu_1$ with $W_{\tau_1}\in E_\epsilon$ we have $\tau_1 =\tilde \tau=\tau_2$. Since $\epsilon \nu_1$ is the final distribution for $\tau_1$ starting from $\epsilon \mu_1$  we have  
\begin{align*}
 \epsilon \nu_1 | _{E_\epsilon} & \le \hbox{ the distribution for $\tau_2$ starting from $\epsilon \mu_1$, restricted to $E_\epsilon$.}\\
 & \le \hbox{ the distribution for $\tau_2$ starting from $\epsilon \mu_2$, restricted to $E_\epsilon$}  \hbox{ (because $\mu_1 \le \mu_2$)}\\
 & = \epsilon \nu_2|_{E_\epsilon}.
\end{align*}
Note that this is a contradiction to the definition of $E_\epsilon$, thus proves the claim that $\nu_1 \le \nu_2$. 
\end{proof}

The rest of the section discusses important consequences of the monotonicity property.

\subsection{Well-posedness}

\begin{theorem}\label{thm:existence} 
Let $f$ be as given in \eqref{eq:f}, and let $\mu\ll Leb$ have bounded, compactly supported density.
Suppose for a given $0\le f\in L^{\infty}(\R^d)$, the set 
$$
\mathcal{A} := \{ \nu: \mu \le_{SH} \nu, \quad \nu \hbox{ is compactly supported and }\nu \leq f\} \hbox{ is nonempty.}% \neq \emptyset
$$
Then \eqref{eqn:problem-upper} yields a unique optimal target measure $\nu^* =\nu(\mu, f, \mathcal{C})$, which depends only on $\mu, f$ and the cost $\mathcal{C}$.
\end{theorem} 

\begin{proof}

Once existence is established, uniqueness is a direct consequence of Theorem~\ref{thm:monotonicity}.  

\medskip

To show existence, let us first consider the case where $f$ is compactly supported. {\red  In this case $\nu \in \mathcal{A}$ is compactly supported, and for any Brownian stopping time $\tau$ with $W_0 \sim \mu$,$W_\tau \sim \nu$, with finite $\mathbb{E}[\tau]$,  we see that  for any $t$, the distribution $\nu_t$ of  $W_{t \wedge \tau}$, that is, $W_{t \wedge \tau} \sim \nu_t$, has support inside the convex hull of the support of $\nu$. This follows because $\tau\wedge t \le \tau$ so $\nu_t \le_{SH} \nu$, and for the convex order the smaller measure has support inside the convex hull of the support of the bigger (in convex order) measure. }
{\red This implies that} the relevant domain both for the measures $\nu \in \mathcal{A}$ and  the Brownian motion (with stopping time) between $\mu$ and $\nu$ is compact, in particular, it is a subset of the convex hull of ${\rm supp\,} f$. 
 Therefore existence of a minimizer $\nu$ follows easily from the weak-* compactness of the minimizing sequence, since that the condition $\nu \le f$ is closed with the weak-* convergence. 

\medskip

Now, {\red consider the general case, and let $f_R = f \wedge 1_{B_R}$.  Then,  for $R \gg 1$, the corresponding admissibility set $\mathcal{A_R}$ with $f_R$ is nonempty. 
  (It is because we assumed $\mathcal{A}$ is nonempty, so there exists a compactly supported admissible target $\nu \in \mathcal{A}$. ) }
 Therefore, for the initial measure $\mu$, and for each $1\ll R < R'$, applying the compactly supported case, we find  $\nu_R, \nu_{R'}$ be the optimal solutions with the constraint $f_R$ and $f_{R'}$, respectively. Let $\tau_R, \tau_{R'}$ be the corresponding optimal stopping times. 
 Notice that $\nu_{R}$  also satisfies the density constraint $f_{R'}$ ($\ge f_R$). We can then apply the monotonicity, Theorem~\ref{thm:monotonicity}(1), and get 
$\tau_{R'} \le \tau_R$. It implies that the support of $\nu_{R'}$ is contained in the convex hull of the support of $\nu_R$, {\red in particular, in the ball $B_R$. In the ball $B_R$} %in the support of $\nu_R$
 we have $f_R = f_{R'}$, so, $\nu_{R'}$ should be an optimal solution for the constraint $f_R$, because for the cost we have $\mathcal{C}(\tau_{R'}) \le \mathcal{C}(\tau_R)$. From uniqueness of optmal solution in the compactly supported constraint case, we have $\nu_R =\nu_{R'}$. This implies that the optimal $\nu_R$ is independent of $R$ as long as $R$ is sufficiently large. That $\nu_R$, $R\gg1$, is the optimal target for $f$. To see this notice that for any compactly supported  target measure $\nu' \le f$, there exists $R>0$ such that  ${\rm supp\,} \nu' \subset B_R$, so $\nu' \le f_R$ and the cost for $\nu_R$ is less than or equal to that of $\nu'$.
 \end{proof} 
 
\begin{remark}
 We note that those $f$ with $\mathcal{A}\not =\emptyset$ are plenty. For example, $f\equiv 1$ on $\R^d$, or  any $f \ge 0$ which has a positive lower bound on a ball $B_R$, $R\gg 1$, will work. Or any $ f\ge 0$ that has a positive lower bound in the annulus $\{ R_1 \le |x| \le R_2\},$ with $1\ll R_2 < R_2\le \infty$ works.
 Of course, here how large $R, R_1, R_2$ depends on $\mu$.  These are some particular types of $f$ we focus on in this paper, especially in Section~\ref{sec:global-Stefan}.
\end{remark}

 \subsection{Universality of the optimal target} 
 {\red We prove in this section a remarkable consequence of the monotonicity property, namely, that} for a given initial measure $\mu$ and the constraint $f$, the optimal target $\nu$ is unique independent of the choice of cost function $\mathcal{C}$, as long as it is either type (I) or  (II).  For the same type of costs {\red we have} Lemma~\ref{lem:unique-R}, which shows a universality of optimal stopping time for a fixed target. What is surprising here  is that the universality holds even across the different types for which the corresponding optimal stopping times behave very differently. For its proof we exploit the fact that Theorem \ref{thm:monotonicity}(1) requires optimality only for one of the targets. This universality  result will have interesting consequences later in the context of Stefan problem (see Section~\ref{sec:global-Stefan}).
\begin{theorem}[Universality]\label{thm:universality}
Let $\mu$ and $f$ be as given in Theorem~\ref{thm:existence}. Let us consider {\red the optimal solutions of \eqref{eqn:problem-upper},} $\nu_i: = \nu (\mu, f, \mathcal{C}_i)$, $i=1,2$, for the cost functions $\mathcal{C}_i$ of either type (I) or (II), same or different types.
 Then $\nu_1 = \nu_2$.
  \end{theorem}
\begin{proof}
 Let $\tau_i$, $i=1,2$ are the corresponding optimal stopping times, for $\mathcal{C}_i$, $i=1,2$, respectively. 
  For the cost $\mathcal{C}_1$ consider the optimal stopping time $\tau_2'$ for the given target $\nu_2$. Theorem \ref{thm:monotonicity}(1) and optimality of $\tau_1$ now yields 
$ \tau_1 \le \tau_2'$.
This implies that 
$\nu_1 \le_{SH} \nu_2$. 
Similarly, for the cost $\mathcal{C}_2$, 
consider the optimal stopping time $\tau_1'$ for the given target $\nu_1$. Theorem \ref{thm:monotonicity}(1) and optimality of $\tau_2$ now yields
$ \tau_2 \le \tau_1'$. This implies that 
$\nu_2 \le_{SH} \nu_1$. 
The two subharmonic orders imply  $\nu_2 =\nu_1 $, completing the proof.
\end{proof}

 {\redf
\begin{remark}\label{rmk:shadow}
In fact, what we have shown in Theorem~\ref{thm:universality} using the monotonicity in Theorem \ref{thm:monotonicity}(1), is the following: For our solution $\nu^*=\nu(\mu, f, \mathcal{C})$ of Problem~\eqref{eqn:problem-upper} (for any cost $\mathcal{C}$ of either type (I) or (II)), for a measure $\tilde \nu$, 
\begin{align*}
 \hbox{if $\tilde \nu \le f$ and $\mu \le_{SH} \tilde \nu$ then $\nu^* \le_{SH} \tilde\nu$.}
\end{align*}
So, our solution $\nu^*=\nu(\mu, f, \mathcal{C})$ coincides with the notion of  the shadow  $S^f(\mu)$ of the measure $\mu$ in $f$ as introduced  in \cite{Rost71, beiglboeck-juillet2016} (see \cite[Lemma 4.6]{beiglboeck-juillet2016}); such a notion was first considered in \cite{Rost71} for general cases, and was named {\em shadow} in \cite{beiglboeck-juillet2016} in the $1$-dimensional case. It is then  further considered in a recent paper \cite{bruckerhoff2021shadows}.  We thank the anonymous referee for drawing our attention to this connection. \end{remark}
}

\subsection{$L^1$ contraction and $BV$ estimate} 

Next we show the $L^1$ contraction, which is a consequence of the monotonicity and the fact that  the total mass of $\mu$ is the  same as that of $\nu$.

\begin{theorem}[$L^1$-Contraction]\label{thm:L1-contraction}
Let $\mu_i$, $i=1,2$, and $f$ be as given in Theorem~\ref{thm:existence}, and consider {\red the optimal solutions of \eqref{eqn:problem-upper},} $\nu_i := \nu(\mu_i, f, \mathcal{C})$ with a cost $\mathcal{C}$ of either type (I) or (II).  Then
\begin{align*}
 \| (\nu_1 - \nu_2)_+ \|_{L^1} \le  \| (\mu_1- \mu_2)_+ \|_{L^1}.
\end{align*}
\end{theorem}
\begin{proof}
Let $\tilde \mu = \mu_1 \wedge \mu_2$; notice that $(\mu_1-\mu_2)_+ = \mu_1 - \tilde \mu$. Let $\tilde \nu$ be an optimal solution of the corresponding problem \eqref{eqn:problem-upper} 
 with the initial distribution $\tilde \mu$. As $\tilde \mu \le \mu_i$, $i=1,2$, we have from the monotonicity (Theorem~\ref{thm:monotonicity}) that 
$$\hbox{$\tilde \nu \le \nu_i$, $i=1,2$.} $$
Now, let $E_+ = \{ x \ |  \  \nu_1(x)- \nu_2 (x) \ge 0\}$. Then, 
\begin{align*}
 (\nu_1 - \nu_2)_+ = (\nu_1 - \nu_2) \chi_{E_+} \le (\nu_1 - \tilde \nu ) \chi_{E_+} \le \nu_1 - \tilde \nu.
\end{align*}
Therefore, 
\begin{align*}
 \| (\nu_1 - \nu_2)_+ \|_{L^1} \le \| \nu_1 - \tilde \nu\|_{L^1} = \| \mu_1 - \tilde \mu \|_{L^1} = \| (\mu_1-\mu_2)_+\|_{L^1}
\end{align*} 
as desired. Notice that the first equality follows from the fact that the problem preserves the total mass, i.e. $\| \mu_1\|_{L^1} = \|\nu_1\|_{L^1}$ , $\| \tilde  \mu\|_{L^1} = \|\tilde \nu\|_{L^1}$. 
\end{proof}

The $BV$ estimate follows as an immediate corollary of the $L^1$ contraction: such consequence is well-known {\red  for problems that are translation invariant; it is the case for \eqref{eqn:problem-upper}  if $f\equiv 1$.}  
For $(St_2)$ {\red that corresponds to type (II),  such a result} is shown by Meirmanov \cite{meirmanov}. 

\begin{theorem}[BV estimate]\label{thm:BV} 
  Let $\nu$ be the solution of \eqref{eqn:problem-upper} {\red  as given in Theorem~\ref{thm:existence}  with $\mu$ and $f\equiv 1$.} Then, 
\begin{align*}
 \| \nu\|_{BV} \le \| \mu\|_{BV}.
\end{align*}
\end{theorem}

\begin{proof}
Let us first consider the case $\mu \in C^1$. 
Let us define $\nu^{\delta}$  to be a smooth approximation to $\nu$ by making convolution with a mollifier supported in a small ball $B_\delta$.
Since $\nu$ is $L^1$ and $\nu$ is compactly supported, $\nu^{\delta}$ uniformly converges to $\nu$ in $L^1$. From lower-semi-continuity of $BV$ norm under $L^1$ convergence  it is enough to prove that 
\begin{equation*}
\|\nu^{\delta}\|_{BV} \leq \|\mu \|_{BV} \hbox{ for small } \delta>0.
\end{equation*}

Observe that for each $y \in B_1$, $\ep >0$, 
we have 
\begin{align*}
 \int |\nu^{\delta}(x+\ep y) - \nu^{\delta}(x) |dx\le \int|\nu(w+\ep y) - \nu(w)| dw.
\end{align*}

Notice that $f\equiv 1$ is translation invariant, therefore from the uniqueness result,  $\nu( \cdot + \ep y)$ is the optimal solution of \eqref{eqn:problem-upper} for the initial measure $\mu(\cdot + \ep y)$. 
We apply then the $L^1$ contraction (Theorem~\ref{thm:L1-contraction}) and get for any $\ep>0$, 
$$
\int \dfrac{ |\nu^{\delta}(x+\ep y) - \nu^{\delta}(x)|}{\ep} dx \leq \int \dfrac{|\mu(x+\ep y) - \mu(x)|}{\ep} dx.
$$
Now for a fixed $\delta$, $\nu^{\delta}$ and $\mu$ are $C^1$ with compact supports, and so by sending $\ep\to 0$ in above expression we obtain
$$
\int |D\nu^{\delta}(x)  \cdot y| dx \leq \int |D\mu(x) \cdot y |dx.
$$
We then integrate both sides in $y \in B_1$, 
\begin{align*}
\int_{B_1} \int |D\nu^{\delta} (x) \cdot y| dx dy \leq \int_{B_1} \int |D\mu(x) \cdot y |dy
\end{align*}
and apply the Fubini's theorem, 
\begin{align*}
\int \int_{B_1}  |D\nu^{\delta}(x) \cdot y|  dydx  \leq  \int \int_{B_1}  |D\mu(x)\cdot y | dydx .
\end{align*}
Observe that for any unit vector $e_1$, we have 
\begin{align*}
 \int_{B_1} |D\nu^{\delta}(x)  \cdot y|dy = |D\nu^\delta(x) | \int_{B_1} |e_1  \cdot y| dy 
\end{align*}
and similarly for $\mu$, which leads to 
\begin{align*}
 \int |D\nu^\delta(x) | \left(\int_{B_1} |e_1  \cdot y| dy  \right) dx  \leq  \int   |D\mu(x) |  \left(\int_{B_1} |e_1  \cdot y| dy \right)dx .
\end{align*}
Therefore, we get the inequality 
\begin{align*}
  \int |D\nu^{\delta}(x) |dx \leq  \int |D\mu(x)| dx .
\end{align*}
Hence we showed that
$
\| \nu^{\delta}\|_{BV} \leq \|\mu\|_{BV}
$
as desired. 

\medskip

Let us now consider the general $\mu$ in  $BV$. We consider a sequence of measures $\mu_k \in C^1$ that converges as $k\to \infty$, to $\mu$ in $BV$ (also in $L^1$); in particular $\|\mu_k\|_{BV} \to \|\mu\|_{BV}$. Consider the corresponding $\nu_k$, namely, the solution of \eqref{eqn:problem-upper} with {\red the}  initial $\mu_k$.  
Apply {\red the previous step}, and get
\begin{align*}
 \| \nu_k\|_{BV} \le \| \mu_k\|_{BV}.
\end{align*}
Now notice that by the $L^1$ contraction $\nu_k$ converges, as $k\to\infty$, to $\nu$ in $L^1$, therefore, 
from the lower-semicontinuity of the BV-norm in $L^1$ convergence, we get 
\begin{align*}
 \| \nu\|_{BV} \le \|\mu\|_{BV}.
\end{align*}
This completes the proof.
\end{proof}

\section{Saturation property of the optimal target}\label{sec:saturation}

Here we prove that the  optimal target  measure  {\red of  \eqref{eqn:problem-upper},} when obtained with positive stopping time, saturates up to  the upper density  limit $f$.
We first show a converse to {\red Corollary \ref{cor:mass_below_s}.}

\begin{lemma}\label{lem:nu_below_W_t} 
Let $\mu$,$f$ and $\nu=\nu^*$  be as given in Theorem~\ref{thm:existence}. Let $\tau$ be the optimal stopping time with $W_0\sim \mu$ and $W_\tau \sim \nu$. Let  $G$ be a measurable set such that $f|_G >0$. Suppose that there exists a constant $t_1>0$ such that
\begin{align*}
 \hbox{   $Prob[W_{t_1} \in G \ \& \ t_1 < \tau] >0$.}
\end{align*}
Then, $\nu[G ]>0$. 
\end{lemma} 
\begin{proof}
Suppose for contradiction, that $\nu|_G=0$. 
We define a randomized stopping time $\bar \tau$ (starting from the distribution $\mu$) as follows. 
\begin{itemize}
 \item For those Brownian trajectories with $\tau \le t_1$: {\red stop at $\tau$.} 
 \item For those Brownian trajectories with $\tau > t_1$: If $W_{t_1} \not \in G$ just proceed until $\tau$. But, if $W_{t_1} \in G$  then drop a portion of mass at the time $t_1$, such that the resulting mass is  positive but has density $\le f$ over the set $G$, 
  and then proceed until $\tau$; note that this is possible because $Prob[W_{t_1} \in G \ \& \ t_1 < \tau] >0$ as well as $f|_G >0$ {\red for the function $ f \in L^\infty$.} %$f\ll Leb$.
 
\begin{itemize}
 \item To be more precise on this second point, notice that the 
density at $x$ of the distribution of $W_{t_1}$ for those Brownian particles with $t_1 < \tau$, is bounded from above by $K_{t_1}(x)$ of  the solution to the heat equation with initial data $\mu$, thus at each $x \in G$  when $W_{t_1}=x \in G$, one can stop the mass with the probability 
{\red $$\frac{\min[f(x), K_{t_1}(x)]}{K_{t_1}(x)}>0,$$}
 and  we get the positive  stopped density  $\le f(x)$.
\end{itemize}
\item {\red Note that because $\nu|_G=0$, the resulting distribution of $W_{\bar \tau}$ on $G$ is only added from the second point above, thus is $\le f$. On other parts, the distribution is $\le$ the distribution of $W_\tau$ by the construction.} 
\end{itemize}
  Therefore this randomized stopping time $\bar \tau$ has distribution $W_{\bar \tau}\sim \bar \nu$ with $\bar \nu \le f$. Moreover, notice that 
 $\bar \tau \le \tau$ for each random path almost surely, and $Prob[\bar \tau < \tau]>0$ {\red because of the second point above}. Then, the cost $\mathcal{C}(\bar \tau) < \mathcal{C}(\tau)$ from the strictness of the cost functional $\mathcal{C}$. This contradicts optimality of $\tau$, thus completing the proof. 
\end{proof}
 As an immediate corollary of this lemma and Lemma~\ref{lem:tau_le_tauG}, we see that for the optimally stopped Brownian motion for our problem  \eqref{eqn:problem-upper}, 
any set $G$ with zero mass of the final distribution blocks off the Brownian motion. 
 \begin{corollary}\label{cor:no-motion-no-mass-barrier} 
 Let $\mu, \nu$ and $\tau$ be as given in Lemma~\ref{lem:nu_below_W_t}.
 Let $G$ is a measurable set with positive measure, and suppose that every point of $G$ has positive Lebesgue density. 
Assume that  $f|_G >0$ and   $\nu[G ]=0$.
 Then, $\tau \le \tau_G$ almost surely.
\end{corollary}
\begin{proof}
 This is a direct consequence of Lemma~\ref{lem:nu_below_W_t} and  Lemma~\ref{lem:tau_le_tauG}. 
\end{proof}

We can combine ideas used in the proof of Lemma~\ref{lem:nu_below_W_t} and the result of Corollary~\ref{cor:mass_below_s} {\red (see also Remark~\ref{rmk:2-cor:mass_below_s})} to prove 
  that the optimal solution $\nu$ saturates the density upper bound. 
\begin{theorem}[saturation]\label{thm:saturation}
{\red Let $\mu$,$f$ and $\nu=\nu^*$  be as given in Theorem~\ref{thm:existence}, and $\tau$ be the corresponding optimal stopping time. }
If $\tau>0$ almost surely, 
then $\nu$ is given in the form
 \begin{align*}
 \nu = f \chi_E \quad \hbox{ for some set $E$}.
\end{align*}

More generally, without assuming $\tau>0$, we have the following results:
\begin{itemize}
\item[(a)] For costs of type (I), we have   
\begin{align*}
 \nu = f \chi_{E}   + \mu|_{F} \quad \hbox{ for some set $E$ and $F$  with $|E\cap F | =0$.}
\end{align*}
where in $E$ the Brownian motion does not stop immediately, i.e. $s(x) >0$ for a.e. $x \in E$, and in  $F$ the Brownian paths stop immediately, i.e. $s(x) =0$, for a.e. $x \in F$.  \\ \\

 \item[(b)] For costs of type (II),  we have
\begin{align*}
 \nu =  \tilde \nu + f \wedge \mu
\end{align*}
where $\tilde \nu$ is the optimal solution from the initial measure  $\tilde \mu = \mu -  f\wedge \mu$, while the upper bound constraint is given by   $\tilde f = f - f \wedge \mu$; here $\tilde \nu$ is given in the form
\begin{align*}
\tilde \nu = \tilde f \chi_E \quad \hbox{ for some set $E$} 
\end{align*}
 and the corresponding optimal stopping time $\tilde \tau$ is positive almost surely. Moreover 
  $$
 (f\wedge \mu)(x) = Prob[W_0 = x, \tau=0]\mu(x).
 $$
\end{itemize}
\end{theorem}
\begin{remark}\label{rmk:nu-less-f}
For costs of type (II),   Theorem~\ref{thm:saturation} (b) implies that the set $\{  \nu < f \}$  is nothing but where $\nu=f\mu \wedge f$, equivalently $\tilde \nu=0$, and it  belongs to the set where $\tau=0$. \end{remark}

Before proceeding to the proof of above theorem, we state the following characterization of the instantly stopped portion for each cost types  
using the potential flow.  
\begin{corollary}\label{cor:stoppingtime_zero} 
For costs of type  (I) {\red and (II)},  and $E$ and $F$ as given in Theorem~\ref{thm:saturation} we have  
\begin{align*}
\hbox{$E=\{w>0\}$ and $F=\{w=0\}$},
\end{align*}
 where $w\geq 0$ is {\red the} continuous solution of the obstacle problem 
 \begin{equation}\label{obstacle}
-\Delta w +(f-\mu)\chi_{\{w>0\}}= 0.
\end{equation}
\end{corollary}

\begin{proof}
 {\red First, for type (I), set} $w := U_\nu - U_\mu$ . Then {\red from Theorem~\ref{thm:saturation}(a)},  $E=\{ w >0\}$ and $F=\{ w=0\}$.  To see this, for example apply Lemma~\ref{lem:Ito-closed-set} and continuity of $U_\nu$ and $U_\mu$.
  Therefore, we see that 
\begin{align*}
\Delta w = \nu - \mu = f\chi_{\{ w >0\}} + \mu \chi_{\{ w =0\}} - \mu =  (f-\mu)\chi_{\{ w >0\}}.
\end{align*}

{\red In the type (II) case, set $w := U_{\tilde \nu} - U_{\tilde \mu}$. with  $\tilde \mu = \mu -  f\wedge \mu$ and  $\tilde \nu = \nu -  f\wedge \mu$.  Then {\red from Theorem~\ref{thm:saturation}(b)},  $E=\{ w >0\}$ and $supp \, \tilde \mu \subset \{  w>0\}$; this is due to the fact that $\tilde \tau >0.$   Therefore, we see that 
\begin{align*}
\Delta w = \tilde \nu - \tilde \mu = \tilde f\chi_{\{ w >0\}} - \tilde \mu \chi_{\{ w>0\}}=  (f-f \wedge \mu -\mu + f \wedge\mu)\chi_{\{ w >0\}}=(f-\mu)\chi_{\{ w >0\}}.
\end{align*}
This completes the proof. }
\end{proof}

As in this corollary, {\red the set $E$ in  Theorem~\ref{thm:saturation} where the target density $\nu$ saturates the constraint $f$,}  is determined by the solution $w$ of \eqref{obstacle}, 
and, it is contained in the convex hull of the support $\nu$. Therefore, it is compact, because the optimal target measure $\nu$ is compactly supported, in particular in a large ball $B_R$ centred at the origin. 
Now, 
$w$ is characterized as the unique solution of the well-known elliptic obstacle problem. Namely $w$ minimizes
\begin{equation}\label{variational}
\int_{B_R} (|D\varphi|^2 + 2(f-\mu) \varphi) dx
\end{equation}
 among nonnegative functions $\varphi \in H^1_0(B_R),$  
 see \cite{blank}, \cite{caffarelli98}.
  The regularity property of the free boundary $\partial E=\partial\{w>0\}$ has been actively studied in the literature, which sensitively depends on the regularity of $f$; for instance, when $f$ is at least H\"{o}lder continuous ($C^{\alpha}$), $\partial E$ is $C^{1,\alpha}$ except at singular points with explicit and unique asymptotic blow-up profiles. In dimensions $d=2,3$, it is expected that $\partial E$ has no singular points when $f$ is sufficiently smooth for a generic choice of $f$, as suggested by the particular discussion in \cite{figalli20generic}. Thus {\red in type (I) case} the set $E$ exhibits an initial expansion of the support of $\mu$ to a larger and (mostly) smooth set; it suggests an instant regularization of the initial data which helps in our case, establishing a well-posed evolution for $(St_1)$, a problem that is well known to be  ill-posed in general.  For computational methods to obtain $E$, through the obstacle problem, see for instance \cite{tran15} and the references therein.

\noindent \begin{proof}[\bf Proof of Theorem~\ref{thm:saturation}]
We proceed  in several steps.

\medskip

1. {\it When $\tau >0$ a.s.}
\,\, For $0<\delta<1$ define $$H_\delta:=\{ x \ | \ \delta \le \frac{d\nu(x)}{dx} < f(x) -\delta\}.$$ We claim that $\nu [H_\delta] =0$ {\red for each sufficiently small $\delta$.} Suppose $\nu[H_\delta]>0$.  
Due to Corollary \ref{cor:mass_below_s} applied to $G= H_{\delta}$,
either
\begin{itemize}
 \item[(i)] $|\mu \wedge (\nu |_{H_\delta}) | >0$, or
 \item[(ii)]
there exists a constant $\bar t>0$ such that 
\begin{align*}
 Prob[W_{\bar t} \in H_\delta \ \& \ 
 \bar t < \tau] >0. 
\end{align*}
\end{itemize}
We will show that neither case is possible.

\medskip

If (i) holds, then first observe that, from the fact that $H_{\delta}$ increases as $\delta$ decreases, we have the set $S:= \{x:\frac{d\mu(x)}{dx} \geq \delta\} \cap \nu |_{H_{\delta}}$ has positive measure for sufficiently small $\delta>0$. We then define a cost-decreasing randomized stopping time $\tilde \tau$  (starting from the distribution $\mu$) as follows. First, $\tilde \tau$ is randomized at the initial time for  those Brownian particles from $S$, in such a way that the distribution of initially stopped particles from  $S$ has positive mass with density $\delta$. For the remaining particles it follows $\tau$. Then the resulting terminal  distribution by $\tilde \tau$, say $\tilde \nu$, has density less than  $f$ {\red on $S$,} by the definition of the set $H_\delta$.  {\red Outside $S$, $\tilde \nu \le \nu$, therefore  overall $\tilde \nu \le f$.}  Moreover, from the construction of $\tilde \tau$ and the strict monotonicity of $\mathcal{C}$ we have $\mathcal{C}(\tilde \tau) < \mathcal{C}(\tau)$. This contradicts optimality of $\tau$, and thus (i)  does not happen.

\medskip

 For (ii), we define $\tilde \tau$ in such a way that  for those Brownian trajectories with $\tau <\bar t$,  stop at $\tau$.  For those Brownian trajectories with $\tau > \bar t$, if $W_{\bar t} \not \in H_\delta$ just proceed until $\tau$, but, if $W_{\bar t} \in H_\delta$  drop a portion of mass at time $\bar t$, such that the resulting mass is  positive but has density $\le \delta$ over the set $H_\delta$, then proceed until $\tau$; note that this is possible because the density at $x$ of the distribution of $W_{\bar t}$ is bounded from above by $K_{\bar t}(x)$ of  the solution to the heat equation with initial data $\mu$, thus at each $x \in H_\delta$  when $W_{\bar t}=x \in H_\delta$, one can stop the mass with the probability $$\frac{\min[\delta, K_{\bar t}(x)]}{K_{\bar t}(x)}>0,$$ and  we get the positive  stopped density  with the total mass $\le \delta$. Therefore this randomized stopping time $\tilde \tau$ has distribution $\tilde \nu$ with $\tilde\nu \le f$. Moreover, notice that 
 $\tilde \tau \le \tau$ for each random path almost surely, and $Prob[\tilde \tau < \tau]>0$. Then, the cost $\mathcal{C}(\tilde \tau) < \mathcal{C}(\tau)$ from the strictness of the cost functional $\mathcal{C}$. This contradicts optimality of $\tau$, thus completing the proof for the case  `$\tau >0$ almost surely':  we verified that $\nu = f\chi_E$ for some set $E$ for both (I) and (II).

\medskip

    {\red We now consider the general case.}  
    
    \medskip

2. (I) case.  Recall that in this case $\tau =\inf \{ t >0 \ | \ t\ge s(W_t)\}$ {\red (see Theorem~\ref{thm:optimal-Skorokhod-existence} and \eqref{eqn:prelim-optimal-barrier})}.   Notice that the set $F:=\{ x \ | \ s(x) =0\}$ gives a barrier set for $W_t$ , $t\le \tau$. That is, any Brownian path starting from a point in $F$ with the stopping time $\tau$ stops immediately, equivalently, $\tau \le \tau_F$. As a consequence the resulting final distribution with $\tau$ starting from $\mu|_F$ is $\mu|_F$. On the other hand, 
denote $\bar \mu := \mu|_{F^c}$ and apply the stopping time $\tau$ for those Brownian particles starting from $\bar \mu$. We let $\bar \tau$ denote such a restriction of $\tau$. Notice that $\bar \tau>0$ almost surely {\red because the set $F^c =\{ s >0\} =R^c =(R^U)^c$ is an open set  (see Theorem~\ref{thm:tau-equal-tau-U}). Moreover $\bar \tau$} is the optimal stopping time for $\bar \mu$ and  $\bar f=f-\mu_F$;  {\red otherwise, $\tau$ will not be optimal}.  Let $\bar \nu$ be the resulting final distribution from the stopping time $\bar \tau$. Then from {\red Step 1,} we have that
\begin{align*}
 \bar \nu = \bar f \chi_E \quad \hbox{ for some set $E$ where $\bar f |_E >0$}.
\end{align*}
Moreover, notice that $\bar \mu \wedge (\bar \nu|_F ) =0$ from the definition  $\bar \mu = \mu|_{F^c}$. Therefore, from $\tau \le \tau_F$ (since $F$ is a part of the barrier)  and Lemma 
~\ref{lem:zero_measure_G} we get $\bar \nu [F] =0$.
This also implies that $ |E \cap F| =0$, therefore, $\bar f \chi_E = f \chi_E$.  In summary, we have 
\begin{align*}
 \nu= \bar \nu + \mu|_F= f \chi_E + \mu|_F \quad \hbox{as desired.}
\end{align*}
This completes the proof for (I) {\red  in part (a).}

\medskip

3. (II) case. 
In this case the stopping time $\tau$, which is the optimal stopping time between $\mu$ and $\nu$,  is randomized at $t=0$ to give $f \wedge \mu$ and proceeds with positive time starting from the rest {\red $\tilde \mu=\mu -\mu\wedge \nu= \mu - f\wedge\mu$,} which gives the additional  final distribution {\red $\tilde \nu = \nu -\mu\wedge \nu=  \nu -f\wedge \mu$}; {\red for explanation (in particular why $\mu\wedge \nu = f\wedge \mu$) see the discussion around \eqref{eqn:wedge-f-the-same}.}
 Let $\tilde \tau$ denote the stoping time $\tau$ restricted to  the initial distribution $\tilde \mu$. Then its  final distribution $\tilde \nu$ gives an optimal solution to \eqref{eqn:problem-upper} with the upper bound constraint $\tilde f = f-f\wedge \mu$. Notice that $\tilde \tau  >0$ almost surely {\red as $\tilde \mu \wedge \tilde \nu =0$.} Thus from Case 1, we have $\tilde \nu = \tilde f \chi_E$ for a set $E$. 
Notice that $\nu = \tilde \nu + f \wedge \mu$. All things combined we verified the result {\red for type (II) in part (b).}
This completes the proof. 
\end{proof}

\subsection{Upper bound for the optimal stopping time in the (I) case}
 One of the applications of the saturation result is the following:
\begin{theorem}\label{thm:time-upper}
{\red Let $\mu$,$f$ and $\nu=\nu^*$  be as given in Theorem~\ref{thm:existence}, and $\tau$ be the corresponding optimal stopping time. }
In addition assume that  $f \ge \delta$ for a constant $\delta >0$. Then for costs of type (I), 
 there is a constant $\bar T=\bar T(\delta,\|\mu\|_{L^1})$ such that
\begin{align*}
  \tau \le \bar T \hbox{ almost surely.}
\end{align*}
\begin{remark}
 It is interesting that the  extinction time $\bar T$ does not  depend on the size of the support of $\mu$.
\end{remark}

\end{theorem}
\begin{proof} 
Recall that $\tau$ has a barrier function $s$ such that 
$\tau = \inf \{ t >0 | \ t \ge s(W_t)\}$. 
Consider the optimal target $\nu$ and the set $E$ and $F$ from Theorem~\ref{thm:saturation} (a) (so in the (I) case) which satisfies
 that {\red for each random path,} 
  \begin{align*}
& \hbox{`$\tau >0$' implies `$W_\tau \in E$' almost surely, and }\\ 
& \hbox{`$\tau=0$' implies `$W_0 \in F$' almost surely, }
\end{align*}
and
 that  for any measurable set $S$,
\begin{align*}
 Prob[W_\tau \in S \ \& \ \tau >0 \ | \ W_0 \sim \mu]= \nu[S \cap E] = \int_{S \cap E} df \ \ge \delta |S\cap E|
\end{align*}
where the last inequality is from 
 the assumption that $f \ge \delta$. Since $\nu \ll Leb$, we may assume, without loss of generality, that $E^c$ consists of its Lebesgue points, by adding the non-Lebesgue density portion of $E^c$, which has zero Lebesgue measure, to $E$.
 Then, from Corollary \ref{cor:no-motion-no-mass-barrier},  we see that
\begin{align}\label{eqn:stay-in-E} 
\hbox{ if $\tau >0$, then almost surely $\tau \le \tau_{E^c}$.}
\end{align}
Now consider for each $T>0$,  the set 
\begin{align*}
 Z_T :=\{ x \ | \ s(x) > T\}.
\end{align*}
Notice that $\nu[Z_T] = \nu [Z_T \cap E].$

It suffices to show that there exists $\bar T$ such that $\nu[Z_{\bar T}]=0$ {\red because from Proposition~\ref{lem:hitting_s} $\tau = s(\tau)$ almost surely.}

{\red Now Proposition~\ref{lem:hitting_s} implies,} almost surely, 
$ W_\tau \in Z_T $  if and only if  $\tau > T .$ 
Also note that  $\tau > T$ implies $s(W_T) > T$ so  $W_T \in Z_T$. 
   Also, \eqref{eqn:stay-in-E} implies $W_T \in E$ when $\tau >T$.  Therefore, almost surely, 
\begin{align*}
 W_\tau \in Z_T   \hbox{ implies }
W_T \in Z_T \cap E \ \& \ \tau > T .
\end{align*}
In particular, we have
\begin{align*}
 Prob[ W_\tau \in Z_T    \ | \ W_0\sim \mu] \le Prob[W_T \in Z_T \cap E \ \& \ \tau > T \ | \ W_0\sim \mu].
\end{align*}

On the other hand, let $\rho_T$ be the distribution of the heat flow in $\R^d$  at time $T$ with the initial condition $\rho_0=\mu$. 
 Notice that $\rho_T \le C_T$ for some constant $C_T$, depending only on $\| \mu\|_{L^1}$ and $T$, decaying exponentially to zero as $T\to \infty$. 
 Consider
\begin{align*}
C_T |Z_T \cap E|  &\ge \int_{Z_T \cap E} d\rho_T =Prob[ W_T \in Z_T \cap E \ | \ W_0\sim \mu ] \\
& \ge Prob[ W_T \in Z_T \cap E \  \& \ \tau  >  T  \ | \ W_0\sim \mu \  ]\\
& \ge Prob[ W_\tau \in Z_T  \ | \ W_0\sim \mu \  ]\\
& = \nu[Z_T] = \nu[Z_T \cap E]\\
& \ge \delta |Z_T \cap E|.
\end{align*}
By examining the limit $T\to \infty$ we see that there exists $\bar T= \bar T (\delta, \|\mu\|_{L^1})>0$ such that  $|Z_{\bar T} \cap E|=0$. Thus we conclude that  $\nu [Z_{\bar T}] =0$, completing the proof. 
\end{proof}

Note that a similar result does not hold for costs of type (II), as easily seen in the following example.
\begin{example}
Let $B_r:= \{x\in\R^d: |x|<r\}$, and solve \eqref{eqn:problem-upper} for a cost of type (II) with  
\begin{align*}
 \mu = 2^d \chi_{B_1} \hbox{ and }  f =1.
 \end{align*}
 Then the optimal measure is given by $ \nu^* =  \chi_{B_2}$
 {\red as seen from the symmetry of the problem and the saturation result Theorem~\ref{thm:saturation}(b).}  
 The corresponding barrier function $s$ is radially symmetric and can be given by a function of the form
\begin{align*}
 s (x) = \begin{cases}
     0  & \text{ for $|x|\le 1 $}, \\
     g(|x|) & \text{ for $1\le  |x| <2$,}\\
     +\infty   & \text{for $|x| \ge 2 $}.
\end{cases}
\end{align*}
where $g: [1, 2) \to \R$ is an increasing function with $g(1)=0$ and $g(2^-) = \infty$. Here the optimal stopping time $\tau$, which is the hitting time to such a barrier, has no upper bound.
\end{example}

\section{Global-time existence of supercooled Stefan problem}\label{sec:global-Stefan}
Translating the saturation result (Theorem~\ref{thm:saturation}) into a PDE formulation (Theorem~\ref{thm:consistency}), we can now derive global-time existence results for both $(St_1)$ (Theorems~\ref{Stefan:I} and \ref{Stefan:I:2}) and $(St_2)$ (Theorem~\ref{Stefan:II}). Our emphasis will be on $(St_1)$ which has not been well understood in the literature beyond one space dimension. We will see that different choices of $f$ (i.e. different  characteristic functions) yield different global-time solutions of $(St_1)$. For instance with $f=1$ (on $\R^d$) we find the unique solution that vanishes in finite time (Theorem~\ref{Stefan:I}), for a certain class of initial data $(\mu, E)$.  Another choice of $f$ yields a solution that does not vanish in finite time (Theorem~\ref{Stefan:I:2}).  Let us point out that the initial domain $E$ needs to be chosen carefully to have such solutions. For instance $E$ can be obtained from a variational problem depending on $\mu$ and $f$ as given in %\eqref{variational}.
{\red \eqref{variational} through  \eqref{obstacle}}.

\medskip

 We first show that the optimal target $\nu$ is positive in the support of {\red the Eulerian variable} $\eta$ as long as  $f$ is. Together with the saturation theorem it follows that $\nu=f$ in the active region if $\tau>0$.

 \begin{lemma}\label{active_region_nu}
Let $\mu, f, \nu=\nu^*$ be as given in Theorem~\ref{thm:existence}, and let $\eta$ be the Eulerian variable associated with $(\mu, \nu)$.
 Then we have
$$
\nu>0 \hbox{ on the set }  \bigcup_{t>0}\{\eta(\cdot,t)>0 \} \cap \{f>0\}.
$$

\end{lemma}

\begin{proof}

We can apply Lemma~\ref{lem:nu_below_W_t} to the set $G:= \{\eta(\cdot,t_1)>0\}$ for each $t_1>0$.
\end{proof}

\subsection{Stable Stefan problem}
 Let us first briefly discuss global well-posedness for $(St_2)$. Let us consider \eqref{eqn:problem-upper}  for cost type $(II)$ with $f=1$ and $\mu>1$ in its support.   
\begin{theorem}[Global-time {\red well-posedness} for $(St_2)$]\label{Stefan:II}
Consider $\mu := (1+\eta_0)\chi_{\Sigma}$ with a positive, bounded function $\eta_0$ and a compact set $\Sigma$. Let $f=1$, \and let $\nu$ solve \eqref{eqn:problem-upper} with $\mu$ and $f$ for cost type (II). Then the corresponding $\eta$-variable for $(\mu, \nu)$  is the unique weak solution of the stable Stefan problem
$$
(\eta +\chi_{\{\eta>0\}})_t - \frac{1}{2}\Delta \eta = 0, \leqno (St_2)
$$
with {\red the} initial data $(\eta_0\chi_{\Sigma}, \Sigma)$.

\end{theorem}

\begin{proof}
Theorem~\ref{thm:saturation} (b) yields the decomposition  
{\red $\mu = \tilde{\mu} + \chi_{\Sigma}$ and  $\nu = \tilde{\nu} + \chi_{\Sigma}$. }
Moreover the "actively accumulated target" $\tilde{\nu}$ solves \eqref{eqn:problem-upper} with $\tilde{f} = \chi_{\Sigma^C}$ and $\tilde{\mu} = \eta_0\chi_{\Sigma}$.
Observe that $\eta$-variable for $(\mu, \nu)$ is the same as that of $(\tilde{\mu}, \tilde{\nu})$, because $\chi_{\Sigma}$ presents precisely the distribution at $\tau=0$.  In particular the optimal stopping time $\tilde{\tau}$ between $\tilde{\nu}$ and $\tilde{\mu}$ is positive, allowing us to apply Theorem~\ref{thm:consistency} for $\tilde{\mu}$ and $\tilde{\nu}$. 

\medskip

Before doing so, let us evaluate $\tilde{\nu}$ in the set $\{0<s(x)<\infty\}$ {\red where $s$ is the barrier function for  $\tilde{\tau}$.} From Lemma~\ref{active_region_nu} and {\red the saturation (Theorem~\ref{thm:saturation})} we conclude that
$$
\tilde{\nu}=1 \hbox{ in } \cup_{t>0}\{\eta(\cdot,t)>0\} \setminus \Sigma.
$$
 Recall from 
 {\red Theorem~\ref{thm:tau-equal-tau-U}}
 we have $\{\eta>0\} = \{s(x)<t\}$. Thus we have $$E = \lim_{t\to 0^+} \{\eta(\cdot,t)>0\} = \{s(x)=0\},$$ which equals $\Sigma$. 
It now follows that $\tilde{\nu}=1$ in the region $\{0<s(x)<\infty\}$. %where $s$ is the barrier function for $\tilde{\tau}$.
Theorem~\ref{thm:consistency} (b) and Remark~\ref{rmk:St-St-nu} finally yields that $\eta$ solves $(St_2)$ {\red uniquely with the} initial data $(\eta_0, \Sigma)$.
\end{proof} 

Note that for the above theorem  the initial set for the active region coincides with the support $\Sigma$ of the initial distribution $\mu$. This is due to the nature for the type (II) costs whose barrier sets in the space-time for the optimal stopping are time-backward monotone, in contrast to the type (I) costs that gives time-forward {\red monotone} barrier sets. 

For the supercooled $(St_1)$ case, not every initial set $\Sigma$ has a global-time weak solution, as we expect from the instability of the problem. This will be apparent in the theorems below for the supercooled case.

\subsection{Supercooled Stefan Problem}

 Next let us discuss existence results for $(St_1)$.  First we apply a parallel reasoning to Theorem~\ref{Stefan:II} but now applied to the type (I) costs, to derive a finite-time vanishing solution.

\begin{theorem}[A global-time existence for $(St_1)$]\label{Stefan:I}
 Let $\mu$ and $f$ be as given in Theorem~\ref{Stefan:II}, and let $\nu$ solve \eqref{eqn:problem-upper} with $\mu$ and $f$ for cost type (I). Then $\nu=\chi_E$, and the Eulerian variable $\eta$ for $(\mu, \nu)$ solves 
the supercooled Stefan problem 
$$
(\eta - \chi_{\{\eta>0\}})_t - \frac{1}{2}\Delta \eta = 0 \leqno(St_1) 
$$
with initial data $(\mu, E)$.  Moreover $E$ can be characterized as the positive set of $w$, where $w$ solves the obstacle problem
\begin{equation}\label{obstacle2}
\chi_{\{w>0\}} -\Delta w = \mu, \quad w\geq 0. 
\end{equation}
In particular $E$ contains $\Sigma$. Lastly, $\eta$ vanishes to zero in finite time. 
\end{theorem}

\begin{proof}
From Theorem~\ref{thm:saturation} (a) and the choice of $\mu$ it follows that $\tau >0$ and  $\nu=\chi_E$.  Due to Theorem~\ref{thm:consistency} (a) it follows that $\eta$ solves $(St_1)$ with initial data $(\mu, \tilde{E})$, where {\red the initial active region $\tilde{E}$ contains $E$. Notice that  $\tilde E = \bigcup_{t>0}\{\eta(\cdot,t)>0 \} $ because the set  $\{\eta(\cdot,t)>0 \}$ decreases monotonically in $t$.} From Lemma~\ref{active_region_nu} and the fact that $f=1$, it follows that $\nu>0$ {\red in $\tilde E$.} % the active region of $\eta$. 
{\red In particular, $\tilde E \subseteq E$, therefore $\tilde E =E$.   This also} means that $E$ contains the support of $\mu$, since there is no instantly frozen part of $\mu$ {\red its support is contained in the initial active region $\tilde E$ which is nothing but $E$ as we saw.}\\ %It also means that $\tilde{E}$ contains $E$, and thus $\tilde{E}$ must equal $E$. \\
Lastly, note that the characterization of $E$ via \eqref{obstacle2} is due to Corollary~\ref{cor:stoppingtime_zero} and the fact that $E$ contains the support of $\mu$.    Lastly, again since $f=1$, Theorem~\ref{thm:time-upper} yields that $\eta$ turns to zero in finite time.
\end{proof}

In the next subsection {\red (see Corollary~\ref{vanishing} )} we will show that $\eta$ in Theorem~\ref{Stefan:I} is indeed the unique solution  for the initial data $\mu$
that vanishes in finite time.

\medskip 

{\red On the other hand, we note that one can} generate solutions of $(St_1)$ from a wider class of initial data by making use of Theorem~\ref{thm:saturation} and Theorem~\ref{thm:consistency}. For instance, one can choose $f = \chi_K$ where $K$ and the support of $\mu$ are disjoint to avoid instant stopping of particles. In this case we obtain the following:

\begin{theorem}[A global-time existence for $(St_1)$, second version]\label{Stefan:I:2} 
Let $\mu\geq 0$ be a compactly supported function in  $L^\infty(\R^d)$, and let $f:=\chi_{K^c}$, where $K$ is a compact set that includes the support of $\mu$. Let $\nu$ solve \eqref{eqn:problem-upper} with $\mu$ and $f$ for cost of type (I). Then the Eulerian variable $\eta$ for $(\mu, \nu)$ solves $(St_1)$ with the  initial data $(\mu, E)$, and $\nu = \chi_{E\cap K^c}$.  Here $E$ is the positive set of $w$, where $w$ solves the obstacle problem
$$
\Delta w = (f-\mu)\chi_{\{w>0\}} = \chi_{\{w>0\}\cap K^c} - \mu, \quad  w\geq 0.
$$

\end{theorem}

\begin{proof}
In this setting, to travel from $\mu$ to $\nu$ the optimal stopping time  $\tau$ must be positive. Theorem~\ref{thm:saturation} now yields that $\nu = \chi_{\Sigma}$, where $\Sigma \subset \{0<s(x)<\infty\} \subset K^C$. {\red (In fact,  $\Sigma =\{0<s(x)<\infty\}$ as one can see from the similar reasoning in the second half of the proof of  Theorem~\ref{Stefan:II}, by using Lemma~\ref{active_region_nu}.)} Hence the consistency theorem {\red (Theorem~\ref{thm:consistency}(a))} yields that $\eta$ solves $(St_1)$ with the initial data $(\mu, E)$, where {\red the initial active region} $E$ contains $\Sigma$. On the other hand $\Sigma\cup K$ contains $E$ due to Lemma~\ref{active_region_nu}. Now we can conclude by characterizing $E$ in terms of $w$ by Corollary~\ref{cor:stoppingtime_zero}. 
\end{proof}

Note that {\red  in the case of Theorem~\ref{Stefan:I:2}  the Eulerian variable} $\eta$ does not vanish in finite time. It is because the Brownian motion will travel forever in  $E\cap K$, where no stopping happens there due to the fact $f |_K =0$.
\begin{remark}
Uniqueness is in general not true for $(St_1)$ with initial data $\mu$, even with the same initial domain $E$ (see the example in Section~\ref{sec:subharmonic-generating}). On the other hand, the optimal target problem with $f$ yields a unique $\eta$-variable corresponding to the optimal stopping time. Thus our approach yields a variationally oriented criteria for a unique solution of $(St_1)$ with {\red a given} initial distribution $\mu$.
\end{remark}

\begin{remark}[Relation between $(St_1)$ and $(St_2)$]\label{rmk:universality-PDE}
   Due to universality (Theorem~\ref{thm:universality}) we have the same optimal target measure $\nu$ for type (I) and (II), for given $\mu$ and $f$. Therefore the solution of  stable Stefan problem $(St_2)$ in Theorem~\ref{Stefan:II} determines the initial set $E$ that yields global-time solution for $(St_1)$. 
\end{remark}

$\circ$ {\bf A discussion on the initial expansion}

\medskip

An interesting feature of our solutions of $(St_1)$ is {\red the initial trace $E$ of} the active region, instantly expanded from the support of the initial data $\mu$.  It is well-understood that the local-in-time solutions to $(St_1)$ develops jump discontinuity when the solution is ``overloaded", namely when the average density goes beyond $1$ (see for instance \cite{CK12}).  When $\mu>1$ in its support, as given in Theorem~\ref{Stefan:I}, $\mu$ is entirely overloaded and thus it is natural that one must expand its support to activate the supercooling process.  On the other hand when $\mu$ has small density as allowed in Theorem~\ref{Stefan:I:2},  such expansion follows from our choice of $f$,  which describes the particular scenario where the freezing stimulus is only available outside of the initial supercooled region. The general criteria for $\mu$ and $E$ to generate a global solution of $(St_1)$, in terms of the size of $\mu$ and the geometry of $E$, remains open. 

\medskip

Our characterization of the initial domain  $E$ via \eqref{obstacle} is also reminiscent of the classical paper by Dibenedetto and Friedman \cite{DbF84}, which considers a time integrated version of $(St_1)$. See the discussion in {\red \cite[Section 4.3]{DbF84}.}

\medskip

\subsection{Existence of finite-time vanishing solution characterizes the initial set}

In this section, we show that the choice of the initial set $E$ in Theorem~\ref{Stefan:I} is {\em necessary} for a finite-time vanishing solution to exist. We begin by connecting finite-time vanishing solutions to the notion of subharmonically generated sets introduced in  Section~\ref{sec:subharmonic-generating}.

\begin{theorem}\label{thm:sh-generate}
For a given compactly supported $0\le \mu \in L^{\infty}(\R^d)$, let $\eta$ be a weak solution of $(St_1)$ with the initial data $(\mu, E)$. Then $\eta$ vanishes in finite time if and only if $(E, E)$ is subharmonically generated by $\mu$ and $\eta$ is the corresponding Eulerian variable to $(\mu, \chi_E)$.  
\end{theorem}
\begin{proof}
The `if' direction is a direct consequence of Theorem~\ref{thm:sh-gen-equiv-St1} and Theorem~\ref{thm:time-upper}. The `only if' direction follows from Theorem~\ref{thm:sh-gen-equiv-St1} and  the fact that 
$s(x):=\sup \{t: \eta(x,t)>0\} \leq T$ for some finite $T$.
\end{proof}

\begin{theorem}[Necessary condition for finite time vanishing solution]\label{thm:unique-E}
{\red For a given compactly supported $0\le \mu \in L^{\infty}(\R^d)$, 
suppose} there is a solution to $(St_1)$ with the initial data $(\mu, E)$ that vanishes in finite time. Then, the set $E$ is determined by the optimal free target problem \eqref{eqn:problem-upper} of type (I) cost,  for $\mu$ and $f\equiv 1$, where  the optimal target is $\nu^*=\chi_E$. 

\end{theorem}

Before proving this theorem, we point out that  to have such an initial active region $E$ {\red which is also the support of $\nu^*$ at the same time,} we must have {\red the corresponding optimal time} $\tau >0$ almost surely, {\red so $E$ contains the support of $\mu$.} 
In particular Corollary~\ref{cor:stoppingtime_zero} yields that such a set $E$  is determined by the obstacle problem \eqref{obstacle2}. Hence we arrive at the following conclusion {\red that  uniquely characterizes the solution in Theorem~\ref{Stefan:I}}:

\begin{corollary}[Unique Characterization]\label{vanishing} 
Let $\mu$ be as given in Theorem~\ref{Stefan:I}.  
Then there is a unique solution $\eta$ of $(St_1)$ with the initial data $\mu$ that vanishes in finite time. 
In this case, the initial domain $E$ is given by \eqref{obstacle2}, and $\eta$ corresponds to the Eulerian variable given by $(\mu, \chi_E)$.
\end{corollary}

\begin{proof}[Proof of Theorem~\ref{thm:unique-E}]
 For this result,  we apply Theorem~\ref{thm:sh-generate} and Theorem~\ref{thm:monotonicity}.
 
 \medskip

 To have a vanishing in finite time solution to $(St_1)$,
  from Theorem~\ref{thm:sh-generate} it is necessary that $(E, E)$ has to be subharmonically generated by $\mu$. By definition, there is a corresponding type (I) optimal stopping time, say, $\tau>0$, solving the optimal Skorokhod problem $\mathcal{P}(\mu, \chi_E)$.  Moreover, since $\tau>0$ and we have $E= \{s(x)>0\}$ where $s(x)$ is the barrier function. 
  {\red From Theorem~\ref{thm:tau-equal-tau-U} it}
   follows that $ \{\eta>0\}= \{t<s(x)\}$,  and thus $E =\limsup_{t \to 0^+} \{  \eta (\cdot, t) >0\}$  for the corresponding Eulerian flow $(\eta, \rho)$. 

\medskip

On the other hand, we have the optimal stopping time $\tau^*$ for the free target problem \eqref{eqn:problem-upper}  for type (I) cost, from the measure $\mu$, with $f \equiv 1$, 
 with the optimal target $\nu^*$  with its density $\nu^* \le 1$ everywhere.

\medskip

We claim that {\red $\nu^*\le \chi_E$. This will imply  $\nu^*= \chi_E$ because they have the same mass.} 

\medskip

To prove the claim, first 
note that  the monotonicity result, Theorem~\ref{thm:monotonicity} part (1),  does not require optimality of one of the target, therefore, we have  $\tau^* \le \tau.$ Therefore, the Eulerian flow $\eta^*$ for $\tau^*$ is supported inside the support of the Eulerian flow $\eta$ of $\tau$.  In particular the initial domain for the Eulerian flow $\eta^*$ is a subset of the initial domain for $\eta$, that is, $E$.  Moreover, even when there is instant stopping $\tau^*=0$ for a set $S$, since $\tau >0$ almost surely, we have  $S \subset E$. From these considerations and the fact that $\nu^* \le 1$, {\red we have  $\nu^* \le \chi_E$ as desired.} 
This completes the proof. 
\end{proof}

From a parallel reasoning we can uniquely characterize $\eta$ from Theorem~\ref{Stefan:I:2} as well:

\begin{corollary}[Unique Characterization]\label{unique_char_2}
Let $\mu, K,\eta$ and $E$ be as given in Theorem~\ref{Stefan:I:2}. Then $\eta$ is the unique weak solution of $(St_1)$ such that $\cap_{\{t>0\}}\{\eta(\cdot,t)>0\} = K$.
In other words, $E$ is the unique set for which  $(E\cap K^c, E)$ is subharmonically generated.
\end{corollary}

\section{Monotonicity of the optimal barrier functions}\label{sec:mono_barrier_function}

We show in this section that the barrier function $s(x)$ for the optimal stopping time $\tau$  {\red (the solution to the optimal free target problem \eqref{eqn:problem-upper})}  enjoys a monotonicity property in the spirit of Theorem~\ref{thm:monotonicity}. Namely, the barrier functions $s_i$ are ordered if the initial distributions $\mu_i$ are ordered (Theorem~\ref{thm:s_mono}), and such {\red an} order is strict (Theorem~\ref{thm:stric-s-mono}). {\red This ordering property yields a  {\it strict} comparison principle for the Stefan problems, where a local strict ordering of $\mu_i$ implies global strict ordering on $s_i$'s;  this local-to-global strictness is new, to our best knowledge, even for $(St_2)$, where only global-global strict comparison was proven before \cite{kim03}.  The proof is entirely probabilistic and is based on the properties of the optimal free target problem. It thus provides new insights into the Stefan problems, in particular for the supercooled Stefan problem $(St_1)$. Especially, we} hope it may shed a light on understanding regularity of the barrier function $s(x)$.  Understanding regularity of $s(x)$, so regularity of the free boundaries of the Stefan problem, is a wide open problem in $(St_1)$ for any class of solutions beyond one space dimension.

\medskip

Our results in Theorem~\ref{thm:tau-equal-tau-U} show 
{\red that for the optimal stopping time, the space-time barrier set $R$ is nothing but $R^U$, thus is closed, and the barrier function $s=s^U$ is lower-semicontinuous  for type (I) and upper semicontinuous for type (II). This then can be translated, via the results in Section~\ref{sec:global-Stefan}, to corresponding properties for the free boundaries of the Stefan problems $(St_1)$ and $(St_2)$.}
To our best knowledge even this very mild regularity result for {\red the free boundary} is new for $(St_1)$; {\red even though it was for an appropriately chosen initial domain $E$.} However, for such semicontinuity we only use the optimality of the stopping time  for the optimal Skorokhod problem $\mathcal{P}(\mu, \nu)$ {\red \eqref{eqn:problem-OSP}}. Our monotonicity of {\red the barrier function} below, which is a result of optimality of the target measure, may lead to  a nicer regularity result. 
\medskip

{\red In view of Theorem~\ref{thm:tau-equal-tau-U}(b) and for clarity, we let $s^U$ denote the barrier function from now on.}

\begin{theorem}[Monotonicity for optimal barrier functions]\label{thm:s_mono}
 Assume {\red $0\le \mu_1 \le \mu_2 \in L^\infty(\R^d)$ be compactly supported} and let
  $0\le f\in L^{\infty}(\R^d)$.
  Let $\nu_i$, $i=1,2,$ be the optimal solutions of  the Problem \eqref{eqn:problem-upper} with $\mu_i$, $i=1,2$, respectively. 
 Let $s_i^U$ be the barrier functions  associated with $(\mu_i, \nu_i)$ as  in Definition~\ref{barrier_set}. {\red In the type (II) case assume that $\tau_i >0$, $i=1,2$.} Then we have 
 \begin{align*}
&  s_1 ^U\le s_2^U \quad  \hbox{ \red   for type (I) cost, and }\\
 &  s_1^U \ge s_2 ^U \quad   \hbox{for type (II) cost.}
\end{align*}
\end{theorem}
\begin{proof}
 Due to Theorem~\ref{thm:tau-equal-tau-U}(b), the corresponding optimal stopping time $\tau_i$'s, the barriers, and the barrier functions for $(\mu_i, \nu_i)$ can be represented as  $\tau_i^U$, $R_i^U$, and $s_i^U$, respectively.
 We utilize the potential flow formulation in Section~\ref{sec:potential} as well as the monotonicity  result for the optimal target problem (Theorem~\ref{thm:monotonicity}).
  {\red In what is below for ease of notation we use the notation $\tau_i$'s for the optimal stopping time $\tau_i^U$, but, we will keep the superscript $U$ for $R_i^U$ and $s_i^U$.}

\medskip

{\red  Let $\mu_t^i$, $i=1,2$,  be the distribution of $W_{\tau_i\wedge t}$, that is, $W_{\tau_i\wedge t} \sim \mu_t^i$.}
Below we treat type (I) and (II) separately. 

{\red   We derive the results by examining  the Ito's formula given in  Lemma~\ref{lem:Ito-closed-set}. 

\medskip
{\bf Type (I).}
Define the functions \begin{align*}
&\hbox{
$U_1(x, t) := U_{\nu_1} (x)-U_{\mu_t^1}(x)$
 and 
$U_2 (x,t) := U_{\nu_2} (x)-U_{\mu_t^2}(x)$.}
\end{align*}
}
Notice that they are all nonnegative, continuous, and monotone decreasing in time,  due to Corollary~\ref{cor:uniform-U} and Corollary~\ref{lem:mono-potential}.
{\red
 {
 Fix $t$ and $x$. Consider an $\epsilon$-ball $B_\epsilon(x)$ around $x$. Then, 
  Lemma~\ref{lem:Ito-closed-set} yields that 
  \begin{align*}%\label{formula:s-i}
 \int_{B_\epsilon (x)} U_1 (y, t)dy = \int_{B_\epsilon (x)} \left(U_{\nu_1} - U_{\mu_t^1}\right)(y) d y \,  &= \,\mathbb{E}\left[ \int_{\tau_1\wedge t}^{\tau_1}   \frac{1}{2} \chi_{B_\epsilon(w)} (W_t) dt\right],\\
 \int_{B_\epsilon (x)} U_2 (y, t) dy =  \int_{B_\epsilon (x)}  \left(U_{\nu_2} - U_{\mu_t^2}\right)(y) d y \, &= \,\mathbb{E}\left[ \int_{\tau_2\wedge t}^{\tau_2}   \frac{1}{2} \chi_{B_\epsilon (x)} (W_t) dt\right].
\end{align*}
For the expected value $\,\mathbb{E}\left[ \int_{\tau_1\wedge t}^{\tau_1}   \frac{1}{2} \chi_{B_\epsilon(w)} (W_t) dt\right],$ exactly those Brownian paths (staring from the distribution $\mu_1$) for which $t < \tau_1$ contributes, while for $\mathbb{E}\left[ \int_{\tau_2\wedge t}^{\tau_2}   \frac{1}{2} \chi_{B_\epsilon (x)} (W_t) dt\right]$ exactly those paths (starting from the distribution $\mu_2$)  for which $t < \tau_2$ contribute. Now, recall that $\mu_1 \le \mu_2$ and $\tau_1 \le \tau_2$ (due to Theorem~\ref{thm:monotonicity}(1), for the restriction of $\tau_2$ on those paths starting from the distribution $\mu_1$). Therefore, the set of paths that contribute to the expected value $\mathbb{E}\left[ \int_{\tau_1\wedge t}^{\tau_1}   \frac{1}{2} \chi_{B_\epsilon(w)} (W_t) dt\right],$ is contained in the he set of paths that contribute to the expected value $\mathbb{E}\left[ \int_{\tau_2\wedge t}^{\tau_2}   \frac{1}{2} \chi_{B_\epsilon (x)} (W_t) dt\right]$. As a result, we have
\begin{align*}
 \mathbb{E}\left[ \int_{\tau_1\wedge t}^{\tau_1}   \frac{1}{2} \chi_{B_\epsilon(w)} (W_t) dt\right] 
 \le \mathbb{E}\left[ \int_{\tau_2\wedge t}^{\tau_2}   \frac{1}{2} \chi_{B_\epsilon (x)} (W_t) dt\right]
\end{align*}
which implies
\begin{align*}
\frac{1}{B_\epsilon(x)}  \int_{B_\epsilon (x)} U_1 (y, t)dy \le \frac{1}{B_\epsilon(x)}   \int_{B_\epsilon (x)} U_2 (y, t) dy.
\end{align*}
However, as $\epsilon\to 0$, due to continuity of $U_1$ and $U_2$, 
we get  $U_1(x, t) \le U_2(x, t)$. 
Note that from the definition
$$s_i^U(x) : = \inf\{ t \ | \ U_i(x, t) =0\}.$$
Therefore, $s_1^U \le s_2^U$ as desired. 

\medskip

{\bf Type (II).} 
 Define the functions 
\begin{align*}
&\hbox{
$U_1(x, t) := U_{\mu_t^1}(x) - U_{\mu_1}(x) $ and 
$U_2 (x,t) := U_{\mu_t^2}(x)- U_{\mu_2} (x)$.}
\end{align*}
Similarly to the type (I) case apply Lemma~\ref{lem:Ito-closed-set}, and get for each $(x,t)$ and $\epsilon>0$, 
 \begin{align*}
 \int_{B_\epsilon (x)} U_1 (y, t)dy = \int_{B_\epsilon (x)} \left(U_{\mu_t^1} - U_{\mu_1}\right)(y) d y \,  &= \,\mathbb{E}\left[ \int_0^{\tau_1\wedge t} \frac{1}{2} \chi_{B_\epsilon(w)} (W_t) dt\right],\\
 \int_{B_\epsilon (x)} U_2 (y, t) dy =  \int_{B_\epsilon (x)}  \left(U_{\mu_t^2} - U_{\mu_2}\right)(y) d y \, &= \,\mathbb{E}\left[ \int_0^{\tau_2\wedge t}  \frac{1}{2} \chi_{B_\epsilon (x)} (W_t) dt\right].
\end{align*}
Since  $\mu_1 \le \mu_2$ and $\tau_1 \le \tau_2$ (due to Theorem~\ref{thm:monotonicity}(1), for $\tau_2$ restricted to $\mu_1$), as in the type (I) case above, we can compare the set of Brownian paths that contribute to the expected values. Then we get 
$\mathbb{E}\left[ \int_0^{\tau_1\wedge t} \frac{1}{2} \chi_{B_\epsilon(w)} (W_t) dt\right]
\le \,\mathbb{E}\left[ \int_0^{\tau_2\wedge t}  \frac{1}{2} \chi_{B_\epsilon (x)} (W_t) dt\right].$ Thus using continuity of $U_1$ and $U_2$, we get 
$U_1(x, t) \le U_2 (x,t)$. Because
$$s_i^U(x) : = \sup\{ t \ | \ U_i(x, t) =0\},$$
this implies $s_1^U \ge s_2^U$ as desired. This completes the proof. 
}
}
\end{proof}

We now prove {\em strict} monotonicity for $s^U$, which is derived by combining the above theorem 
with the monotonicity of the stopping time (Theorem \ref{thm:monotonicity}), {\red Proposition~\ref{lem:hitting_s},}  and  the saturation result (Theorem~\ref{thm:saturation}). {\red In the discussion below we use the assumption given in Theorem~\ref{thm:s_mono}.}

\medskip

As a preparation of the statement of the result, let us recall that the {\red barrier sets $R_i= R_i^U$,} $i=1,2$ for the stopping times in Theorem~\ref{thm:s_mono}
 are closed, therefore their complements are open. 
 Define for $i=1,2$, {\red by using the Eulerian variable $(\eta_i, \rho_i)$ of $\tau_i$, }
 \begin{align}\label{eqn:E-i}
\hbox{in type (I),} & \quad  E_i  := \limsup_{t \to 0^+} \{ x \ | \ \eta_i (x, t) >0 \}, \\\nonumber
\hbox{in type (II),} & \quad  E_i  := \limsup_{t \to \infty} \{ x \ | \ \eta_i (x, t) >0 \}.
 \end{align} 
 Note that due to the {\red time} monotonicity of $R_i$, $E_i= R_i^c = \cup_{t>0}\{\eta_i(\cdot,t)>0\}$,  thus they are open.
 Note also that as $E_i$'s are connected to  the active regions, the Brownian paths that start from $E_i$'s have $\tau_i>0$ almost surely, except those in type (II) that may immediately stop; see Theorem~\ref{thm:saturation}, where the mass that immediately stop in type (II) is characterized by the initial data as $\mu \wedge f$. Thus {\red in type (II) case,} by taking the initial data as $\mu-\mu\wedge f$ we may assume $\tau_i>0$ in $E_i's$.
  The complement $E_2^c$ is the set where there is immediate and complete stopping  for $\tau_2$, so $\tau_2$ is zero, and so is $\tau_1=0$ from monotonicity (Theorem \ref{thm:monotonicity});
{\red  in  $E_2^c$}    the corresponding $s_i^U$'s are $0$ in type (I) case, $\infty$ in type (II) case.
Moreover, in type (I) case, $E_2^c$ is exactly the set where the immediate stopping occurs for $\tau_2$ {\red  so also for $\tau_1$.} This justifies that for comparison between $s_i^U$ we can without loss of generality assume that $\tau_i>0$ in both type (I) and (II), and compare $s_i^U$'s only over $E_2$. 
Also assuming $f>0$ everywhere does not cost much generality  for our purpose of comparing $s_i^U$'s, because in the region $f=0$ no stopping to accumulate mass occurs, so the value $s_i^U$ can only be either $\infty$ or $0$. With these considerations we see that the following theorem essentially covers the general case and the whole domain for the strict monotonicity of $s^U$. 

\medskip

\begin{theorem}[Strict monotonicity for optimal barrier functions]\label{thm:stric-s-mono}
Let $f$, $\mu_i$, $\tau_i$ and $s_i^U$ be as given in Theorem~\ref{thm:s_mono}, in particular, with $\mu_1 \le \mu_2$ and {\red $s_1^U \leq s_2^U$.} Assume that $f>0$ everywhere and that   
$\tau_i >0$, $i=1,2$, almost surely. 
Let $O$ be a path connected component of $E_2$ given in \eqref{eqn:E-i}. 
Suppose $\mu_2 > \mu_1$ on a subset $G \subset O$ with $|G|>0$. Then $s_1^U (x) \not = s_2^U(x)$ for a.e. $x \in O$. 
\end{theorem}

\begin{proof}
From the continuity of Brownian paths, almost surely any Brownian path starting from $O$ should stay inside the corresponding path-connected component of the active region, say $A \subset \R^d \times \R_{\ge 0}$,  that is connected to $O$, that is, 
\begin{align*}
\hbox{$A=\{ (x, t) \ | \ x \in O, \ \ t < s_2^U(x)\}$ in type (I),  $A=\{ (x, t) \ | \ x \in O, \ \ t > s^U_2(x)\}$ in type (II).}
\end{align*}
From this and the condition $\mu_1\le \mu_2$, we may assume that $O=E_2$, namely, we consider only those mass  distributions, Brownian motion, and the barriers, associated to $O$; for example, we assume $\mu_i|_E =\mu_i|_O$, $\nu_i|_E =\nu_i|_O$.  
Notice that $A$ is a connected and an open set and $\eta_2>0$ on $A$. From this we also see that  
$$
\nu_2 >0 \hbox{ on }O.
$$
(To see this,  we can for instance apply   Lemma~\ref{lem:nu_below_W_t}, which uses optimality of $\nu_2$ with respect to Problem~\ref{eqn:problem-upper}.)
Therefore, it suffices to show that $s_1^U(x) \not = s_2^U(x)$ for $\nu_2$-a.e. $x$.

\medskip

{\red  Also Theorem~\ref{thm:saturation} yields that $\nu_i =f $ on its support, since we assume that $\tau_i>0$.
Therefore to prove the theorem it suffices to show that
\begin{align}\label{eq:last}
\nu_2 [S]=0,
\end{align}
where
\begin{align*}
S:=    \{ x \ | \ s_1^U (x) = s_2^U (x) <\infty \} \cap \{ x \ | \ \nu_1 (x)  = \nu_2 (x) =f(x) \} \cap O.
\end{align*}
For the rest of the proof, we verify this claim. 
}

\medskip

Recall  that $\tau_1 \le \tau_2$ and $\nu_1 \le \nu_2$ from Theorem~\ref{thm:monotonicity}.
Because of the Markov property of Brownian motion, the inequality {\red $\tau_1 \le \tau_2$} can be applied to all Brownian paths from the initial distribution $\mu_2$ not only from $\mu_1$ (see Remark~\ref{rmk:tau2-tau1-general}).

\medskip

{\red 
From optimality of $\tau_i$'s we can apply
 Proposition~\ref{lem:hitting_s}, to get $s_i^U(W_{\tau_i}) =\tau_i$, $i=1,2$ a.s. Now if  $s_1^U(W_{\tau_1}) = s_2^U(W_{\tau_1})$ then almost surely $\tau_1 =  s_2^U(W_{\tau_1})$, so from the hitting time charaterisation of $\tau_2$, with the barrier $s_2^U$, we have $\tau_1 \ge \tau_2$. Since we already have $\tau_1 \le \tau_2$ a.e., we therefore have,}
\begin{align}\label{eqn:s-to-tau}
\hbox{almost surely, if $s_1^U(W_{\tau_1}) = s_2^U(W_{\tau_1})$ then $ \tau_2=\tau_1$.}
\end{align}

\medskip

 We have  from $W_0\sim \mu_i|_O$, $W_{\tau_i} \sim \nu_i|_O$ that 
 \begin{align*}
 \nu_i [S] = \int_O Prob[W_{\tau_i} \in S | W_0 = x ] d\mu_i(x) \quad \hbox{ for $i=1,2$.}
\end{align*}

Let $G$ be as given in the {\red statement of this} theorem, and observe that $\mu_2 [G]>0$. 
Then, for $i=1,2$, 
\begin{align}\label{eqn:S-G-integral}
\nu_i [S] 
&  = \int_{G} Prob[W_{\tau_i} \in S | W_0 = x ] d\mu_i(x)  + \int_{O\setminus G}  Prob[W_{\tau_i} \in S | W_0 = x ] d\mu_i(x).
\end{align}

\medskip

{\red On the other hand,}  
\begin{align*}
&Prob[W_{\tau_2} \in S \ | \ W_0 = x ] \\
&= Prob[W_{\tau_2} \in S \  \& \ \tau_2 = \tau_1 \ | \ W_0 = x  ]  + Prob[W_{\tau_2} \in S \ \& \ \tau_2 > \tau_1  | \  W_0 = x  \ ]  \\
& \ge Prob[W_{\tau_2} \in S  \ \& \ \tau_2=\tau_1 \ | \ W_0 = x  ]  \\
& = Prob[W_{\tau_1} \in S  \ \& \ \tau_2=\tau_1  \ | \ W_0 = x  ]  .
\end{align*}

\medskip

From \eqref{eqn:s-to-tau} it holds that  for $\mu_2$-a.e. $x$, 
  \begin{align*}
 Prob[W_{\tau_1} \in S \ \& \ \tau_2=\tau_1 \ | \ W_0 = x  ]   = Prob[W_{\tau_1} \in S \ | \ W_0 = x]  .
\end{align*}
Therefore, from the previous inequality we have for $\mu_2$-a.e. $x$, 
\begin{align*}
 Prob[W_{\tau_2} \in S \ | \ W_0 = x ] \ge  Prob[W_{\tau_1} \in S \ | \ W_0 = x]  .
\end{align*}
 Apply this to the above integrals \eqref{eqn:S-G-integral}
 and get
 {\red 
\begin{align*}
\nu_2 [S] - \nu_1 [S] & \ge \int_G Prob[W_{\tau_2} \in S | W_0 = x ] d\mu_2(x) - \int_G Prob[W_{\tau_1} \in S | W_0 = x ] d\mu_1(x)\\
& \ge \int_G Prob[W_{\tau_1} \in S | W_0 = x ] d\mu_2(x) - \int_G Prob[W_{\tau_1} \in S | W_0 = x ] d\mu_1(x)\\
& = \int_G Prob[W_{\tau_1} \in S | W_0 = x ] d(\mu_2 - \mu_1)(x) \\
& \ge 0.
\end{align*}
}
On the other  hand, due to the definition of $S$, $\nu_1(S) = \nu_2(S)$. This and the condition $\mu_2 > \mu_1$ on $G$ yields
\begin{equation}\label{last}
Prob[W_{\tau_2} \in S \ | \ W_0 = x ]=  Prob[W_{\tau_1} \in S \ | \ W_0 =x] =0\hbox{ for } \mu_2\hbox{-a.e. } x \hbox{ in }  G.
\end{equation}

\medskip

Let $\tilde \tau_2$ be the restriction of $\tau_2$ to the to the initial distribution $\mu_2|_G$ and consider its Eulerian flow $\tilde \eta_2. $  Since the active set $A$ is a {\red path-}connected open set, we have $\tilde \eta_2>0$ everywhere in $A$. Recall from above that $\nu_2>0$ on $O$ and that $\eta_2>0$ on $A$ for the Eulerian flow $\eta_2$ of $\tau_2$ with $\mu_2$. In other words, $\tilde \eta_2>0$ wherever $\eta_2>0$. Therefore, from the Markov property of the Brownian motion, the resulting target distribution $\tilde \nu_2$ of $\tilde \tau_2$ has $\tilde \nu_2 >0$ on $O$, because it has to be positive wherever $\nu_2 >0$. 
Now, \eqref{last}  implies that $\tilde \nu_2[S]=0$, so we can conclude  with \eqref{eq:last}. This completes the proof. 
\end{proof}

%%%%%%%%%%%%%%%%%%%%%%%%%%%%%%%%%%%%%%%%%%%%%%
%% Single Appendix:                         %%
%%%%%%%%%%%%%%%%%%%%%%%%%%%%%%%%%%%%%%%%%%%%%%
%\begin{appendix}
%\section*{???}%% if no title is needed, leave empty \section*{}.
%\end{appendix}
%%%%%%%%%%%%%%%%%%%%%%%%%%%%%%%%%%%%%%%%%%%%%%
%% Multiple Appendixes:                     %%
%%%%%%%%%%%%%%%%%%%%%%%%%%%%%%%%%%%%%%%%%%%%%%
%\begin{appendix}
%\section{???}
%
%\section{???}
%
%\end{appendix}

%%%%%%%%%%%%%%%%%%%%%%%%%%%%%%%%%%%%%%%%%%%%%%
%% Support information, if any,             %%
%% should be provided in the                %%
%% Acknowledgements section.                %%
%%%%%%%%%%%%%%%%%%%%%%%%%%%%%%%%%%%%%%%%%%%%%%
\begin{acks}[Acknowledgments]
We thank Mathav Murugan for various helps regarding probabilistic aspects.  We also thank Paul Gassiat for helpful comments on Section~\ref{sec:potential}. {\red We also thank the anoymous referees for helpful comments, especially 
pointing out an error in the original statement of Proposition~\ref{lem:hitting_s} in the previous version, {\redf as well as letting us know the connection between Problem~\ref{eqn:problem-upper} and the notion of the shadow of a measure introduced in \cite{Rost71, beiglboeck-juillet2016} (see also \cite{bruckerhoff2021shadows}),  among others.}}
%
% The authors would like to thank ...
\end{acks}
%%%%%%%%%%%%%%%%%%%%%%%%%%%%%%%%%%%%%%%%%%%%%%
%% Funding information, if any,             %%
%% should be provided in the                %%
%% funding section.                         %%
%%%%%%%%%%%%%%%%%%%%%%%%%%%%%%%%%%%%%%%%%%%%%%
\begin{funding}
IK is partially supported by {\redf the National Science Foundation (NSF)  with the grant DMS 1900804, and the} Simons foundation. YHK is partially supported by  the 
Natural Sciences and Engineering Research Council of Canada (NSERC), {\redf with Discovery Grant RGPIN-2019-03926,}  as well as  Exploration Grant {\redf (NFRFE-2019-00944)}  from the New Frontiers in Research Fund (NFRF). 

This work has been initiated while IK was visiting UBC under the Pacific Institute for the Mathematical Sciences (PIMS) distinguished visitor program: January--May, 2020. {\redf YHK is also a member of the Kantorovich Initiative (KI) that is supported by PIMS Research Network (PRN) program. We thank PIMS for their generous support.

The final revision of this paper is completed while YHK is visiting Korea Advanced Institute of Science and Technology (KAIST). We thank for their hospitality and support.
}
% The first author was supported by ...
%
% The second author was supported in part by ...
\end{funding}

%%%%%%%%%%%%%%%%%%%%%%%%%%%%%%%%%%%%%%%%%%%%%%
%% Supplementary Material, including data   %%
%% sets and code, should be provided in     %%
%% {supplement} environment with title      %%
%% and short description. It cannot be      %%
%% available exclusively as external link.  %%
%% All Supplementary Material must be       %%
%% available to the reader on Project       %%
%% Euclid with the published article.       %%
%%%%%%%%%%%%%%%%%%%%%%%%%%%%%%%%%%%%%%%%%%%%%%
%\begin{supplement}
%\stitle{???}
%\sdescription{???.}
%\end{supplement}

%%  MR numbers will be added by VTeX.                      %%
%%                                                         %%
%%  Use \cite{...} to cite references in text.             %%
%%                                                         %%
%%%%%%%%%%%%%%%%%%%%%%%%%%%%%%%%%%%%%%%%%%%%%%%%%%%%%%%%%%%%%

%% if your bibliography is in bibtex format, uncomment commands:
\bibliographystyle{imsart-number} % Style BST file (imsart-number.bst or imsart-nameyear.bst)
   % Bibliography file (usually '*.bib')
\bibliography{ref_stoppingFB}
%% or include bibliography directly:
% \begin{thebibliography}{}
% \bibitem{b1}
% \end{thebibliography}

%
%
%\appendix

\end{document}